\documentclass[11pt]{amsart}

\usepackage[utf8]{inputenc}
\usepackage[T1]{fontenc}
\usepackage[french, english]{babel}

\usepackage{amsthm}
\usepackage{amsmath}
\usepackage{amssymb}
\usepackage{mathrsfs}
\usepackage{bbm}
\usepackage{graphicx}
\usepackage{listings}
\usepackage{mathrsfs}
\usepackage{enumerate}
\usepackage{pict2e}
\usepackage{stmaryrd}
\usepackage{mathtools}

\usepackage[all,cmtip]{xy}

\usepackage{multirow}

\usepackage{color}
\definecolor{marin}{rgb}   {0.,   0.1,   0.5} 
\definecolor{rouge}{rgb}   {0.8,   0.,   0.} 
\definecolor{sepia}{rgb}   {0.4,   0.25,   0.} 
\definecolor{mag}{rgb}   {0.3,   0,   0.3} 

\usepackage[colorlinks,citecolor=sepia,linkcolor=marin,urlcolor=mag ,bookmarksopen,bookmarksnumbered]{hyperref}

\newtheorem{theorem}{Theorem}[section]
\newtheorem{corollary}[theorem]{Corollary}
\newtheorem{lemma}[theorem]{Lemma}
\newtheorem{proposition}[theorem]{Proposition}
\newtheorem{definition}[theorem]{Definition}
\newtheorem{remark}[theorem]{Remark}

\newcommand{\N}{\mathbb{N}}

\newcommand{\R}{\mathbb{R}}

\newcommand{\T}{\mathbb{T}}
\newcommand{\Z}{\mathbb{Z}}

\newcommand{\eps}{\varepsilon}

\begin{document}

\title[Dynamics of quintic nonlinear Schr\"odinger equations in $H^{2/5^+}(\mathbb{T})$
]{Dynamics of quintic nonlinear Schr\"odinger equations in $H^{2/5^+}(\mathbb{T})$}

\author{Joackim Bernier}

\address{\small{Nantes Universit\'e, CNRS, Laboratoire de Math\'ematiques Jean Leray, LMJL,
F-44000 Nantes, France
}}

\email{joackim.bernier@univ-nantes.fr}

\author{Beno\^it Gr\'ebert}

\address{\small{Nantes Universit\'e, CNRS, Laboratoire de Math\'ematiques Jean Leray, LMJL,
F-44000 Nantes, France
}}

\email{benoit.grebert@univ-nantes.fr}

\author{Tristan Robert}

\address{\small{Université de Lorraine, CNRS, IECL, F-54000 Nancy, France}}

\email{tristan.robert@univ-lorraine.fr}



\begin{abstract} 
In this paper, we succeed in integrating Strichartz estimates (encoding the dispersive effects of the equations) in Birkhoff normal form techniques. As a consequence, we  deduce a result on the long time behavior of quintic NLS solutions on the circle for small but very irregular initial data (in $H^s(\mathbb{T})$ for $s>2/5$). Note that since $2/5<1$ we cannot claim conservation of energy and, more importantly, since $2/5<1/2$, we must dispense with the algebra property of $H^s$. This is the first dynamical result where we use the dispersive properties of NLS in a context of Birkhoff normal form.

\end{abstract} 
\maketitle

\setcounter{tocdepth}{1} 
\tableofcontents

\section{Introduction}
Schematically, the Birkhoff normal form method consists of a first algebraic step where we transform the Hamiltonian of the PDE on a space of functions depending only on the space variable, and then of a second dynamic step where we deduce a long time behavior of the solutions of this PDE. In most of the results using this approach, the first step essentially involves multilinear estimates based on algebraic properties of the function space used, here the  Sobolev space on the d-dimensional torus $H^s(\T^d)$, and in this case a minimal regularity is required, here $s>d/2$. In this paper, we develop a new approach: we use dispersion properties already in the first step. The time oscillatory nature of the solutions, encoded in the Strichartz estimates, allows us to improve the multilinear estimates (essentially by lowering the regularity) and to propagate them.\\
As a result, combining normal form techniques and dispersive techniques, we are able  to specify the dynamics  of the quintic nonlinear Schr\"odinger equation (NLS) in $H^s(\T)$ with $s<1/2$, i.e. we can get rid of the algebra property so useful for non-linear equations. The proof is based on a normal form result without regularity, i.e. in $L^2$,  inspired by \cite{Bou04b} appendix 7 (see also \cite{CKO} and section \ref{sec:sketch}).

\subsection{Main results and comments}
To clarify our point we will focus on an example, the quintic NLS on the circle, but with two different linear perturbations:
\begin{itemize}
\item quintic NLS on the circle with a multiplicative potential
\begin{equation}\label{eq:NLS}\tag{NLS} 
i \partial_t u=-\partial_{x}^2u+Wu+\sigma |u|^4u, \quad x\in\T:=\R/2\pi\Z,\ t\in\R,
\end{equation}

\item quintic NLS on the circle with a convolution potential
\begin{equation}\label{eq:NLS*}\tag{NLS*} 
i \partial_t u=-\partial_{x}^2u+V\ast u+\sigma |u|^4u, \quad x\in\T:=\R/2\pi\Z,\ t\in\R,
\end{equation}
\end{itemize}
where, in both cases, $\sigma=\pm1$ allows considering
both the focusing and the defocusing cases and the potentials $W,V\in H^\theta(\T)$ ($\theta\geq0$ will be specified later) 
will be chosen to avoid resonances issues.

Before considering the long time behavior of the solutions,  we recall that according to Bourgain \cite{Bou93}, \eqref{eq:NLS}, with $W=0$, is locally well posed in $H^s(\mathbb{T})$ for $s>0$ and according to Li-Wu-Xu (see \cite{LWX}) it is globally well posed in\footnote{After completion of this work, small data global well-posedness of \eqref{eq:NLS} for $W=0$ was extended to $s>\frac13$ by Schippa \cite{Schippa}.} $H^s(\mathbb{T})$ for $s>2/5$. In section \ref{sec:gwp} we extend  these results  for both  \eqref{eq:NLS} and \eqref{eq:NLS*} to obtain the following.

\begin{proposition}\label{thm:gwp} Let $s>\frac25$. We assume that $V\in \mathcal{F}L^\infty(\mathbb{T};\mathbb{C})$ has real-valued Fourier coefficients\footnote{defined by $V_j:=(2\pi)^{-\frac12}\int_{\mathbb{T}} V(x) e^{-ijx} \ \mathrm{d}x$, see also \eqref{Fourier} below. Then $V\in\mathcal{F}L^\infty$ means $\|V\|_{\ell^\infty}:=\sup_{k\in\Z}|V_k|<\infty$.}, and that $W\in H^4(\mathbb{T};\mathbb{R})$ is even. Then there exists constants $\epsilon_0\in(0;1]$, $\beta_s\ge 1$, and $C_s>0$, Banach spaces $X^V,X^W\subset C(\R;H^s(\mathbb{T};\mathbb{C}))$, such that the following holds. For any initial datum $u_0\in H^s(\mathbb{T})$ with\footnote{Here, we only consider small initial data since this is the regime we are interested in to perform a Birkhoff normal form transformation. But the same global well-posedness result holds for large data in the defocusing case $\sigma>0$.} $\|u_0\|_{H^s}\le \epsilon_0$ there exists a unique global mild solution $u\in X^V$  (resp. $u\in X^W$)  with initial data $u(0)=u_0$ to \eqref{eq:NLS*} (resp. to \eqref{eq:NLS}).
Moreover, we have the growth estimate
\begin{equation}\label{growth1}
\|u(t)\|_{H^s}\le C_s(1+|t|)^{\beta_s}\|u_0\|_{H^s},~~t\in\mathbb{R}.
\end{equation}
\end{proposition}
The spaces $X^V,X^W$ appearing in the statement above are Bourgain spaces adapted to the operator $-\partial_x^2+V\ast$, respectively $-\partial_x^2+W$. We refer to section~\ref{sec:gwp} below for proper definitions and properties of these spaces.

We stress out that although this result is not surprising for specialists, it requires many generalizations of multilinear estimates in the I-method of the first and second generation. Moreover, due to homogeneity problems, the convolutional and multiplicative cases must be considered differently. We refer the reader to the introduction of section \ref{sec:gwp} for a general presentation of the method, and to appendices \ref{appendix} and \ref{appendixproof} for the technical details.

\medskip

To state our dynamical results, which are the core of this work, we need to define the concept of strongly-non-resonant frequencies:
\begin{definition} \label{def:strg_nr}
Being given $\alpha >0$ and $ \mathcal{I} \subset \mathbb{Z}$, we say that a family of frequencies $\omega \in \mathbb{R}^\mathbb{\mathcal{I}}$ is strongly-non-resonant  if there exist $\alpha,\rho>0$ such that  for all $q\geq 1$, $\boldsymbol{m} \in (\mathbb{Z}^*)^q$ satisfying $\boldsymbol{m}_1+\cdots+\boldsymbol{m}_q=0$, $\boldsymbol{h}_1,\cdots,\boldsymbol{h}_q \in \mathcal{I}$ all distinct, it satisfies
$$
\big|   \sum_{j=1}^{q} \boldsymbol{m}_j \omega_{\boldsymbol{h}_j}  \big| \geq \rho \big(2  \min_{1\leq j \leq q} \langle \boldsymbol{h}_j \rangle\big)^{-\exp(\alpha |\boldsymbol{m}|_1 )}
$$
where $|\boldsymbol{m}|_1 := |\boldsymbol{m}_1|+\cdots +|\boldsymbol{m}_q|$.
\end{definition}
The reader used to the non-resonance conditions for PDEs may be surprised by our definition: the estimate seems quite weak since, in the right-hand side, the exponent decreases exponentially with the length of the linear combination of frequencies considered. We are more used to a polynomial decay. However, one must keep in mind that here the control is done with respect to the smallest index of the frequencies involved whereas, more classically, it is done with respect to the largest index (weak non-resonance) or with respect to the third largest (condition used in \cite{BG06}). This type of condition was already used in \cite{BG21} but quantified in a less precise way. It is this additional precision in the exponent that will allow us to optimize the procedure and reach exponential times.

\subsubsection{Results with a convolution potential}

\begin{theorem} \label{thm:conv}Let $V\in \mathcal{F}L^\infty(\mathbb{T};\mathbb{C})$ be a potential whose Fourier coefficients, $V_j$, $j\in \mathbb{Z}$, are real. If the frequencies $\omega_j = j^2+(2\pi)^\frac12 V_j$ are strongly-non-resonant according to Definition \ref{def:strg_nr}, then the solutions of \eqref{eq:NLS*} enjoy the following property.\\
For all $s>2/5$ and $\nu >0$, there exists $\varepsilon_1\in (0;\epsilon_0]$ and $\mu>0$ such that, if $u^{(0)} \in H^s(\mathbb{T})$ is a function satisfying
$$
\varepsilon := \| u^{(0)} \|_{H^s} \leq \varepsilon_1,
$$
then the global solution $u\in C^0(\mathbb{R};H^s(\mathbb{T}))$ of \eqref{eq:NLS*} with initial condition $u(0)=u^{(0)}$ provided by Proposition~\ref{thm:gwp} satisfies, for all $k \in \mathbb{Z}$ and all $t\in \mathbb{R}$,
$$
|t|< \varepsilon^{- \mu \log \frac{\log \varepsilon^{-1}}{\log (2\langle k\rangle)}} \quad \Longrightarrow \quad \big||u_k(t)|^2 - |u_k(0)|^2\big|\leq \varepsilon^{6-\nu}
$$
where $u_k = (2\pi)^{-\frac12}\int_{\mathbb{T}} u(x) e^{-ikx} \mathrm{d}x$.
\end{theorem}
The next proposition states that, by randomizing the Fourier coefficient of $V$, with a reasonable law (we choose Gaussian law, but other choices are possible),   the strong non-resonance condition is almost surely satisfied.

\begin{proposition} \label{prop:main_conv} Let $s_*>0$ and $V^{\eqref{eq:NLS*}} \in \mathcal{D}'(\mathbb{T};\mathbb{C})$ be the random potential defined by
\begin{equation}\label{random*}
V^{\eqref{eq:NLS*}}(x) =(2\pi)^{-\frac12}\sum_{k \in \mathbb{Z}} X_k \langle k \rangle^{-s_*} e^{ikx}, 
\end{equation}
where  $X_k \sim \mathcal{N}(0;1)$ are normalized independent real Gaussian random variables. For $k\in \mathbb{Z}$, let $\omega^{\eqref{eq:NLS*}}_k := k^2 + (2\pi)^\frac12 V^{\eqref{eq:NLS*}}_k$ be the frequencies of \eqref{eq:NLS*}. \\
Then, almost surely, the frequencies $\omega^{\eqref{eq:NLS*}}$ of \eqref{eq:NLS*} are strongly-non-resonant.
\end{proposition}

\begin{remark}\label{rk:mu}
The constant $\mu$ in Theorem \ref{thm:conv} depends only on the potential $V$ through the parameter $\alpha$ (the exponent in Definition \ref{def:strg_nr}). Moreover, in Proposition \ref{prop:main_conv}, the parameter $\alpha$ depends only on the potential $V$ though its regularity $s_*$. \end{remark}

We postpone to section \ref{sec:comments} the comments about Theorem \ref{thm:conv}.

\subsubsection{Results with a multiplicative potential}
Our result concerning \eqref{eq:NLS} is a bit more complicated to state, essentially due to spectral complications: $-\partial^2_{x} +V\ast$ diagonalizes in the Fourier basis but $-\partial^2_{x}+W$ diagonalizes in its own Hilbert basis. To simplify the presentation, we will focus on the Dirichlet problem, and refer to \cite{BG21} to explain why the result is  more complicated (but reachable) in the periodic case. We assume $W$ to be even, so that we can identify the Dirichlet condition with a symmetry condition on the solution of the periodic problem: we are interested in solutions of \eqref{eq:NLS} that satisfy $u(x)=-u(-x)$ for almost any $x\in\T$. This definition of the Dirichlet problem still makes sense in low regularity, $u\in H^s$ with $s<1/2$. First, we need some results about the Dirichlet spectrum of the Sturm--Liouville operator. Given a potential $W\in L^2(\T)$, we still denote by $W$ its restriction on $[0;\pi]$.

 \begin{proposition}[Thm 7 page 43 of \cite{PT}]\label{prop_dir}  
  For all real-valued $W\in L^2(0;\pi)$, there exist an increasing sequence of real numbers $(\lambda_n)_{n\geq 1}$ and a Hilbertian basis $(f_n)_{n\geq 1}$ of $L^2(0;\pi)$, composed of functions $f_n \in H^2 \cap H^1_0$, such that for all $n\geq 1$ we have $f_n(0) = f_n(\pi) = 0$ and 
\begin{equation}
\label{eq_SL_dir}
-\partial_x^2 f_n(x) + W(x) f_n(x) =\lambda_n f_n(x), \quad  \forall x\in (0;\pi).
\end{equation}
\end{proposition}
Now we can state our result for  \eqref{eq:NLS}:
\begin{theorem} \label{thm:mult}
Let $W\in H^4(\mathbb{T};\mathbb{R})$ be a real valued even potential, $(\lambda_n)_{n\geq 1}$ be the increasing sequence of eigenvalues of the Sturm--Liouville operator $-\partial_x^2 + W_{|[0;\pi]}$ with homogeneous Dirichlet boundary conditions and $(f_n)_{n\geq 1}$ be the associated eigenfunctions (see Prop \ref{prop_dir}).  If the frequencies $\omega = (\lambda_n)_{n\geq 1}$ are strongly-non-resonant according to Definition \ref{def:strg_nr}, then the solutions of \eqref{eq:NLS} enjoy the following property. \\
For all $s>2/5$ and $\nu >0$, there exists $\varepsilon_0\in(0;1]$ and $\mu>0$ such that, if  $u^{(0)} \in H^s(\mathbb{T})$ is an odd function satisfying
$$
\varepsilon := \| u^{(0)} \|_{H^s} \leq \varepsilon_0,
$$
then the global solution $u\in C^0(\mathbb{R};H^s(\mathbb{T}))$ of \eqref{eq:NLS} with initial condition $u(0)=u^{(0)}$ provided by Proposition~\ref{prop:GWPW} satisfies, for all $k \geq 1$ and all $t\in \mathbb{R}$,
$$
|t|< \varepsilon^{- \mu \log \frac{\log \varepsilon^{-1}}{\log (\langle k\rangle)}} \quad \Longrightarrow \quad \big||u_k(t)|^2 - |u_k(0)|^2\big|\leq \varepsilon^{6-\nu}
$$
where $u_k =  \int_{0}^\pi u(x) f_k(x) \mathrm{d}x$.
\end{theorem}

\begin{proposition} \label{prop:main_mult} Let $s_*> 3/2$ and $ W^{\eqref{eq:NLS}} \in L^2(\mathbb{T};\mathbb{R})$ be the even random potential defined by
\begin{equation}\label{random}
W^{\eqref{eq:NLS}}(x) = \sum_{k \geq 1} X_k \langle k \rangle^{-s_*} \cos(k x),
\end{equation}
where  $X_k \sim \mathcal{N}(0;1)$ are normalized independent real Gaussian random variables. For $k\geq 1$, let $\omega^{\eqref{eq:NLS}}_k := \lambda_k$ be the $k$-th smallest eigenvalue of the Sturm--Liouville operator $-\partial_x^2 + W^{\eqref{eq:NLS}}$ with homogeneous Dirichlet boundary conditions on $[0;\pi]$ (see Proposition \ref{prop_dir}). \\
Then there exists a constant $\eta>0$ such that, almost surely, provided that $\| W^{\eqref{eq:NLS}}\|_{H_1} \leq \eta$, the frequencies $\omega^{\eqref{eq:NLS}}$  are strongly-non-resonant.
\end{proposition}
\begin{remark}
\begin{itemize}
\item Remark \ref{rk:mu} also holds in the case of  Theorem \ref{thm:mult} and Proposition \ref{prop:main_mult}.
\item The constant $\eta$ is universal: it does not depend on $s_*$. 
\item In Proposition \ref{prop:main_mult}, the average of the potential is equal to $0$ (i.e. $W^{\eqref{eq:NLS}}_0=0$). Nevertheless, this assumption is not restrictive. Indeed, due to the condition $\boldsymbol{m}_1+\cdots+\boldsymbol{m}_q=0$ in Definition \ref{def:strg_nr}, if the frequencies associated with an even potential $W \in L^\infty(\mathbb{T};\mathbb{R})$ are strongly-non-resonant then the frequencies associated with $W+\upsilon$ are also strongly-non-resonant for all $\upsilon \in \mathbb{R}$.
\end{itemize}
\end{remark}


\subsubsection{Comments on both results}\label{sec:comments}

\begin{itemize}
\item Theorem \ref{thm:conv}  and Theorem \ref{thm:mult} give a control on all the Fourier modes, but the time during which we have this control depends on the index of the considered mode: the higher this index is (high mode) the less the time is. The result is mostly interesting for low modes. In that case we have a control during exponentially long time in the spirit of some Nekhoroshev results recently obtained (see \cite{BG22,BMP20,FG13}). In fact, for the very high modes, the time becoming very short, the result is rather the consequence of the well-posed character of the equation in $H^s$ (see section \ref{sec:high}).
\item From a more physical point of view, these results prove that if there are reverse energy cascades, they are necessarily very slow. Indeed, since we consider non-smooth solutions, an important part of the energy of the solution could be on high modes. However, we prove that, for very long times, there is no transfer of this energy to low modes. 
\item In fact, if we would focus only on the low modes (concretely $2| k | \leq \varepsilon^{-\upsilon_{\alpha,s,\nu}} $; see section \ref{sec:low}), we would not have to assume $u(0)$ to be small in $H^s$ norm but only in $L^2$ norm (i.e. $u(0)\in H^s$ with $\|u(0)\|_{L^2}\ll 1$). The reason being that we develop the normal form in $L^2$.
\item The Strichartz estimate (\cite{Bou93})
\begin{equation}\label{strichartz}
\big(\int_{\T^2}|e^{it\Delta}\varphi|^6dxdt\big)^{1/6}\leq \big(\exp C\frac{\log N}{\log \log N}\big) \|\varphi\|_{L^2}
\end{equation}
assuming $\text{supp}\,\hat\varphi \subset [-N,\cdots,N]$,
 is used to initiate the Birkhoff normal form procedure (see the sketch of proof below). This is actually one of the reasons why we have to truncate to a finite number of modes from the very beginning. 
\item As said above, the results are consequences of Birkhoff normal forms in $H^{0^+}$ (or more precisely in $L^2$ with a logarithmic loss in terms of the order of the Fourier truncation).   The assumption on the regularity of the initial datum, $u(0)\in H^{s}$,  is used to have a control on the remainder term generated by the truncation to a finite number of modes of the nonlinear term. Unfortunately, we are not able to control such remainder for solutions that belong only in $L^2$. 
\item Once we assume  $u(0)\in H^{s}$, we need to control $u(t)\in H^{s}$. This a priori control of $u(t)$ for $t$ large is a by-product of the argument used to globalize solutions in Li-Wu-Xu \cite{LWX}: using the $I$-method, their argument implies, in the case $V=W=0$, that
$\|u(t)\|$ growths at most polynomially in time. In section \ref{sec:gwp} we extend this result to $V\neq 0$ in \eqref{eq:NLS*} and to $W\neq 0$ in \eqref{eq:NLS}.
\item It is very likely that we could prove the same result for the {cubic NLS}
$$
i\partial_t u = -\partial_x^2 u + W u + |u|^2 u, \quad x\in \mathbb{T}
$$
 which is {globally well-posed in $L^2$} \cite{Bou93}. Nevertheless, our method would require to deal with solutions at least in $H^{1/6}$ (to ensure the Sobolev embedding $H^s \subset L^3$ used to control the error term coming from the Fourier truncation). To get a dynamic result even in $L^2$ more work is needed, but it seems conceivable...
\item In this paper, we really use the regularizing effect of the integration in time, since we crucially use the Strichartz estimate \eqref{strichartz}. But we transform this effect in a structural property on $P_6(u)=\frac16\int_\T |u|^6 dx$ (see \eqref{esti:P6}). So we do not work in  space-time (Fourier-Lebesgue spaces) but only in space. It is likely that by working on our normal forms directly in Bourgain spaces, and thus in space-time, the results would improve and in any case be more intrinsic. Nevertheless, the normal forms as we know them at the moment do not take into account the time variable, so it would be a non-trivial conceptual jump.

\end{itemize}

\subsubsection{Related literature}
 Initiated by Bourgain \cite{Bou96}, and refined by Bambusi \cite{Bam03} and Bambusi-Gr\'ebert \cite{BG06}, the Birkhoff normal form method has been widely used in the last decades to show, in its non-resonant version, stability over long times \cite{Bou96,BG06,BDGS07,GIP,Del12,BD17,KillBill,BMP20,FI19,BFM22,BMM22}.
However, all these results  have a major flaw, they only concern very regular solutions (in $H^s$ for $s\gg 1$). The numerical simulations of Cohen-Hairer-Lubich (\cite{CHL08a, CHL08b}) rather suggest that stability over long times is not related to the regularity of solutions. On the other hand, and at the same time, the dispersive PDE community has developed a lot of ingeniousness, based on linear (Strichartz) and multilinear estimates, to demonstrate the local well-posedness for less and less regular initial data. By working in space-time spaces  to use the regularizing effect of the integration in time and by working in a neighborhood of the linear solutions (Bourgain space), one finally succeeds in showing the well-posedness in Sobolev spaces $H^s$ with $s$ very small, even $s=0$ \cite{Bou93}, for the cubic 1d-NLS; in any case below $s=d/2$ ($d$ being the spatial dimension, which in this paper will always be $d=1$). In this kind of space the nonlinear analysis becomes very delicate since the multiplication of two functions is not a stable operation anymore. 
Dispersive properties have been first used  in the context of the whole space $\R^d$ (see \cite{Klai,Sha}) but then extended in the periodic case by Bourgain \cite{Bou93} and more generally in a compact manifold \cite{BGT}. For a general overview, one could consult the book by Tao \cite{Tao06} or the book by Erdoğan-Tzirakis \cite{ET}.\\
Recently (in \cite{BG21,BGR21,Charb}) we proved  Birkhoff normal form results in the energy space ($H^1$ for NLS) leading to a control of the low actions of the equation. In \cite{BG22} we succeeded to control also the $H^s$ norm but only for NLS (in any dimension, $s>d/2$) with specific convolutional potentials. 


\subsection{Sketch of proof}\label{sec:sketch}
\subsubsection{General strategy}
Let us first briefly recall the general strategy of the Birkhoff normal form (see \cite{Bam07} or \cite{Gre07} for a more detailed introduction to Birkhoff normal forms for Hamiltonian PDEs). We begin with the Hamiltonian formulation of \eqref{eq:NLS*} (in this section we focus on the convolution version of NLS which is a bit simpler). 
 Identifying a function with the sequence of its Fourier coefficients  $L^2(\T)\ni u \equiv (u_n)_{n\in\Z}$ where $u_n:=(2\pi)^{-\frac12}\int_\T u(x)e^{-inx}dx$, \eqref{eq:NLS*} reads
$$i\partial_t u_k =\nabla H(u)_k$$
where the Hamiltonian function of \eqref{eq:NLS*} is given by
$$H(u)= Z_2(u)+P_6(u),$$
$$Z_2(u)=\sum_{k\in\Z}\omega_k |u_k|^2$$
and
$$P_6(u)=\frac16\int_\T |u|^6 dx= \frac16\sum_{{\boldsymbol{k}_1} +{\boldsymbol{k}_2} + {\boldsymbol{k}_3} ={\boldsymbol{\ell}_1} + {\boldsymbol{\ell}_2} + {\boldsymbol{\ell}_3}}u_{\boldsymbol{k}_1}u_{\boldsymbol{k}_2} u_{\boldsymbol{k}_3} \ \overline{u_{\boldsymbol{\ell}_1}} \overline{u_{\boldsymbol{\ell}_2}} \overline{u_{\boldsymbol{\ell}_3}} .$$
To a monomial $ u_{\boldsymbol{k}_1} \cdots u_{\boldsymbol{k}_q} \ \overline{u_{\boldsymbol{\ell}_1}} \cdots\overline{u_{\boldsymbol{\ell}_q}}$ ($q\geq 3$) we associate the small divisor
 $$
 \Omega(\boldsymbol{k}, \boldsymbol{\ell}):=\omega_{\boldsymbol{k}_1}+\cdots+\omega_{\boldsymbol{k}_q}-\omega_{\boldsymbol{\ell}_1}-\cdots-\omega_{\boldsymbol{\ell}_q}\neq 0.
 $$
Given $\gamma>0$, by solving a so-called cohomological equation, we can remove any monomials with $|\Omega(k,\ell)|>\gamma$ replacing it by a higher order term.
So, for a given $r\geq 3$, we formally construct a change a variable $\tau$ such that
$$H\circ \tau= Z_r +R_r$$
where $Z_r$ contains only monomials for which $|\Omega(k,\ell)|\leq\gamma$ (i.e. $\gamma$-resonant monomials in the sense of \eqref{eq:gamma-res})and $R_r$ is of order $r$: $R_r(u)=O(u^r)$. 
We prove in section \ref{sec:proof_thm_conv} that such $\gamma$-resonant polynomial $Z_r$ will not modify\footnote{see in particular the beginning of the optimization procedure, section \ref{sub:opt*}.} the dynamics of the low actions (the precise meaning of "low" depending on the value of $\gamma$).
On the other hand, $R_r$ is small in the sense that it has a high order, but the precise meaning of this smallness will depend a lot on the topology in which we perform the normal form. 

\subsubsection{BNF in Euclidean topology} Here the idea is to perform the normal form step without any regularity, i.e. in $L^2$. For that purpose, we use a strategy inspired by \cite{Bou04b} appendix 7. As explained above, we will need to truncate the nonlinear term to a finite number of modes. So we shall consider polynomials depending only on a finite number of complex variables $u_n$, $n \in \mathcal{M} := \llbracket -M,M\rrbracket $ (i.e. polynomials defined on the space of the trigonometric polynomials of degree smaller than or equal to $M\gg 1$). The principal difficulty lies in the choice of the norm (let us call it $\|\cdot\|_{\mathscr C}$) that we can put on polynomials $P$ homogeneous of degree $2q$ in order to have an estimate on its gradient of the form
$$\|\nabla P(u)\|\leq C\| P\|_{\mathscr C} \| u\|^{2q-1},$$
with $C$ independent of $M$ or at most with a logarithmic dependency and $\| \cdot \|$ denotes the canonical Euclidean norm on $\mathbb{C}^\mathcal{M}$ (i.e. the $L^2$ norm of the associated trigonometric polynomial up to the usual Fourier identification).  Because we work in Euclidean topology, we have that
$$\|\nabla P(u)\|\leq 2q \| \tilde P\|_{\mathscr L_{2q}} \| u\|^{2q-1},$$
where $\tilde P \in \mathscr L_{2q}$ is the $2q$-linear map that we can naturally associate with the homogeneous polynomials of degree $2q$. Furthermore, a standard result (due to S. Banach (1937)) says that 
$$ 
\| \tilde P\|_{\mathscr L_{2q}}:=\sup_{\| u\|_{L^2}=1}|P(u)|=\| P\|_\infty.
$$
So the good norm could be $\|\cdot \|_\infty$. But this norm is not controlled from the beginning (for $P_6$), and furthermore we cannot propagate such a control by Poisson brackets (which is necessary to  implement a Birkhoff normal form). The idea, inspired by Bourgain, consists in considering the  level sets of $P$ according to $\Omega^{(i)}(\boldsymbol{k},\boldsymbol{\ell})={\boldsymbol{k}^2_1} + \cdots + {\boldsymbol{k}^2_q} -{\boldsymbol{\ell}^2_1} - \cdots - {\boldsymbol{\ell}^2_q}$, the small divisor associated to the integer part of $\omega_k=k^2+\hat V_k$. We define (see section \ref{sec:setting} for a more intrinsic definition)
$$ \| P\|_{\mathscr{H}}:= \sup_{a\in \Z}\| { \Pi_a\lfloor P \rceil(u)}\|_\infty, \text{ and }  \| P\|_{\mathscr C}=: \sup_{a\in \Z}\langle a\rangle \| {\Pi_a\lfloor P \rceil(u)}\|_\infty$$
where, given a polynomial 
$$P(u)=  \sum_{{\boldsymbol{k}_1} + \cdots + {\boldsymbol{k}_q} ={\boldsymbol{\ell}_1} + \cdots + {\boldsymbol{\ell}_q} } P_{\boldsymbol{k},\boldsymbol{\ell}} u_{\boldsymbol{k}_1} \dots u_{\boldsymbol{k}_q} \overline{u_{\boldsymbol{\ell}_1}} \dots \overline{u_{\boldsymbol{\ell}_q}}$$ we define
$$\lfloor P \rceil(u) =  \sum_{{\boldsymbol{k}_1} + \cdots + {\boldsymbol{k}_q} ={\boldsymbol{\ell}_1} + \cdots + {\boldsymbol{\ell}_q}} |P_{\boldsymbol{k},\boldsymbol{\ell}}| u_{\boldsymbol{k}_1} \dots u_{\boldsymbol{k}_q} \overline{u_{\boldsymbol{\ell}_1}} \dots \overline{u_{\boldsymbol{\ell}_q}}$$
(the so-called modulus of $P$) and\footnote{Take care that in section \ref{sec:setting} we have a more general definition of this projection, see \eqref{pi}. }
$${\Pi_aP}(u):=  \sum_{\substack{{{\boldsymbol{k}_1} + \cdots + {\boldsymbol{k}_q} ={\boldsymbol{\ell}_1} + \cdots + {\boldsymbol{\ell}_q}}\\{{\boldsymbol{k}^2_1} + \cdots + {\boldsymbol{k}^2_q} -{\boldsymbol{\ell}^2_1} - \cdots - {\boldsymbol{\ell}^2_q}=a}}} P_{\boldsymbol{k},\boldsymbol{\ell}} u_{\boldsymbol{k}_1} \dots u_{\boldsymbol{k}_q} \overline{u_{\boldsymbol{\ell}_1}} \dots \overline{u_{\boldsymbol{\ell}_q}}.$$
It is important to notice that, solving the cohomological equation, we transform a polynomial controlled by $\|\cdot\|_{\mathscr H}$ to polynomials controlled by $\|\cdot\|_{\mathscr C}$ (this is a consequence of Lemma \ref{lem:cestdiagonal}).
With these two topologies on homogeneous polynomials we prove in Lemma \ref{lem:comp_with_classical_norms} that $\| P\|_\infty \leq 5 \log(2qM^2)\| P\|_{\mathscr C}$, which implies the desired estimate on $\|\nabla P(u)\|$, and in Proposition \ref{prop:est_poisson} that
$$\|\{ P, Q\}\|_{\mathscr H}\leq 40 qq' \log(2q' M^2)\|P\|_{\mathscr H}\| Q\|_{\mathscr C},$$
which is perfect to implement the Birkhoff normal form procedure (up to an unessential $\log M$ loss).

\begin{remark}
The use of restrictions to level sets of the resonance function in multilinear estimates is also reminiscent of the estimates used in performing \emph{Poincaré-Dulac} normal forms; see e.g. \cite{BIT,KO,GKO} and the abstract framework highlighted in \cite{Ki}. These estimates are essentially just another facet of the multilinear estimates in Bourgain spaces, as was recently pointed out in \cite{Cor} (in the case of $\mathbb{R}^d$ instead of $\mathbb{T}^d$).
\end{remark}

\subsubsection{End of the proof}
For the first step of normal form, we have to prove that $P_6$ can be controlled by $\|\cdot\|_{\mathscr H}$ and this is a simple consequence of the Strichartz estimate \eqref{strichartz} which leads to (see section \ref{sec:strichartz})
\begin{equation}\label{esti:P6}
\| P^{(M)}_6\|_{\mathscr H}\leq  \big(\exp C\frac{\log M}{\log \log M}\big),
\end{equation}
where $P^{(M)}_6$ is the restriction of $P_6$ to the modes whose indices are in $\llbracket -M,M \rrbracket$. At some point we have to take into account this truncation, i.e. we have to control the remainder term
$$\|\Pi_M\big( \nabla P_6(u)- \nabla P^{(M)}_6(u))\|_{L^2}$$
where
$\Pi_M$ denotes the projection on the modes whose indices are in $\llbracket -M,M\rrbracket$ (note that $ P^{(M)}_6 := P_6 \circ \Pi_M$). Unfortunately, we were not able to estimate such quantity for solutions that only belong to $L^2$. For solutions  in $H^s$, $s> 2/5$, this follows by using the Sobolev embedding $H^s\subset L^{10}$, which leads to (see section \ref{sub:estg*})
$$\|\Pi_M\big( \nabla P_6(u)- \nabla P^{(M)}_6(u))\|_{L^2}\leq M^{-\alpha(s)}\|u\|^5_{H^s},$$
with $\alpha(s)>0$ for $s> 2/5$.
Then it suffices to control the growth of the $H^s$ norm of the solution, $\| u(t)\|_{H^s}$, and this follows from the argument in \cite{LWX} for $s>2/5$ when $V=0$, a result that we extend to non-vanishing $V$ in section \ref{sec:gwp}.
The last, but not least, step is to optimize the set of parameters as a function of $\eps$ as it is usual to obtain an exponential time (see section \ref{sub:opt*}).

\subsection{Acknowledgments} During the preparation of this work the authors benefited from the support of the Centre Henri Lebesgue ANR-11-LABX-0020-0 and J.B. was also supported by the region "Pays de la Loire" through the project "MasCan".
T.R. was partially supported by the ANR project Smooth ANR-22-CE40-0017. J.B. and B.G. were partially supported by the ANR project KEN ANR-22-CE40-0016.

\section{A Birkhoff normal form theorem in Euclidean spaces}

\subsection{Functional setting}\label{sec:setting} Let $\mathcal{M}$ be a finite set. We endow  $\mathbb{C}^{\mathcal{M}}$ with its canonical Euclidean structure 
$$
\forall u,v \in  \mathbb{C}^{\mathcal{M}}, \quad \langle u,v \rangle = \Re \sum_{k\in \mathcal{M}} u_k \overline{v_k},\quad \|u\|^2 = \langle u,u\rangle, 
$$
and with its canonical symplectic form $ \langle  i\cdot,\cdot \rangle .$

Being given two smooth real valued functions $P,Q : \mathbb{C}^{\mathcal{M}} \to \mathbb{R}$, their \emph{Poisson bracket} is defined by
$$
\{P,Q\}(u) = \langle i\nabla P(u), \nabla Q(u) \rangle,
$$
where the gradients are defined by duality by
$$
\forall u,v\in  \mathbb{C}^{\mathcal{M}}, \quad \langle \nabla P(u) , v \rangle = \mathrm{d}P(u)(v)
$$
and satisfy the usual formula 
$$
\forall k \in \mathcal{M}, \quad (\nabla P(u))_k = 2\partial_{\overline{u_k}} P(u) := \partial_{\Re u_k} P(u) +  i  \partial_{\Im u_k} P(u).
$$
Thus we have the formula
\begin{equation}
\label{eq:Poisson_formula}
\{  P,Q\} (u) 
  =2i \sum_{k\in \mathcal{M}} \partial_{\overline{u_k}}P(u) \partial_{u_k} Q(u) - \partial_{u_k}P(u) \partial_{\overline{u_k}} Q(u).
\end{equation}
Note that this formula also makes sense if $P,Q$ are complex-valued.

\subsubsection{Polynomials}
Being given two real vector spaces, a map $P : E \to F$ is called \emph{homogeneous polynomial of degree $d$} if there exists a $d$-$\mathbb{R}$-linear symmetric map $L :(E )^d \to F $ such that $P(u) = L(u,\cdots,u)$. Note that $L$ is unique. In order to estimate these polynomials, we recall the following useful proposition:
\begin{proposition}[Prop 1 page 61 of \cite{BS71}]\label{prop:BS} Let $E$ be a real Hilbert space, $F$ be a real Banach space, $P : E \to F$ be a homogeneous polynomial of degree $d$ and $L$ be the associated $d$-linear symmetric map. Then we have
$$
\|P\|_{\infty}:= \sup_{\|u\|_E\leq 1} \|P(u)\|_F = \sup_{\|u^{(1)}\|_E\leq 1, \cdots, \|u^{(d)}\|_E\leq 1} \|L(u^{(1)},\cdots,u^{(d)})\|_F.
$$
\end{proposition}
\begin{corollary} \label{cor:very_nice}
Under the assumptions of Proposition~\ref{prop:BS}, if $P$ is real valued (i.e. $F=\mathbb{R}$), then for all $u,v\in E$,
$$
\| \nabla P(u) \|_E \leq d \|P\|_{\infty} \|u\|_E^{d-1}  \quad \mathrm{and} \quad \| \mathrm{d}\nabla P(u)(v) \|_E \leq d (d-1)  \|P\|_{\infty}  \|u\|^{d-2}_E \|v\|_E.
$$
\end{corollary}
Being given $q\geq 1$, we denote by $\mathcal{P}^{\mathcal{M}}_{\mathbb{K},2q}$ the set of the $\mathbb{K}$ valued homogeneous polynomial of degree $2q$ which commute with the Euclidean norm, i.e.
$$
\mathcal{P}^{\mathcal{M}}_{\mathbb{K},2q} := \Big\{P : \mathbb{C}^{\mathcal{M}} \to \mathbb{K} \ | \ P\ \mathrm{is \ a \ }\mathbb{R}\mathrm{-homogeneous\ polynomial \ of\ degree} \ 2q \ \mathrm{and} \ \{ P ,\|\cdot\|^2\} =0\Big\}.
$$
Note that the polynomials $P\in \mathcal{P}^{\mathcal{M}}_{\mathbb{C},2q}$ are exactly those admitting a decomposition of the form
$$
P(u) =  \sum_{\boldsymbol{k},\boldsymbol{\ell}\in \mathcal{M}^q  } P_{\boldsymbol{k},\boldsymbol{\ell} } \, u_{\boldsymbol{k}_1} \dots u_{\boldsymbol{k}_q} \overline{u_{\boldsymbol{\ell}_1}} \dots \overline{u_{\boldsymbol{\ell}_q}}
$$
with $P_{\boldsymbol{k},\boldsymbol{\ell} }  \in \mathbb{C}$ satisfying the symmetry condition
$$
 \forall \phi,\sigma \in \mathfrak{S}_q, \quad P_{\phi \boldsymbol{k},\sigma\boldsymbol{\ell} } =  P_{\boldsymbol{k},\boldsymbol{\ell} }.
$$
Moreover thanks to the symmetry condition, this decomposition is unique. Furthermore, if $P \in \mathcal{P}^{\mathcal{M}}_{\mathbb{R},2q}$ is real-valued, its coefficients satisfy the reality condition
$$
P_{\boldsymbol{\ell} , \boldsymbol{k}}  = \overline{P_{\boldsymbol{k},\boldsymbol{\ell} } }.
$$
Of course, as stated in the following lemma, thanks to the Jacobi identity, this class of Hamiltonian is stable by Poisson bracket.
\begin{lemma} Let $\mathbb{K}\in \{ \mathbb{R},\mathbb{C} \}$, $q,q'\geq 1$, $P\in \mathcal{P}^{\mathcal{M}}_{\mathbb{K},2q} $ and $ Q\in \mathcal{P}^{\mathcal{M}}_{\mathbb{K},2q'}$ be two $\mathbb{R}$-homogeneous polynomials commuting with the Euclidean norm, then $\{P,Q\} \in \mathcal{P}^{\mathcal{M}}_{\mathbb{K},2(q+q'-1)}$  is also a homogeneous polynomial commuting with the Euclidean norm. 
\end{lemma}

\subsubsection{Hamiltonian flows} We recall that a smooth map $\Psi : \mathbb{C}^\mathcal{M} \to \mathbb{C}^\mathcal{M}$ is \emph{symplectic} if its derivative preserves the canonical symplectic form, i.e.
$$
\forall u,v,w\in \mathbb{C}^\mathcal{M}, \quad \langle i\mathrm{d}\Psi(u)(v) , \mathrm{d}\Psi(u)(w) \rangle = \langle iv , w \rangle.
$$ 
In the following proposition, we give some of the properties enjoyed by Hamiltonian flows generated by real valued homogeneous polynomials commuting with the Euclidean norm.
\begin{proposition} \label{prop:ca_flotte} Let $q\geq 2$ and $\chi \in \mathcal{P}^{\mathcal{M}}_{\mathbb{R},2q}$. Then the flow $\Phi^t_\chi$ of the equation
\begin{equation}
\label{eq:kiki}
i\partial_t u = \nabla \chi(u)
\end{equation}
is smooth and global. Moreover, it enjoys the following properties:
\begin{itemize}
\item preservation of the Euclidean norm: 
$$
\forall t\in \mathbb{R}, \forall u \in \mathbb{C}^\mathcal{M}, \quad \|\Phi_\chi^t(u)\| = \|u\|.
$$
\item it is close to the identity:
$$
\forall t\in \mathbb{R}, \forall u \in \mathbb{C}^\mathcal{M}, \quad \|\Phi_\chi^t(u)-u\| \leq 2q |t| \| \chi \|_\infty  \|u\|^{2q-1}.
$$
\item it is symplectic: for all $t\in \mathbb{R}$, $\Phi_\chi^t$ is symplectic.
\item its differential is under control:
\begin{equation}
\label{eq:cequonveut}
\forall t\in \mathbb{R}, \forall u,v \in \mathbb{C}^\mathcal{M}, \quad \|\mathrm{d}\Phi_\chi^t(u)(v)\| \leq \exp( 4 q^2 t \| \chi \|_\infty \|u \|^{2q-2} ) \|v\|.
\end{equation}
\end{itemize}
\end{proposition}
\begin{proof}
The local well-posedness of the equation \eqref{eq:kiki} follows directly from the Cauchy--Lipschitz theorem. The preservation of the Euclidean norm comes directly from the commutation between $\chi$ and $\| \cdot \|^2$. This conserved quantity provides directly the global well-posedness of \eqref{eq:kiki}. Since \eqref{eq:kiki} is Hamiltonian, it is well-known that its flow is symplectic. 

Integrating \eqref{eq:kiki}, thanks to Corollary \ref{cor:very_nice}, we get
$$
\|\Phi_\chi^t(u)-u\| =  \Big\| \int_0^t \nabla \chi(\Phi_\chi^\tau(u) ) \mathrm{d}\tau \Big\| \leq 2q \| \chi \|_\infty  \int_{[0;t]}  \|\Phi_\chi^\tau(u)\|^{2q-1}  \mathrm{d}\tau = 2q |t| \| \chi \|_\infty  \|u\|^{2q-1}.
$$
Differentiating \eqref{eq:kiki}, we have
$$
-i\partial_t \mathrm{d}\Phi_\chi^t(u)(v) = \mathrm{d}\nabla \chi(\Phi_\chi^t(u))(\mathrm{d}\Phi_\chi^t(u)(v) ).
$$
Moreover, thanks to Corollary \ref{cor:very_nice}, we have
$$
\| \mathrm{d}\nabla \chi(\Phi_\chi^t(u))(\mathrm{d}\Phi_\chi^t(u)(v) ) \| \leq 4q^2 \| \chi \|_\infty \|u \|^{2q-2} \| (\mathrm{d}\Phi_\chi^t(u))^*(v)\|
$$
and so, by Gr\"onwall's lemma, we get \eqref{eq:cequonveut}.
\end{proof}

\subsubsection{Modulus}

Following \cite{Nik86,BG06}, being given $P\in \mathcal{P}^{\mathcal{M}}_{\mathbb{C},2q}$, we define its \emph{modulus} $\lfloor P \rceil \in \mathcal{P}^{\mathcal{M}}_{\mathbb{C},2q}$ by
$$
\lfloor P \rceil(u) = \sum_{\boldsymbol{k},\boldsymbol{\ell}\in \mathcal{M}^q  } |P_{\boldsymbol{k},\boldsymbol{\ell} }| \, u_{\boldsymbol{k}_1} \dots u_{\boldsymbol{k}_q} \overline{u_{\boldsymbol{\ell}_1}} \dots \overline{u_{\boldsymbol{\ell}_q}}.
$$
Of course, thanks to the triangle inequality, it is clear that $\| P  \|_{\infty} \leq \| \lfloor P \rceil \|_{\infty}$. Furthermore, we have the following useful lemma.
\begin{lemma} \label{lem:avg} Let $q\geq 1$ and $c_{\boldsymbol{k},\boldsymbol{\ell}} \in \mathbb{C}$ be some coefficients, with $\boldsymbol{k},\boldsymbol{\ell}\in \mathcal{M}^q$. If $P\in \mathcal{P}^{\mathcal{M}}_{\mathbb{C},2q}$ is the polynomial defined by
$$
P(u) = \sum_{\boldsymbol{k},\boldsymbol{\ell}\in \mathcal{M}^q} c_{\boldsymbol{k},\boldsymbol{\ell}} u_{\boldsymbol{k}_1} \dots u_{\boldsymbol{k}_q} \overline{u_{\boldsymbol{\ell}_1}} \dots \overline{u_{\boldsymbol{\ell}_q}}
$$
then, for all $u\in \mathbb{C}^\mathcal{M}$, we have
$$
|\lfloor P \rceil(u) | \leq \sum_{\boldsymbol{k},\boldsymbol{\ell}\in \mathcal{M}^q} |c_{\boldsymbol{k},\boldsymbol{\ell}}| |u_{\boldsymbol{k}_1} \dots u_{\boldsymbol{k}_q} \overline{u_{\boldsymbol{\ell}_1}} \dots \overline{u_{\boldsymbol{\ell}_q}} |
$$
\end{lemma}
\begin{proof}
Indeed, since the coefficients of $P$ are given by
$$
P_{\boldsymbol{k},\boldsymbol{\ell} } = (q!)^{-2} \sum_{ \phi,\sigma \in \mathfrak{S}_q}  c_{\phi \boldsymbol{k},\sigma\boldsymbol{\ell} }
$$
then, by the triangle inequality, we have
\begin{equation*}
\begin{split}
|\lfloor P \rceil(u)| &= \Big|  \sum_{\boldsymbol{k},\boldsymbol{\ell}\in \mathcal{M}^q} \big|(q!)^{-2} \sum_{ \phi,\sigma \in \mathfrak{S}_q}  c_{\phi \boldsymbol{k},\sigma\boldsymbol{\ell} }\big| u_{\boldsymbol{k}_1} \dots u_{\boldsymbol{k}_q} \overline{u_{\boldsymbol{\ell}_1}} \dots \overline{u_{\boldsymbol{\ell}_q}} \Big| \\
&\leq (q!)^{-2} \sum_{ \phi,\sigma \in \mathfrak{S}_q}   \sum_{\boldsymbol{k},\boldsymbol{\ell}\in \mathcal{M}^q} | c_{\phi \boldsymbol{k},\sigma\boldsymbol{\ell} }| | u_{\boldsymbol{k}_1} \dots u_{\boldsymbol{k}_q} \overline{u_{\boldsymbol{\ell}_1}} \dots \overline{u_{\boldsymbol{\ell}_q}} | \\
&= \sum_{\boldsymbol{k},\boldsymbol{\ell}\in \mathcal{M}^q} | c_{ \boldsymbol{k},\boldsymbol{\ell} }|  (q!)^{-2} \sum_{ \phi,\sigma \in \mathfrak{S}_q}  | u_{\boldsymbol{k}_{\phi_1^{-1}}} \dots u_{\boldsymbol{k}_{\phi_q^{-1}}} \overline{u_{\boldsymbol{\ell}_{\sigma_1^{-1}}}} \dots \overline{u_{\boldsymbol{\ell}_{\sigma_q^{-1}}}} | \\
&= \sum_{\boldsymbol{k},\boldsymbol{\ell}\in \mathcal{M}^q} |c_{\boldsymbol{k},\boldsymbol{\ell}}| |u_{\boldsymbol{k}_1} \dots u_{\boldsymbol{k}_q} \overline{u_{\boldsymbol{\ell}_1}} \dots \overline{u_{\boldsymbol{\ell}_q}} |.
\end{split}
\end{equation*}
\end{proof}
 Moreover, we also have the following useful bilinear estimate.
\begin{lemma}\label{lem:est_poisson_weak} Let $q,q'\geq 1$, $P\in \mathcal{P}^{\mathcal{M}}_{\mathbb{C},2q} $ and $ Q\in \mathcal{P}^{\mathcal{M}}_{\mathbb{C},2q'}$ be two $\mathbb{R}$-homogeneous polynomials commuting with the Euclidean norm, then $\{P,Q\}$ enjoys the following estimate:
$$
\| \lfloor \{ P,Q\} \rceil \|_\infty \leq 8 q q' \| \lfloor P \rceil \|_\infty \| \lfloor Q \rceil \|_\infty.
$$
\end{lemma}
\begin{proof} Thanks to the formula \eqref{eq:Poisson_formula}, setting $q''=q+q'-1$, the polynomial $ \{ P,Q\}$, writes
$$
\{ P,Q\}(u) = \sum_{\boldsymbol{k},\boldsymbol{\ell}\in \mathcal{M}^{q''}} c_{\boldsymbol{k},\boldsymbol{\ell}} u_{\boldsymbol{k}_1} \dots u_{\boldsymbol{k}_{q''}} \overline{u_{\boldsymbol{\ell}_1}} \dots \overline{u_{\boldsymbol{\ell}_{q''}}},
$$
with
$$
c_{\boldsymbol{k},\boldsymbol{\ell}} = 2i q\sum_{j\in \mathcal{M}} P_{\boldsymbol{k}_1,\cdots, \boldsymbol{k}_q, \boldsymbol{\ell}_1,\cdots, \boldsymbol{\ell}_{q-1},j } Q_{\boldsymbol{k}_1,\cdots, \boldsymbol{k}_{q-1},j, \boldsymbol{\ell}_1,\cdots, \boldsymbol{\ell}_{q} } - P_{\boldsymbol{k}_1,\cdots, \boldsymbol{k}_{q-1},j, \boldsymbol{\ell}_1,\cdots, \boldsymbol{\ell}_{q} } Q_{\boldsymbol{k}_1,\cdots, \boldsymbol{k}_q, \boldsymbol{\ell}_1,\cdots, \boldsymbol{\ell}_{q-1},j } .
$$
Therefore, thanks to Lemma \ref{lem:avg}, we have
$$
|\lfloor \{ P,Q\} \rceil(u)| \leq  \sum_{\boldsymbol{k},\boldsymbol{\ell}\in \mathcal{M}^{q''}} |c_{\boldsymbol{k},\boldsymbol{\ell}}| |u_{\boldsymbol{k}_1} \dots u_{\boldsymbol{k}_{q''}} \overline{u_{\boldsymbol{\ell}_1}} \dots \overline{u_{\boldsymbol{\ell}_{q''}}}|.
$$
From now on, without loss of generality, we only consider vectors $u\in (\mathbb{R}_+)^\mathcal{M}$ with real non-negative components. Then applying the triangle inequality on the expansion of $c_{\boldsymbol{k},\boldsymbol{\ell}}$, we get (on $(\mathbb{R}_+)^\mathcal{M}$)
$$
\lfloor \{ P,Q\} \rceil  \leq 2 \sum_{j\in \mathcal{M}} \partial_{\overline{u_j}} \lfloor P \rceil \partial_{u_j} \lfloor Q \rceil+ \partial_{u_j} \lfloor P \rceil \partial_{\overline{u_j}} \lfloor Q \rceil .
$$
Since all these partial derivatives are non-negative, applying the Cauchy--Schwarz inequality, we get (on $(\mathbb{R}_+)^\mathcal{M}$)
\begin{equation*}
\begin{split}
\lfloor \{ P,Q\} \rceil  &\leq 2 \sum_{j\in \mathcal{M}} ( \partial_{u_j} \lfloor P \rceil + \partial_{\overline{u_j}} \lfloor P \rceil )( \partial_{u_j} \lfloor Q \rceil + \partial_{\overline{u_j}} \lfloor Q \rceil ) \\
 &= 2 \| (\partial_{u_j} \lfloor P \rceil + \partial_{\overline{u_j}} \lfloor P \rceil)_j \| \| (\partial_{u_j} \lfloor Q \rceil + \partial_{\overline{u_j}} \lfloor Q \rceil)_j \|.
\end{split}
\end{equation*}
Finally, noticing that 
$$  
(\nabla (\lfloor P \rceil_{| \mathbb{R}^\mathcal{M}}))_j = \partial_{u_j} \lfloor P \rceil + \partial_{\overline{u_j}} \lfloor P \rceil.
$$
and applying Corollary \ref{cor:very_nice},  we get that, for all $u\in (\mathbb{R}_+)^\mathcal{M}$,
$$
\lfloor \{ P,Q\} \rceil(u) \leq 2 \|\nabla (\lfloor P \rceil_{| \mathbb{R}^\mathcal{M}})(u) \|  \|\nabla (\lfloor Q \rceil_{| \mathbb{R}^\mathcal{M}})(u) \| \leq 8 q q' \|u\|^{2q''}  \| \lfloor P \rceil \|_\infty \| \lfloor Q \rceil \|_\infty.
$$
\end{proof}

\subsubsection{Frequencies and spectral projectors}
Being given a vector of frequencies with real coefficients $\omega \in \mathbb{R}^\mathcal{M}$, we define the quadratic Hamiltonian
\begin{equation}
\label{eq:defZ2}
Z_{2,\omega} = \frac12 \sum_{k\in \mathcal{M}}  \omega_{k}  |u_k|^2.
\end{equation}
The following lemma describes the action of $\mathrm{ad}_{Z_{2,\omega}} := \{ Z_{2,\omega} , \cdot \}$ on $\mathcal{P}^{\mathcal{M}}_{\mathbb{C},2q}$ (it follows from a straightforward calculation).
\begin{lemma} \label{lem:cestdiagonal} For all $q\geq 1$, $\mathrm{ad}_{Z_{2,\omega}}$ is an endomorphism on $\mathcal{P}^{\mathcal{M}}_{\mathbb{C},2q}$ which is diagonal is the basis of the monomials, i.e.
$$
 \{ Z_{2,\omega} , u_{\boldsymbol{k}_1} \dots u_{\boldsymbol{k}_{q}} \overline{u_{\boldsymbol{\ell}_1}} \dots \overline{u_{\boldsymbol{\ell}_{q}}} \} = i  ( \omega_{\boldsymbol{k}_1} + \cdots + \omega_{\boldsymbol{k}_q} -\omega_{\boldsymbol{\ell}_1} - \cdots - \omega_{\boldsymbol{\ell}_q}   ) u_{\boldsymbol{k}_1} \dots u_{\boldsymbol{k}_{q}} \overline{u_{\boldsymbol{\ell}_1}} \dots \overline{u_{\boldsymbol{\ell}_{q}}}
$$
with  $u\in \mathbb{C}^\mathcal{M}$ and $\boldsymbol{k},\boldsymbol{\ell}\in \mathcal{M}^q$.
\end{lemma}
For all $a\in \mathbb{R}$ and $q\geq 1$, we define 
\begin{equation}\label{pi}
 \Pi_{\omega,a} :  \mathcal{P}^{\mathcal{M}}_{\mathbb{C},2q} \to  \mathrm{Ker}( \mathrm{ad}_{Z_{2,\omega}} - i a \mathrm{Id} )
\end{equation}
as the spectral projector on the eigenspace of $\mathrm{ad}_{Z_{2,\omega}} $ associated with the eigenvalue $ia$. More concretely, $\Pi_{\omega,a} $ is also defined through the formula
$$
\Pi_{\omega,a} P(u) :=  \sum_{ \omega_{\boldsymbol{k}_1} + \cdots + \omega_{\boldsymbol{k}_q} -\omega_{\boldsymbol{\ell}_1} - \cdots - \omega_{\boldsymbol{\ell}_q} = a  } P_{\boldsymbol{k},\boldsymbol{\ell}} u_{\boldsymbol{k}_1} \dots u_{\boldsymbol{k}_q} \overline{u_{\boldsymbol{\ell}_1}} \dots \overline{u_{\boldsymbol{\ell}_q}}.
$$
Thanks to these projectors, as stated in the following, the Poisson bracket can be seen as a kind of convolution.
\begin{lemma}\label{lem:conv} Let $q,q'\geq 1$, $P \in   \mathcal{P}^{\mathcal{M}}_{\mathbb{C},2q}$ and $\chi \in   \mathcal{P}^{\mathcal{M}}_{\mathbb{C},2q'}$ then for all $a\in\mathbb{R}$, we have
$$
\Pi_{\omega,a} \{ P,\chi\} = \sum_{b+c = a} \{ \Pi_{\omega,b} P, \Pi_{\omega,c} \chi \}.
$$
\end{lemma}
\begin{proof} Decomposing $P$ and $\chi$ as a sum of eigenvectors of $\mathrm{ad}_{Z_{2,\omega}}$ and then expanding the Poisson bracket, we get
$$
\{ P,\chi\}  = \sum_{b,c\in \mathbb{R}}  \{\Pi_{\omega,b} P,  \Pi_{\omega,c} \chi\}.
$$
As a consequence, it is enough to see that $\{\Pi_{\omega,b} P,  \Pi_{\omega,c} \chi\}$ is an eigenvector of $\mathrm{ad}_{Z_{2,\omega}}$ associated with the eigenvalue $i(b+c)$. Applying the Jacobi identity, we have
\begin{equation*}
\begin{split}
\{ Z_{2,\omega},  \{\Pi_{\omega,b} P,  \Pi_{\omega,c} \chi\} \} &= - \{ \Pi_{\omega,c} \chi ,  \{Z_{2,\omega} , \Pi_{\omega,b} P\} \}  -  \{ \Pi_{\omega,b} P  ,  \{\Pi_{\omega,c} \chi ,Z_{2,\omega} \} \} \\
&= - i b  \{ \Pi_{\omega,c} \chi ,    \Pi_{\omega,b} P \} + i c \{ \Pi_{\omega,b} P  ,  \Pi_{\omega,c} \chi \} = i(b+c) \{ \Pi_{\omega,b} P  ,  \Pi_{\omega,c} \chi \}.
\end{split}
\end{equation*}
\end{proof}

\subsubsection{Spectral norms}
Being given a vector of frequencies with integer coefficients $\omega^{(i)} \in \mathbb{Z}^\mathcal{M}$, we define the two following norms
$$
\| P \|_{\mathscr{H}(\omega^{(i)})} := \sup_{a\in \mathbb{Z}}  \|  \lfloor \Pi_{\omega^{(i)},a} P \rceil \|_\infty \quad \mathrm{and} \quad \| P \|_{\mathscr{C}(\omega^{(i)})} := \sup_{a\in \mathbb{Z}}  \langle a \rangle \| \lfloor \Pi_{\omega^{(i)},a} P \rceil \|_\infty
$$
for all $P\in \mathcal{P}^{\mathcal{M}}_{\mathbb{C},2q}$. First, we note, in the following lemma, that the norm $\| \cdot \|_{\mathscr{H}(\omega^{(i)})}$ is weaker than the norm $\| \lfloor \cdot \rceil  \|_{\infty}$ and that the norm $\| \cdot \|_{\mathscr{C}(\omega^{(i)})} $ controls the norm $\| \lfloor \cdot \rceil  \|_{\infty}$ up to a logarithmic loss  .
\begin{lemma}\label{lem:comp_with_classical_norms} Let $P\in \mathcal{P}^{\mathcal{M}}_{\mathbb{C},2q}$, $q\geq 1$, and $\omega^{(i)} \in \mathbb{Z}^\mathcal{M} \setminus \{0\}$ be a vector of frequencies with integer coefficients. Then we have
$$
\| P \|_{\mathscr{H}(\omega^{(i)})} \leq \| \lfloor P \rceil  \|_{\infty}  \leq 5q |\omega^{(i)}|_\infty \| P \|_{\mathscr{H}(\omega^{(i)})}  \quad \mathrm{and} \quad  \| \lfloor P \rceil  \|_{\infty}  \leq 5\log(2q |\omega^{(i)}|_\infty)\, \| P \|_{\mathscr{C}(\omega^{(i)})}.
$$
where $|\omega^{(i)}|_\infty = \max_{j\in \mathcal{M}} |\omega^{(i)}_j|$.
\end{lemma}
\begin{proof}
First, we note that the projectors $\Pi_{\omega^{(i)},a} $ commute with the modulus $ \lfloor \cdot \rceil$. Therefore, we have
$$
 \lfloor P \rceil = \sum_{a\in \mathbb{Z}} \Pi_{\omega^{(i)},a} \lfloor P \rceil .
$$
As a consequence, since, if $u\in (\mathbb{R}_+)^\mathcal{M}$, we have $\Pi_{\omega^{(i)},a} \lfloor P \rceil (u) \geq 0$, for all $a\in \mathbb{Z}$, we get directly that $\| P \|_{\mathscr{H}(\omega^{(i)})} \leq \| \lfloor P \rceil  \|_{\infty}$.

Then, we note that if $|a|> 2 q |\omega^{(i)}|_\infty $, then $\Pi_{\omega^{(i)},a} \lfloor P \rceil=0$. As a consequence, we get directly that $ \| \lfloor P \rceil  \|_{\infty}  \leq 5q |\omega^{(i)}|_\infty \| P \|_{\mathscr{H}(\omega^{(i)})} $.
Moreover, in the same way, we have
$$
 \| \lfloor P \rceil  \|_{\infty}  \leq  \| P \|_{\mathscr{C}(\omega^{(i)})} \sum_{|a|\leq 2 q |\omega^{(i)}|_\infty } \langle a \rangle^{-1} \leq  \| P \|_{\mathscr{C}(\omega^{(i)})} (3+ 2 \log (2 q |\omega^{(i)}|_\infty)).
$$
Finally, since $ |\omega^{(i)}|_\infty\geq 1$ and $2\log 2 \geq 1$, we get the second estimate.
\end{proof}
\begin{remark}
The logarithmic loss coming from the control of the $\|\cdot\|_{\infty}$ norm by the $\|\cdot\|_{\mathscr{C}(\omega^{(i)})}$ norm is different from that coming from the Strichartz estimate \eqref{strichartz}, and is rather comparable to the logarithmic difference between $X^{s,\frac12}$ and $C(\R;H^s)$. This loss should be avoided by refining the choice of topology on $\mathcal{P}^{\mathcal{M}}_{\mathbb{C},2q}$, in particular the norm controlling the dependence in $a$ for $\Pi_{\omega^{(i)},a}P$.
\end{remark}

The following proposition provides a very useful refinement of the bilinear estimate given by Lemma \ref{lem:est_poisson_weak}.
\begin{proposition}\label{prop:est_poisson} Let $\omega^{(i)} \in \mathbb{Z}^\mathcal{M} \setminus \{0\}$ be a vector of frequencies with integer coefficients, $q,q'\geq 1$, $P \in   \mathcal{P}^{\mathcal{M}}_{\mathbb{C},2q}$ and $\chi \in   \mathcal{P}^{\mathcal{M}}_{\mathbb{C},2q'}$, then $\{P,\chi\}$ enjoys the following bilinear estimate :
\begin{equation}
\label{eq:cest_tres_joli}
\| \{ P,\chi\} \|_{\mathscr{H}(\omega^{(i)})} \leq 40 \, q q'\,  \log ( 2q'|\omega^{(i)}|_\infty) \, \| P \|_{\mathscr{H}(\omega^{(i)})}  \| \chi \|_{\mathscr{C}(\omega^{(i)})}.
\end{equation}
\end{proposition}
\begin{proof}
First, applying Lemma \ref{lem:conv}, we have
$$
\Pi_{\omega^{(i)},a} \{ P,\chi\} = \sum_{\substack{b+c = a \\ |c|\leq 2 q' |\omega^{(i)}|_\infty }} \{ \Pi_{\omega^{(i)},b} P, \Pi_{\omega^{(i)},c} \chi \},
$$
the second condition coming from the fact that if $|c|>  2 q' |\omega^{(i)}|_\infty $ then $\Pi_{\omega^{(i)},c} \chi = 0$.
Then applying the triangle inequality and the bilinear estimate of Lemma \eqref{lem:est_poisson_weak}, it comes
\begin{equation*}
\begin{split}
  \| \lfloor \Pi_{\omega^{(i)},a} \{ P,\chi\} \rceil  \|_{\infty}&\leq \sum_{\substack{b+c = a \\ |c|\leq 2 q' |\omega^{(i)}|_\infty }}   \| \lfloor \{ \Pi_{\omega^{(i)},b} P, \Pi_{\omega^{(i)},c} \chi \} \rceil  \|_{\infty} \\
  &\leq 8 q q' \sum_{\substack{b+c = a \\ |c|\leq 2 q' |\omega^{(i)}|_\infty }}   \| \lfloor \Pi_{\omega^{(i)},b} P \rceil  \|_{\infty}   \| \lfloor \Pi_{\omega^{(i)},c} Q \rceil  \|_{\infty} \\
  &\leq 8qq' \| P \|_{\mathscr{H}(\omega^{(i)})}  \| \chi \|_{\mathscr{C}(\omega^{(i)})} \sum_{|c|\leq 2 q' |\omega^{(i)}|_\infty} \langle c \rangle^{-1}.
\end{split}
\end{equation*}
Estimating this last sum as in Lemma \ref{lem:comp_with_classical_norms}, we get the estimate we aimed at proving \eqref{eq:cest_tres_joli}.

\end{proof}

\subsection{Birkhoff normal form}

\begin{theorem} \label{thm:Birk} Let $H : \mathbb{C}^\mathcal{M} \to \mathbb{R}$ be a polynomial of the form
$$
H = Z_{2,\omega} +  P,
$$
where $\omega \in \mathbb{R}^\mathcal{M}$ is a vector of frequencies with real coefficients, $Z_{2,\omega} \in \mathcal{P}^{\mathcal{M}}_{\mathbb{R},2}$ is the quadratic diagonal polynomial given by \eqref{eq:defZ2}, and $P \in \mathcal{P}^{\mathcal{M}}_{\mathbb{R},2p}$ is a real valued homogeneous polynomial of degree $2p\geq 4$ commuting with the Euclidean norm. 

Let  $\omega^{(i)} \in \mathbb{Z}^\mathcal{M} \setminus \{0\}$ be a vector of frequencies with integer coefficients and $\omega^{(f)} \in \mathbb{R}^\mathcal{M}$  be the vector of frequencies with real coefficients such that
$$
\omega = \omega^{(i)} + \omega^{(f)}.
$$

For all $r \geq p-1$ and all $\gamma\in (0;1)$, setting 
\begin{equation}
\label{eq:defepsr}
\varepsilon_r :=  \left( \frac{\gamma}{ A\, B_p\, r^5 \langle |\omega^{(f)}|_\infty \rangle \| P \|_{\mathscr{H}(\omega^{(i)})}    \log \langle |\omega^{(i)}|_\infty \rangle  } \right)^{\frac1{2p-2}},
\end{equation}
where $A>1$ is a universal constant and $B_p>1$ depends only on $p$,
there exists a symplectomorphism $\tau : \mathbb{C}^\mathcal{M} \to \mathbb{C}^\mathcal{M}$ such that $H\circ \tau^{-1}$ is analytic on the ball $B(0,\varepsilon_r)$ with an analytic expansion of the form
\begin{equation}
\label{eq:the_dec}
H\circ \tau^{-1} = Z_{2,\omega} + \sum_{j\geq p} Q^{(2j)}
\end{equation}
where $Q^{(2j)} \in \mathcal{P}^{\mathcal{M}}_{\mathbb{R},2j}$ is a real-valued homogeneous polynomial of degree $2j$ commuting with the Euclidean norm such that:
\begin{itemize}
\item for $j\leq r$, $Q^{(2j)} $ is $\gamma$-resonant, i.e.
\begin{equation}
\label{eq:gamma-res}
| \omega_{\boldsymbol{k}_1} + \cdots + \omega_{\boldsymbol{k}_j} -\omega_{\boldsymbol{\ell}_1} - \cdots - \omega_{\boldsymbol{\ell}_j}   | \geq \gamma \quad \Rightarrow \quad Q^{(2j)}_{\boldsymbol{k},\boldsymbol{\ell}} = 0;
\end{equation}
\item for all $j \geq p$, $Q^{(2j)}$ enjoys the estimate
\begin{equation}
\label{eq:je_controle_mes_termes}
\|  Q^{(2j)} \|_{\mathscr{H}(\omega^{(i)})}  \leq \varepsilon_r^{-2(j-p)} \|  P \|_{\mathscr{H}(\omega^{(i)})} .
\end{equation}
\end{itemize} 
Moreover, the symplectomorphism $\tau$ enjoys the three following properties:
\begin{itemize}
\item it preserves the Euclidean norm, i.e. $\|\tau(u)\| = \|u\|$ for all $u\in \mathbb{C}^\mathcal{M}$;
\item it is close to the identity, i.e.
\begin{equation}
\label{eq:hehe_cest_proche_de_lidentite}
\| u\| \leq \varepsilon_r \quad \Rightarrow \quad  \| \tau(u) - u \| \leq  \Big( \frac{\|u\|}{\varepsilon_r} \Big)^{2p-2} \|u\|;
\end{equation}
\item its differential enjoys the following estimate
\begin{equation}
\label{eq:la_differentiel_bouge_pas_trop}
\| u\| \leq \varepsilon_r \quad \Rightarrow \quad  \forall v \in \mathbb{C}^\mathcal{M}, \quad  \| \mathrm{d}\tau(u)(v) \| \leq \exp\Big( \big(\frac{\|u\|}{\varepsilon_r} \big)^{2p-2} \Big) \|v\|.
\end{equation}
\end{itemize}
\end{theorem}
\begin{remark}
Note that the convergence of the entire series \eqref{eq:the_dec} on $B(0,\varepsilon_r)$ is ensured  by estimate \eqref{eq:je_controle_mes_termes} and Lemma \ref{lem:comp_with_classical_norms}.
\end{remark}
\begin{proof}[Proof of Theorem \ref{thm:Birk}] We proceed by induction on $r\geq p-1$. First, we note that the initialization is trivial. It is enough to choose $\tau = \mathrm{Id}$.  Now, we assume that the theorem holds at the step $r$ and we aim at proving it at the step $r+1$. The object we are going to design at the step $r+1$ will be identified by a subscript $\sharp$ (e.g. $\tau^\sharp$ will be the change of variable at the step $r+1$ while $\tau$ denotes the change of variables at the step $r$).

\noindent \underline{$\triangleright$ \emph{Step 1 : the new variables}.} Let $\chi \in \mathcal{P}^{\mathcal{M}}_{\mathbb{R},2r+2}$ be the polynomial defined by
$$
\chi_{\boldsymbol{k},\boldsymbol{\ell}} = \frac{Q^{(2r+2)}_{\boldsymbol{k},\boldsymbol{\ell}}}{i \Omega_{\omega} (\boldsymbol{k},\boldsymbol{\ell}) } \quad \mathrm{if} \quad |\Omega_{\omega} (\boldsymbol{k},\boldsymbol{\ell})|\geq \gamma \quad \quad  \mathrm{and} \quad  \quad \chi_{\boldsymbol{k},\boldsymbol{\ell}} = 0 \quad \mathrm{else,}
$$
where
$$
\Omega_{\omega} (\boldsymbol{k},\boldsymbol{\ell}):= \omega_{\boldsymbol{k}_1} + \cdots + \omega_{\boldsymbol{k}_{r+1}} -\omega_{\boldsymbol{\ell}_1} - \cdots - \omega_{\boldsymbol{\ell}_{r+1}}.
$$
Thanks to Lemma \ref{lem:cestdiagonal}, it is clear that
\begin{equation}
\label{eq:coho}
Q^{(2r+2),\sharp} := Q^{(2r+2)} + \{\chi, Z_{2,\omega}\} \quad \mathrm{is} \quad \gamma-\mathrm{resonant} \ (\text{see} \ \eqref{eq:gamma-res}),
\end{equation}
because it satisfies
\begin{equation}
\label{eq:maisquicest}
Q^{(2r+2),\sharp} _{\boldsymbol{k},\boldsymbol{\ell}} = 0 \quad \mathrm{if} \quad |\Omega_{\omega} (\boldsymbol{k},\boldsymbol{\ell})|\geq \gamma \quad \quad  \mathrm{and} \quad  \quad Q^{(2r+2),\sharp} _{\boldsymbol{k},\boldsymbol{\ell}} = Q^{(2r+2)}_{\boldsymbol{k},\boldsymbol{\ell}} \quad \mathrm{else}.
\end{equation}
According to Proposition \ref{prop:ca_flotte}, let $\Phi_\chi^t$ the Hamiltonian flow generated by $\chi$. We define the new change of variable by
$$
\tau^\sharp := \Phi_\chi^1 \circ \tau.
$$
Its properties will be studied in the last step of the proof. In order to have Poisson bracket estimates, for the moment, let us estimate $\| \chi \|_{\mathscr{C}(\omega^{(i)})}$.
First, we note that if $|\Omega_{\omega} (\boldsymbol{k},\boldsymbol{\ell})|\geq \gamma$ then
\begin{equation}
\label{eq:cest_sympa_comme_estime_non}
|\Omega_{\omega} (\boldsymbol{k},\boldsymbol{\ell})| \geq  \frac{\gamma}{8(r+1) \langle |\omega^{(f)}|_\infty \rangle} \,\langle \Omega_{\omega^{(i)}} (\boldsymbol{k},\boldsymbol{\ell}) \rangle.
\end{equation}
Indeed,
\begin{itemize}
\item either $\langle \Omega_{\omega^{(i)}} (\boldsymbol{k},\boldsymbol{\ell}) \rangle \leq 8(r+1) \langle |\omega^{(f)}|_\infty \rangle$ and so \eqref{eq:cest_sympa_comme_estime_non} is trivial
\item or $\langle \Omega_{\omega^{(i)}} (\boldsymbol{k},\boldsymbol{\ell}) \rangle > 8(r+1) \langle |\omega^{(f)}|_\infty \rangle$ and so\footnote{note that for the last estimate we have used the assumption $\gamma \leq 1$.}
\begin{align*}
|\Omega_{\omega} (\boldsymbol{k},\boldsymbol{\ell})|  &\geq |\Omega_{\omega^{(i)}} (\boldsymbol{k},\boldsymbol{\ell})| - |\Omega_{\omega^{(f)}} (\boldsymbol{k},\boldsymbol{\ell})| \geq \frac12 \langle \Omega_{\omega^{(i)}} (\boldsymbol{k},\boldsymbol{\ell}) \rangle - |\Omega_{\omega^{(f)}} (\boldsymbol{k},\boldsymbol{\ell})|  \\ &\geq \frac14 \langle \Omega_{\omega^{(i)}} (\boldsymbol{k},\boldsymbol{\ell}) \rangle + \frac14 \langle \Omega_{\omega^{(i)}} (\boldsymbol{k},\boldsymbol{\ell}) \rangle - 2(r+1) \langle |\omega^{(f)}|_\infty \rangle \\ &\geq \frac14 \langle \Omega_{\omega^{(i)}} (\boldsymbol{k},\boldsymbol{\ell}) \rangle \geq \frac{\gamma}{8(r+1) \langle |\omega^{(f)}|_\infty \rangle} \,\langle \Omega_{\omega^{(i)}} (\boldsymbol{k},\boldsymbol{\ell}) \rangle.
\end{align*}

\end{itemize}
Finally, as a consequence of \eqref{eq:cest_sympa_comme_estime_non} and the induction hypothesis \eqref{eq:je_controle_mes_termes}, $\chi$ enjoys the estimate
\begin{equation}
\label{eq:est_nablachi}
\begin{split}
\| \chi \|_{\mathscr{C}(\omega^{(i)})} &\leq  8(r+1)\gamma^{-1} \langle |\omega^{(f)}|_\infty \rangle  \| Q^{(2r+2)} \|_{\mathscr{H}(\omega^{(i)})} \\
&\leq 8(r+1)\gamma^{-1} \langle |\omega^{(f)}|_\infty \rangle \| P \|_{\mathscr{H}(\omega^{(i)})} \varepsilon_r^{-2(r+1-p)} .
\end{split}
\end{equation}
Therefore, defining 
\begin{equation}
\label{eq:def_etarp1}
\eta_{r+1}=  \Big(  \Big[ \frac{8(r+1) \langle |\omega^{(f)}|_\infty \rangle}{\gamma}  \Big]    \Big[   40  (r+1)  \log ( 2(r+1)|\omega^{(i)}|_\infty)\Big] \| P \|_{\mathscr{H}(\omega^{(i)})} \Big)^{-\frac1{2p-2}},
\end{equation}
as a consequence of Proposition \ref{prop:est_poisson}, we get that for all $q\geq 1$ and all $Q\in \mathcal{P}^{\mathcal{M}}_{\mathbb{R},2q}$,
$$
\| \{ \chi, Q\} \|_{\mathscr{H}(\omega^{(i)})} \leq  q  \varepsilon_r^{-2(r+1-p)} \eta_{r+1}^{-(2p-2)} \|Q\|_{ \mathscr{H}(\omega^{(i)}) }.
$$
Therefore, noticing that (provided that $A$ is chosen large enough)
\begin{equation}
\label{eq:cestbientotlesvacances}
\eta_{r+1}^{-(2p-2)} \leq B_p^{-1} r^{-2} \varepsilon_r^{-(2p-2)},
\end{equation}
we get
\begin{equation}
\label{eq:toutcapourca}
\| \{ \chi, Q\} \|_{\mathscr{H}(\omega^{(i)})} \leq \frac{q}{B_pr^2}  \varepsilon_r^{-2(r+1-p)} \varepsilon_{r}^{-(2p-2)} \|Q\|_{ \mathscr{H}(\omega^{(i)}) } =  \frac{q}{B_pr^2} \varepsilon_{r}^{-2r}  \|Q\|_{ \mathscr{H}(\omega^{(i)}) } .
\end{equation}

\noindent \underline{$\triangleright$ \emph{Step 2 : the new expansion (algebra)}.} We recall that by definition of $\Phi_\chi^{-t}$, if $K : \mathbb{C}^\mathcal{M} \to \mathbb{R}$ is a smooth function then for all $t\in \mathbb{R}$ and $u\in \mathbb{C}^\mathcal{M}$, we have
$$
\partial_t K(\Phi_\chi^{-t}(u)) = \{\chi,K\}(\Phi_\chi^{-t}(u)).
$$
Therefore, doing a Taylor expansion\footnote{at the order $N+1\geq 3$ for $Z_{2,\omega}\circ \Phi_\chi^{-t}$ and at the order $N$ for $(H\circ \tau^{-1} - Z_{2,\omega})\circ \Phi_\chi^{-t} $.}, we get
$$
H\circ (\tau^\sharp)^{-1}(u) = \sum_{n=0}^N \frac{\mathrm{ad}_\chi^n}{n !} (H\circ \tau^{-1})(u) +   \frac{\mathrm{ad}_\chi^{N+1}}{(N+1) !} Z_{2,\omega}(u)  + R^{(N)}(u) 
$$
with
\begin{equation*}
\begin{split}
R^{(N)}(u)  :=& \int_0^1 \frac{(1-t)^{N}}{N !}  \mathrm{ad}_\chi^{N+1} (H\circ \tau^{-1} - Z_{2,\omega})  ( \Phi_\chi^{-t}  (u)) \, \mathrm{d}t \\
&+ \int_0^1 \frac{(1-t)^{N+1}}{(N+1) !}  \mathrm{ad}_\chi^{N+2}  Z_{2,\omega}  ( \Phi_\chi^{-t}  (u)) \, \mathrm{d}t .
\end{split}
\end{equation*}
Recalling that the analytic expansion of $ H\circ \tau^{-1}$ is given by \eqref{eq:the_dec}, if $\|u\| < \varepsilon_r$, we have\footnote{since the series \eqref{eq:the_dec} is analytic, we can permute sums and derivatives (here Poisson brackets) inside the domain of convergence.}
$$
H\circ (\tau^\sharp)^{-1}(u) = \sum_{n=0}^{N+1} \frac{\mathrm{ad}_\chi^n}{n !} Z_{2,\omega}(u) + \sum_{j\geq p} \sum_{n=0}^N \frac{\mathrm{ad}_\chi^n}{n !} Q^{(2j)}(u)  + R^{(N)}(u) .
$$
Recalling that, by definition, $\chi$ solves the cohomological equation \eqref{eq:coho}, we have
\begin{equation*}
\begin{split}
H\circ (\tau^\sharp)^{-1}(u) &=Z_{2,\omega}(u) +   \sum_{n=1}^{N+1} \frac{\mathrm{ad}_\chi^{n-1}}{n !} (Q^{(2r+2),\sharp} - Q^{(2r+2)})(u) + \sum_{j\geq p} \sum_{n=0}^N \frac{\mathrm{ad}_\chi^n}{n !} Q^{(2j)}(u)  + R^{(N)}(u) \\
&= Z_{2,\omega}(u)  + \sum_{j\geq p} \sum_{n=0}^N \frac{\mathrm{ad}_\chi^n}{n !} Q^{(2j,n)}(u) + R^{(N)}(u) 
\end{split}
\end{equation*}
where $Q^{(2j,n)} \in \mathcal{P}^{\mathcal{M}}_{\mathbb{R},2r+2}$ is defined by
$$
 Q^{(2j,n)} = Q^{(2j)} \quad \mathrm{if} \quad j\neq r+1 \quad \mathrm{and} \quad  Q^{(2r+2,n)} = (1-\frac1{n+1}) Q^{(2r+2)} + \frac1{n+1} Q^{(2r+2),\sharp}.
$$
Note that, by definition of $Q^{(2r+2),\sharp}$ (see \eqref{eq:maisquicest}), it is clear that 
$$
\forall j\geq p, \forall n\geq 0, \quad \|  Q^{(2j,n)} \|_{\mathscr{H}(\omega^{(i)})} \leq \|  Q^{(2j)} \|_{\mathscr{H}(\omega^{(i)})}.
$$
Then, ordering the terms by degrees, we get
$$
H\circ (\tau^\sharp)^{-1}(u) = Z_{2,\omega}(u) + \sum_{j\geq p} K^{(2j,N)}(u)  + R^{(N)}(u) 
$$
where $K^{(2j,N)} \in \mathcal{P}^{\mathcal{M}}_{\mathbb{R},2j}$ is given by
\begin{equation}
\label{eq:defK2jN}
K^{(2j,N)} = \sum_{ \substack{j = nr + k\\  n\leq N}} \frac{\mathrm{ad}_\chi^n}{n !} Q^{(2k,n)} .
\end{equation}
Then, as usual, we also define its limit as
\begin{equation}
\label{eq:coucou}
Q^{(2j),\sharp} := K^{(2j,\infty)} = \sum_{ j = nr + k} \frac{\mathrm{ad}_\chi^n}{n !} Q^{(2k,n)}.
\end{equation}
Of course, it can be easily checked that this definition is consistent with \eqref{eq:coho} if $j=r+1$ and that
\begin{equation}
\label{eq:coucou2}
Q^{(2j),\sharp}  = Q^{(2j)} \quad \mathrm{if} \quad j\leq r.
\end{equation}
As a consequence, $Q^{(2j),\sharp}$ is $\gamma$-resonant for $j\leq r+1$ (see \eqref{eq:gamma-res}).

Then, assuming for one instant that the series $\sum Q^{(2j),\sharp}(u)$ converges if $\|u\| < \varepsilon_{r+1}$, we have proven that
$$
H\circ (\tau^\sharp)^{-1}(u) = Z_{2,\omega}(u) + \sum_{j\geq p} Q^{(2j),\sharp}(u) + \sum_{j\geq p} (K^{(2j,N)} - Q^{(2j),\sharp})(u) + R^{(N)}(u) .
$$
Therefore, it remains to prove that $\|  Q^{(2j),\sharp} \|_{\mathscr{H}(\omega^{(i)})} \leq \varepsilon_{r+1}^{-2(j-p)}$ (which will imply the convergence of the series) and that the last two go to zero as $N$ goes to $+\infty$.

\noindent \underline{$\triangleright$ \emph{Step 3 : control of $\|  Q^{(2j),\sharp} \|_{\mathscr{H}(\omega^{(i)})}$.}} First, we note that thanks to the relations \eqref{eq:maisquicest} and \eqref{eq:coucou2}, since $\varepsilon_r > \varepsilon_{r+1}$, it is enough to estimate $\|  Q^{(2j),\sharp} \|_{\mathscr{H}(\omega^{(i)})}$ when $j\geq r+2$.

Recalling that $Q^{(2j),\sharp}$ is given by \eqref{eq:coucou}, by the triangle inequality, we have
\begin{equation}
\label{eq:jy_reviendrai_a_la_prochaine_etape}
\| Q^{(2j),\sharp}  \|_{\mathscr{H}(\omega^{(i)})} \leq  \sum_{ j = nr + k} \| \frac{\mathrm{ad}_\chi^n}{n !} Q^{(2k,n)} \|_{\mathscr{H}(\omega^{(i)})} .
\end{equation}
Applying then the estimate \eqref{eq:toutcapourca} and the induction hypothesis \eqref{eq:je_controle_mes_termes}, we get
\begin{equation*}
\begin{split}
 \| Q^{(2j),\sharp}  \|_{\mathscr{H}(\omega^{(i)})} &\leq    \sum_{ j = nr + k} \frac{k (k+r) \cdots (j-r)}{n !}  (B_pr^2)^{-n}  \varepsilon_{r}^{-2 j}  \|P\|_{ \mathscr{H}(\omega^{(i)}) } \\
  & \leq  \sum_{ j = nr + k} \frac{j^n}{n !} (B_pr)^{-n}  \left(\frac{r}{r+1}\right)^{2j \frac{5}{2p-2}}  \varepsilon_{r+1}^{-2 j}  \|P\|_{ \mathscr{H}(\omega^{(i)}) }.
 \end{split}
 \end{equation*}
 Then, using the estimate $C^{-j} j^n \leq e^{-n} n^n (\log(C))^{-n} $, $C>1$, we get
 $$
 \frac{ \| Q^{(2j),\sharp}  \|_{\mathscr{H}(\omega^{(i)})} }{\varepsilon_{r+1}^{-2 j}  \|P\|_{ \mathscr{H}(\omega^{(i)}) }}\leq \left(\frac{r}{r+1}\right)^{ \frac{5j}{p-1}}  +\sum_{ n\geq 1}  \frac{n^n (eB_pr)^{-n}}{n!} \left(\log \Big(1+\frac1r\Big) \right)^{-n} \left( \frac{p-1}{5} \right)^{n} .
 $$
 Since $r\geq p\geq 1$, $j\geq r+2$ and $n\geq 1$, we use the three following useful estimates:
 $$
 \left(\frac{r}{r+1}\right)^{ j} \leq \left(\frac{r}{r+1}\right)^{ r+1} \leq e^{-1}, \quad \quad n^n e^{-n} \leq n!,  \quad \quad \log \Big(1+\frac1r\Big) \geq \frac{\log(2)}{r},
 $$
to get
$$
 \frac{ \| Q^{(2j),\sharp}  \|_{\mathscr{H}(\omega^{(i)})} }{\varepsilon_{r+1}^{-2 j}  \|P\|_{ \mathscr{H}(\omega^{(i)}) }}\leq e^{- \frac{5}{p-1}} + \sum_{ n\geq 1}  B_p^{-n} (\log(2))^{-n} \left( \frac{p-1}{5} \right)^{n} \mathop{\longrightarrow}_{B_p \to +\infty} e^{- \frac{5}{p-1}} <1,
$$
and so  $\| Q^{(2j),\sharp}  \|_{\mathscr{H}(\omega^{(i)})} \leq \varepsilon_{r+1}^{-2 j}  \|P\|_{ \mathscr{H}(\omega^{(i)}) }$ provided that $B_p$ is chosen large enough.

\noindent \underline{$\triangleright$ \emph{Step 4 : limit $N\to +\infty$.}} First, we note that by definition of $K^{(2j,N)}$ (see \eqref{eq:defK2jN}) if $j\leq Nr$ then $K^{(2j,N)} = Q^{(2j),\sharp}$. As a consequence, by Lemma \ref{lem:comp_with_classical_norms}, if $\|u\| < \varepsilon_{r+1}$, we have 
$$
| \sum_{j\geq p} (K^{(2j,N)} - Q^{(2j),\sharp})(u) | \leq 5 |\omega^{(i)}|_\infty \sum_{j\geq rN} j \| K^{(2j,N)} - Q^{(2j),\sharp} \|_{\mathscr{H}(\omega^{(i)})} \| u\|^{2j}.
$$
But applying the triangle inequality, we have
$$
\| K^{(2j,N)} - Q^{(2j),\sharp} \|_{\mathscr{H}(\omega^{(i)})} \leq  \sum_{ j = nr + k} \| \frac{\mathrm{ad}_\chi^n}{n !} Q^{(2k,n)} \|_{\mathscr{H}(\omega^{(i)})}, 
$$
and so thanks to the estimate proved at the previous step (see \eqref{eq:jy_reviendrai_a_la_prochaine_etape}), we get
$$
| \sum_{j\geq p} (K^{(2j,N)} - Q^{(2j),\sharp})(u) | \leq 5 |\omega^{(i)}|_\infty \|P\|_{ \mathscr{H}(\omega^{(i)}) } \sum_{j\geq rN} j \varepsilon_{r+1}^{-2(j-p)} \| u\|^{2j} \mathop{\longrightarrow}_{N \to +\infty} 0. 
$$

Now, it only remains to prove that $R^{(N)}(u)$ goes to $0$ as $N$ goes to $+\infty$. First, we note that using as previously the cohomological equation, the remainder term rewrites
$$
R^{(N)}(u) = \int_0^1 \frac{(1-t)^{N}}{N !} \sum_{j\geq p}  \mathrm{ad}_\chi^{N+1}Q^{(2j,N+1)}_t  ( \Phi_\chi^{-t}  (u)) \, \mathrm{d}t,
$$
where $Q^{(2j,n)}_t  \in \mathcal{P}^{\mathcal{M}}_{\mathbb{R},2r+2}$ is defined by
$$
 Q^{(2j,n)}_t = Q^{(2j)} \quad \mathrm{if} \quad j\neq r+1, \quad \mathrm{and} \quad  Q^{(2r+2,n)}_t = (1-\frac{1-t}{n}) Q^{(2r+2)} + \frac{1-t}{n} Q^{(2r+2),\sharp}.
$$
Note that as previously, we have $\|  Q^{(2j,n)}_t \|_{\mathscr{H}(\omega^{(i)})} \leq \|  Q^{(2j)} \|_{\mathscr{H}(\omega^{(i)})}$.
Then, since $\Phi_\chi^{-t}$ preserves the Euclidean norm, we have
\begin{equation*}
\begin{split}
|R^{(N)}(u)| &\leq \sum_{k\geq p} \sup_{0\leq t \leq 1}  \| \frac{\mathrm{ad}_\chi^{N+1}}{(N+1) !} Q^{(2k,N+1)}_t\|_\infty  \| u\|^{2(k+(N+1)r)} \\
&= \sum_{j\geq (N+1)r+p} \sup_{0\leq t \leq 1}  \| \frac{\mathrm{ad}_\chi^{N+1}}{(N+1) !} Q^{(2 (j-(N+1)r), N+1)}_t\|_\infty  \| u\|^{2j} \\
&\leq  \sum_{j\geq (N+1)r+p}  \sum_{ j = nr + k}  \sup_{0\leq t \leq 1}  \| \frac{\mathrm{ad}_\chi^n}{n !} Q^{(2k,n)}_t \|_{\infty}  \| u\|^{2j}.
\end{split}
\end{equation*}
Then applying Lemma \ref{lem:comp_with_classical_norms} to control the $\| \cdot \|_\infty$ norm and proceeding as we did for the other remainder term\footnote{note that we have the same estimates as for the term in \eqref{eq:jy_reviendrai_a_la_prochaine_etape} (the bound are uniform with respect to $t\in [0;1]$).}, we deduce that  $R^{(N)}(u)$ goes to $0$ as $N$ goes to $+\infty$.

\noindent \underline{$\triangleright$ \emph{Step 5 : properties of $\tau^\sharp$}.} First, we note that, by composition, it is clear that $\tau^\sharp$ is a symplectomorphism which preserves the Euclidean norm. 

Now, recalling the estimate \eqref{eq:est_nablachi} of $\| \chi \|_{\mathscr{C}(\omega^{(i)})}$ and the definition \eqref{eq:def_etarp1} of $\eta_{r+1}$ and applying the last estimate of Lemma \ref{lem:comp_with_classical_norms}, we get
$$
2(r+1) \| \chi \|_{\infty} \leq \eta_{r+1}^{-(2p-2)} \varepsilon_r^{-2(r+1-p)}.
$$
Therefore, as previously, provided that $A>1$ is large enough (see \eqref{eq:cestbientotlesvacances}), we have
\begin{equation}
\label{eq:hihi_ler2}
2(r+1)\| \chi \|_{\infty} \leq B_p^{-1} r^{-2} \varepsilon_{r}^{-2r}.
\end{equation}
As a consequence, by Proposition \ref{prop:ca_flotte}, we have
$$
\| \Phi_\chi^{-1} (u) - u  \| \leq B_p^{-1} r^{-2} \Big( \frac{\|u\|}{\varepsilon_r} \Big)^{2r} \|u\|.
$$
Since $\Phi_\chi^{-1}$ preserves the Euclidean norm, applying the triangle inequality and the induction hypothesis \eqref{eq:hehe_cest_proche_de_lidentite}, provided that $\|u\|\leq \varepsilon_r$, we get
\begin{equation*}
\begin{split}
\|\tau^{\sharp}(u) - u \| &\leq \| \tau(\Phi_\chi^{-1}(u)) - \Phi_\chi^{-1}(u)  \| + \| \Phi_\chi^{-1} (u) - u  \| \\
&\leq \Big( \frac{\|u\|}{\varepsilon_r} \Big)^{2p-2} \|u\| + B_p^{-1} r^{-2} \Big( \frac{\|u\|}{\varepsilon_r} \Big)^{2r} \|u\| \\
&\leq \big(\frac{r}{r+1}\big)^5 (1+ B_p^{-1} r^{-2}) \Big( \frac{\|u\|}{\varepsilon_{r+1}} \Big)^{2p-2} \|u\|.
\end{split}
\end{equation*}
Since, provided that $B_p>1$ is large enough, we have $\big(\frac{r}{r+1}\big)^5 (1+ B_p^{-1} r^{-2}) \leq 1$, we deduce that $\|\tau^{\sharp}(u) - u \|$ is close to the identity.

Finally, it only remains to control $\mathrm{d}\tau(u)$. Applying the estimate \eqref{eq:cequonveut} of Proposition \ref{prop:ca_flotte}, for all $v\in \mathbb{C}^\mathcal{M}$, we have
$$
\| \mathrm{d}\Phi_\chi^{-1}(u)(v) \| \leq \exp \big( 4 r^2  \| \chi \|_\infty \|u \|^{2r} \big) \|v\| \mathop{\leq}^{\eqref{eq:hihi_ler2}} \exp\Big( \frac{2}{B_p r}  \big(\frac{\| u\|}{\varepsilon_r} \big)^{2r} \Big) \|v\|.
$$
Therefore, thanks to the induction hypothesis \eqref{eq:la_differentiel_bouge_pas_trop}, if $\|u\|\leq \varepsilon_r$, we get
\begin{equation*}
\begin{split}
\| \mathrm{d}\tau^{-1}(u)(v) \| &\leq \exp\Big( \big(\frac{\| u\|}{\varepsilon_r} \big)^{2p-2} + \frac{2}{B_p r}  \big(\frac{\| u\|}{\varepsilon_r} \big)^{2r} \Big) \|v\| \\
&\leq \exp\Big( \big(\frac{r}{r+1}\big)^5 (1+ 2B_p^{-1} r^{-1})   \big(\frac{\| u\|}{\varepsilon_{r+1}} \big)^{2p-2}   \Big) \|v\|.
\end{split}
\end{equation*}
Again, provided that $B_p>1$ is large enough, we have $\big(\frac{r}{r+1}\big)^5 (1+ 2B_p^{-1} r^{-1}) \leq 1$ for all $r\geq p$, so we get the expected estimate on $\| \mathrm{d}\tau^{-1}(u)(v) \|$.
\end{proof}

\subsection{Dynamical corollary}

\begin{definition}[$(k,r,\gamma)$ non-resonance] \label{def:krgam_nr} Being given $k\in \mathcal{M}$, $r\geq 1$ and $\gamma >0$, a vector of frequencies $\omega \in \mathbb{R}^\mathcal{M}$ is \emph{$(k,r,\gamma)$ non-resonant} if for all $q\leq r$ and all $Q \in \mathcal{P}^{\mathcal{M}}_{\mathbb{R},2q}$,
\begin{center}
if $Q$ is $\gamma$-resonant (see \eqref{eq:gamma-res}) then $Q$ commutes with $|u_k|^2$ (i.e. $\{Q,|u_k|^2\}=0$).
\end{center}
\end{definition}

\begin{remark}
The previous definition is the notion of non-resonance that we actually need for our dynamical corollary; see Corollary~\ref{cor:dyn} below. In the proof of Theorems~\ref{thm:conv} and~\ref{thm:mult}, we will use that strong non-resonance according to Definition~\ref{def:strg_nr} implies the non-resonance condition of Definition~\ref{def:krgam_nr} for all the ``low modes'' $k$ satisfying $\langle k\rangle \lesssim \epsilon^{-\nu}$ for some exponent $\nu>0$.
\end{remark}

\begin{corollary} \label{cor:dyn} In the setting of the result of Theorem \ref{thm:Birk}, if $k \in \mathcal{M}$ is an index such that  the frequencies $\omega$ are $(k,r,\gamma)$ non-resonant and $u \in C^1([0;T];\mathbb{C}^\mathcal{M})$ is the solution of an equation of the form
\begin{equation}
\label{eq:decomp_cor}
i \partial_t u(t) = \nabla H (u(t))+ g(t),
\end{equation}
with $g\in C^0([0;T];\mathbb{C}^\mathcal{M})$, and if it satisfies the bound
$$
\delta := \| u\|_{L^\infty(0;T)} := \sup_{0\leq t \leq T} \| u(t) \| < \frac{\varepsilon_r}2,
$$
then we have 
$$
||u_k(T)|^2 - |u_k(0)|^2| \lesssim \varepsilon_r^{-(2p-2)} \delta^{2p} +   \|P\|_{\mathscr{H}(\omega^{(i)})}  |\omega^{(i)}|_\infty \delta^{2r+2} \varepsilon_r^{-2r} T + \delta  \|g\|_{L^\infty(0;T)} T.
$$
\end{corollary}
\begin{remark} The first error term, $\varepsilon_r^{-(2p-2)} \delta^{2p}$, is due to the difference between the change of variable and the identity. The second error term,  $\|P\|_{\mathscr{H}(\omega^{(i)})}  |\omega^{(i)}|_\infty \delta^{2r+2} \varepsilon_r^{-2r}$, controls the growth of the remainder term in the new variables (that is why it is of high order $\delta^{2r+2} $). Finally, the third error term $\delta  \|g\|_{L^\infty(0;T)} T$ is just due to the presence of the source term $g$.
\end{remark}
\begin{proof}[Proof of Corollary \ref{cor:dyn}]
Let $v = \tau(u)$. By composition, since $\tau$ is a symplectomorphism,  we have
$$
i\partial_t v = i\mathrm{d}\tau(u) ( -i \nabla H(u) -i g(t) ) = \nabla (H\circ \tau^{-1})(v) - i \mathrm{d}\tau(u)(i g)
$$
and thus
$$
\partial_t |v_k|^2 = \{ H\circ \tau^{-1} , |v_k|^2\} - 2 \Re (\overline{v_k}  [i \mathrm{d}\tau(u)(i g(t))]_k).
$$
Now, note that since $\tau$ preserves the Euclidean norm, we also have 
$$
\|v\|_{L^\infty(0;T)} = \delta < \frac{\varepsilon_r}2.
$$
Therefore, on the one hand, thanks to the estimate \eqref{eq:la_differentiel_bouge_pas_trop} on  $\mathrm{d}\tau(u)$, we have
$$
| \Re (\overline{v_k}  [i \mathrm{d}\tau(u)(i g)]_k) | \leq \| v\| \|  \mathrm{d}\tau(u)(i g)\| \leq \delta e \| g\|.
$$
 On the other hand, since the frequencies $\omega$ are $(k,r,\gamma)$ non-resonant and $ Q^{(2j)}$ is $\gamma$-resonant for $j\leq r$ (see \eqref{eq:gamma-res}), we have
$$
| \{ H\circ \tau^{-1} , |v_k|^2\} |  =  \Big| \sum_{j> r}  \{ Q^{(2j)}, |v_k|^2\} \Big| \leq  2\| v\| \sum_{j> r} \| \nabla  Q^{(2j)}(v) \| .
$$
Applying then Corollary \ref{cor:very_nice} and Lemma \ref{lem:comp_with_classical_norms}, we get
\begin{equation*}
\begin{split}
| \{ H\circ \tau^{-1} , |v_k|^2\} |   &\leq  2 \sum_{j> r} 2j \|v\|^{2j} \| Q^{(2j)}\|_{\infty} \leq  \sum_{j> r} 2j \|v\|^{2j} \| Q^{(2j)}\|_{\infty} \\
 &\leq 2\sum_{j> r} 2j \|v\|^{2j} 5j |\omega^{(i)}|_\infty \| Q \|_{\mathscr{H}(\omega^{(i)})}   \intertext{Next, using \eqref{eq:je_controle_mes_termes}, we can continue with}
 &\leq 20 |\omega^{(i)}|_\infty \| P \|_{\mathscr{H}(\omega^{(i)})} \sum_{j> r} j^2 \delta^{2j} \varepsilon_r^{-2(j-p)}  \\
 &\leq  20 |\omega^{(i)}|_\infty \| P \|_{\mathscr{H}(\omega^{(i)})} \delta^{2r+2} \varepsilon_r^{2(r+1-p)} \sum_{j = r+1}^\infty j^2  2^{-j}.
\end{split}
\end{equation*}
At the end, we have proven that
$$
||v_k(T)|^2 - |v_k(0)|^2| \lesssim    \|P\|_{\mathscr{H}(\omega^{(i)})}  |\omega^{(i)}|_\infty \delta^{2r+2} \varepsilon_r^{-2r} T + \delta \|g\|_{L^1(0;T)}.
$$
Finally, to conclude, it is enough to note that, since $\tau$ is close to the identity (see \eqref{eq:hehe_cest_proche_de_lidentite}), we have
$$
||u_k|^2 - |v_k|^2| \leq  \| u - v\| (\|u\| + \|v\|) = 2\| \tau(u) - u\| \delta \leq 2 \Big( \frac{\delta}{\varepsilon_r} \Big)^{2p-2} \delta^2.
$$
\end{proof}

\section{Proof of Theorem \ref{thm:conv}}
\label{sec:proof_thm_conv} The proof of Theorem \ref{thm:conv} relies on Corollary \ref{cor:dyn} of our Birkhoff normal form (Theorem \ref{thm:Birk}). Therefore, we are going to provide a decomposition of \eqref{eq:NLS*} of the form \eqref{eq:decomp_cor} and to estimate carefully each of its terms. Finally, Theorem \ref{thm:conv} will follow from a careful optimization of all the parameters involved.

In all the proof we consider a potential $V \in \mathcal{F}L^\infty(\mathbb{T};\mathbb{C})$, whose Fourier coefficients are real numbers such that the frequencies
$$
 \omega_j = j^2 + (2\pi)^{\frac12} V_j
$$
 are strongly-non-resonant according to Definition \ref{def:strg_nr} (and so we get an exponent $\alpha >0$).

Now, thanks to Proposition \ref{thm:gwp}, we consider a global solution $u\in C^0(\mathbb{R}; H^s(\mathbb{T};\mathbb{C}))$, $s>2/5$, of  \eqref{eq:NLS*} such that 
$$
\| u(0) \|_{H^s} =: \varepsilon \leq \frac1{100}.
$$
Since the $L^2$ norm is a constant of the motion for \eqref{eq:NLS*}, we have
\begin{equation}
\label{eq:pres_L2*}
\forall t\in \mathbb{R}, \quad   \|u(t) \|_{L^2} = \|u(0) \|_{L^2} \leq \|u(0) \|_{H^s} = \varepsilon.
\end{equation}
Moreover thanks to Proposition \ref{thm:gwp}, we know that there exists $\beta_s\geq 1$ such that
\begin{equation}
\label{eq:grow_hs*}
\forall t\in \mathbb{R}, \quad \| u(t) \|_{H^s} \lesssim_s \varepsilon \langle t  \rangle^{\beta_s}.
\end{equation}

\subsection{Control of the high actions}\label{sec:high}
\label{subsec:high} When $k$ is large enough, the normal form theorem is not well suited to prove Theorem \ref{thm:conv}. Nevertheless, in this case the time of stability is not too long and Theorem \ref{thm:conv} is just a consequence of the local well-posedness of \eqref{eq:NLS*}. For simplicity (to avoid the use of Bourgain spaces), here we propose a simple proof of this point, relying only on the estimate \eqref{eq:grow_hs*}.

Let $k\in \mathbb{Z}$ be such that
$$
2\langle k \rangle \geq \varepsilon^{-\upsilon_{\alpha,s,\nu}} 
$$
where $\upsilon_{\alpha,s,\nu} \in(0;1)$ is a positive constant depending only on $\alpha,\nu$ and $s$ that will be optimized later. Therefore, by assumption, we have
$$
\frac{\log \varepsilon^{-1}}{\log ( 2\langle k \rangle )} \leq \upsilon_{\alpha,s,\nu}^{-1}.
$$
Consequently, it is enough to prove that there exists a constant $\mu$ depending on $s,\nu$ and $\alpha$ such that, provided that $\varepsilon$ is smaller than a constant $\varepsilon_0$ depending only on $V,s,\nu$, we have
\begin{equation}
\label{eq:what_we_want_triv*}
|t|\leq \varepsilon^{\mu \log(\upsilon_{\alpha,s,\nu}^{-1})} \quad\Rightarrow \quad ||u_k(t)|^2 - |u_k(0)|^2| \leq \varepsilon^{6-\nu}.
\end{equation}

To prove such a property, setting $\mathcal{L} u := -\partial_x^2 u + V\ast u$, we control the variation of the actions uniformly with respect to $k$ as follows:
\begin{equation*}
\begin{split}
||u_k(t)|^2 - |u_k(0)|^2| &= ||u_k(t)|^2 - |e^{-it \omega_k}u_k(0)|^2| = ||u_k(t)|^2 - |(e^{-it \mathcal{L}}u^{(0)})_k|^2| \\
&=  (|u_k(t)| + |u_k(0)|)\big||u_k(t)| - |(e^{-it \mathcal{L}}u^{(0)})_k|\big| \\
&\leq (\| u(t)\|_{L^2} + \| u(0)\|_{L^2})| (u(t) -e^{-it \mathcal{L} }u^{(0)})_k|\\
&\leq 2\| u\|_{L^\infty_t L^2_x} \| u(t) -e^{-it \mathcal{L} } u^{(0)}\|_{L^2}
\end{split}
\end{equation*}
On the one hand, using the preservation of the $L^2$ norm (see \eqref{eq:pres_L2*}), we have
$
 \| u\|_{L^\infty_t L^2_x} = \varepsilon.
$
On the other hand, using the Duhamel formula, we have
\begin{equation*}
\begin{split}
\| u(t) -e^{-it \mathcal{L} } u^{(0)} \|_{L^2} &= \Big\| \int_0^t e^{-i (t-\tau) \mathcal{L} } |u(\tau)|^4 u(\tau) \mathrm{d}\tau \Big\|_{L^2} \\
&\leq \int_{[0;t]} \| |u(\tau)|^4 u(\tau) \|_{L^2} \mathrm{d}\tau \leq |t| \| u\|_{L^\infty([0;t] ; L^{10}_x)}^5.
\end{split}
\end{equation*}
Then, since $s>2/5$, using the Sobolev embedding $H^{2/5} \subset L^{10}$ and the a priori estimate \eqref{eq:grow_hs*} on the growth of the $H^s$ norm, we get
$$
\| u(t) -e^{-it \mathcal{L} } u^{(0)}\|_{L^2} \lesssim_s \varepsilon^5 \langle t  \rangle^{1+5\beta_s}.
$$

Finally, plugging these estimates together, we have proven that 
$$
||u_k(t)|^2 - |u_k(0)|^2| \lesssim_s \varepsilon^6 \langle t  \rangle^{1+5\beta_s}.
$$
which implies, as we wanted, the estimate \eqref{eq:what_we_want_triv*}, provided that
$$
(1+ 5 \beta_s) \mu \log(\upsilon_{\alpha,s,\nu}^{-1}) \leq \frac{\nu}2
$$
(i.e. that $\mu$ is small enough) and $\varepsilon$ is smaller than a constant depending only on $s$ and $\nu$.

\subsection{Setting for the low actions} \label{sec:low}
Now, and until the end of this proof, we aim at controlling the variations of $|u_k(t)|^2$ when $2\langle k \rangle < \varepsilon^{-\upsilon_{\alpha,s,\nu}}$. We consider a  large parameter $M\geq \varepsilon^{-\upsilon_{\alpha,s,\nu}}$ that will be optimized later, and we define
$$
 \mathcal{E}_M:= \mathrm{Span}_\mathbb{C} \{e^{ijx} \ | \ |j|\leq M\}.
$$
As usual, we identify $\mathcal{E}_M$ with $\mathbb{C}^\mathcal{M}$ (through the Fourier transform) where
$$
\mathcal{M}= \{-M,-M+1,\dots,M\}.
$$
We denote by $\Pi^{(M)}$ the $L^2$-orthogonal projection on $\mathcal{E}_M$ and we set 
$$
u^{(\leq M)} :=  \Pi^{(M)}(u(t)).
$$
Note that, $L^2$ norm being a constant of the motion and $\Pi^{(M)}$ being an orthogonal projection, we have
\begin{equation}
\label{eq:pres_L2_M*}
\forall t\in \mathbb{R}, \quad \|u^{(\leq M)}(t) \| \leq  \|u(t) \|_{L^2} = \|u(0) \|_{L^2} \leq \|u(0) \|_{H^s} = \varepsilon  .
\end{equation}
Moreover, $u^{(\leq M)}$ solves the equation
\begin{equation}
\label{eq:defgt*}
i\partial_t u^{(\leq M)} = \nabla H (u^{(\leq M)}) + g(t) \quad \mathrm{with} \quad g(t):= \sigma \Pi^{(M)} \big[  |u(t)|^4u(t) -   |\Pi^{(M)} u(t)|^4 \Pi^{(M)} u(t) \big]
\end{equation}
where
$$
H = Z_{2,\omega} + P \quad \mathrm{with} \quad P = \frac1{12} (\| \cdot \|_{L^6}^6)_{| \mathcal{E}_M} \quad \mathrm{and} \quad Z_{2,\omega} \mathrm{\ is \ given \ by \ \eqref{eq:defZ2}}.
$$
\subsection{Strichartz estimates}\label{sec:strichartz} Now we aim at estimating $\| P \|_{\mathscr{H}(\omega^{(i)})}$ where $\omega^{(i)}_j :=j^2$. Let $a\in \mathbb{Z}$, by definition, we have
$$
\forall u \in \mathcal{E}_M, \quad \lfloor \Pi_{\omega^{(i)},a} P \rceil(u) =  \frac1{12} \sum_{ \substack{\boldsymbol{k}_1 + \boldsymbol{k}_2 + \boldsymbol{k}_3  -\boldsymbol{\ell}_1 - \boldsymbol{\ell}_2 - \boldsymbol{\ell}_3 = 0 \\ \boldsymbol{k}_1^2 + \boldsymbol{k}_2^2 + \boldsymbol{k}_3^2  -\boldsymbol{\ell}_1^2 - \boldsymbol{\ell}_2^2 - \boldsymbol{\ell}_3^2 = a }  }  u_{\boldsymbol{k}_1}u_{\boldsymbol{k}_2} u_{\boldsymbol{k}_3} \ \overline{u_{\boldsymbol{\ell}_1}} \overline{u_{\boldsymbol{\ell}_2}} \overline{u_{\boldsymbol{\ell}_3}}.
$$
As a consequence, following a remark of Bourgain \cite[eq. (7.20)]{Bou04b}, we have\footnote{Note that the ``time'' $\tau$ in \eqref{StrichartzPolynom} has nothing to do with the ``actual'' time $t$ of the evolution equation \eqref{eq:NLS*}. In particular we use \eqref{StrichartzPolynom} with $u=\Pi^{(M)}u(t)$ for any fixed $t\in \R$.}
\begin{equation}\label{StrichartzPolynom}
 \lfloor \Pi_{\omega^{(i)},a} P \rceil(u)  = \frac1{24\pi} \int_{\mathbb{T}} e^{i\tau a}\| e^{i \tau\partial_x^2} u \|_{L^6}^6 \mathrm{d}\tau
\end{equation}
and so, we have
$$
| \lfloor \Pi_{\omega^{(i)},a} P \rceil(u) | \leq \frac1{24\pi} \|  e^{i \tau\partial_x^2} u \|_{L^6(\mathbb{T}_x \times \mathbb{T}_\tau)}^6.
$$
Then, applying the Strichartz estimate \eqref{strichartz} to control this $L^6$ space-time norm, we get
\begin{equation}
\label{eq:Str_NLS*}
\| P \|_{\mathscr{H}(\omega^{(i)})} = \sup_{a\in \mathbb{Z}} \sup_{\substack{\| u\|_{L^2}\leq 1\\ u\in \mathcal{E}_M }}| \lfloor \Pi_{\omega^{(i)},a} P \rceil(u) | \lesssim e^{c \frac{\log M}{\log \log M}},
\end{equation}
where $c>0$ is a universal constant.

\subsection{Estimate of the remainder term}\label{sec:trunc}
\label{sub:estg*}
By definition of the remainder term $g(t)$ (see\eqref{eq:defgt*}) and the mean value inequality, we have
$$
| g | \leq 5 ( | u|^4 + | \Pi^{(M)} u|^4 )  | u - \Pi^{(M)} u |,
$$
and so, by H\"older, we have
$$
\| g(t) \|_{L^2} \lesssim (\| u(t)\|_{L^{10}}^4 +  \| \Pi^{(M)} u(t)\|_{L^{10}}^4 ) \| u(t) - \Pi^{(M)} u(t) \|_{L^{10}}.
$$
Since by assumption $s>2/5$ and since the Sobolev embedding $H^{2/5} \subset L^{10}$ holds, we have 
\begin{equation}
\label{eq:jy_reviendrai}
\| g(t) \|_{L^2} \lesssim \| u(t)\|_{H^s}^4 \| u(t) - \Pi^{(M)} u(t) \|_{H^{2/5}}.
\end{equation}
Therefore, by definition of $\Pi^{(M)}$, we deduce that
$$
\| g(t) \|_{L^2} \lesssim \| u(t)\|_{H^s}^5 M^{-(s-\frac25)}.
$$
Finally, using the a priori bound we proved on $\| u(t)\|_{H^s}$, we get
\begin{equation}
\label{eq:est_g_NLS*}
\| g(t) \|_{L^2} \lesssim  \varepsilon^5 \langle t\rangle^{5 \beta_s} M^{-(s-\frac25)}.
\end{equation}

\subsection{Optimization of the parameters} 
\label{sub:opt*}
~\\
\noindent \underline{$\triangleright$ \emph{Step 1 : Setting.}} Now we aim at controlling the variations of $|u_k|$ where $k\in \mathbb{Z}$ satisfies
$$
 2\langle k \rangle < \varepsilon^{-\upsilon_{\alpha,s,\nu}}.
$$ 
We recall that we dealt with the case $2\langle k \rangle \geq \varepsilon^{-\upsilon_{\alpha,s,\nu}}$ at the beginning of the proof. Note that since $M$ has to satisfy $M \geq \varepsilon^{-\upsilon_{\alpha,s,\nu}}$, this implies that $|k|\leq M$ (i.e. $k\in \mathcal{M}$).

We introduce an integer $r\geq p=3$ that will be optimized later, and we set 
$$
\gamma := \rho \big( 2  \langle k \rangle\big)^{-e^{\alpha r} },
$$
in such a way that, since, by assumption, the frequencies $\omega$ are strongly-non-resonant according to Definition \ref{def:strg_nr},  $\omega$ is $(k,r,\gamma)$-non-resonant according to Definition \ref{def:krgam_nr}.

Therefore, applying Corollary \ref{cor:dyn} and using the preservation of the $L^2$ norm (see \eqref{eq:pres_L2*}), we know that if $\varepsilon \leq \frac{\varepsilon_r}2$ then for all $t\in \mathbb{R}$, we have
$$
||u_k(t)|^2 - |u_k(0)|^2| \lesssim \varepsilon_r^{-4} \varepsilon^{6} +   \|P\|_{\mathscr{H}(\omega^{(i)})}  |\omega^{(i)}|_\infty \varepsilon^{2r+2} \varepsilon_r^{-2r} |t| + \varepsilon  \|g\|_{L^\infty(0;t)L^2_x} |t|.
$$

First, we aim at establishing a simple lower bound on $\varepsilon_r$. Indeed, using the bound \eqref{eq:Str_NLS*} on $ \| P \|_{\mathscr{H}(\omega^{(i)})} $ (with $|\omega^{(i)}|_\infty \lesssim M^2$ and $ |\omega^{(f)}|_\infty \lesssim 1$), we have
$$
 \varepsilon_r = \left( \frac{\gamma}{ A\, B_3\, r^5 \langle |\omega^{(f)}|_\infty \rangle \| P \|_{\mathscr{H}(\omega^{(i)})}    \log \langle |\omega^{(i)}|_\infty \rangle  } \right)^{\frac1{4}} \gtrsim  \left( \frac{ e^{-c \frac{\log M}{\log \log M}} }{  r^5     \log M  }  \big( 2  \langle k \rangle\big)^{-e^{\alpha r} } \right)^{\frac1{4}}.  
 $$
 Therefore, there exists a constant $\kappa\in (0;1)$, depending only on $V$, such that we have
 $$
 2 \varepsilon_r \geq \kappa e^{-\frac{c}2 \frac{\log M}{\log \log M}} \big( 2  \langle k \rangle\big)^{-\frac12e^{\alpha r} } =: \eta_r.
 $$
As a consequence, provided that $\varepsilon \leq \eta_r$, we have
$$
||u_k(t)|^2 - |u_k(0)|^2| \lesssim \varepsilon_r^{-4} \varepsilon^{6} +   \|P\|_{\mathscr{H}(\omega^{(i)})}  |\omega^{(i)}|_\infty \varepsilon^{2r+2} \eta_r^{-2r} |t| + \varepsilon  \|g\|_{L^\infty(0;t)L^2_x} |t|.
$$
Hence, to prove that if 
$ |t|\leq T_\varepsilon$ (which will be optimized later), we have $||u_k(t)|^2 - |u_k(0)|^2|\leq \varepsilon^{6-\nu}$, it is enough to prove that the following estimates holds
\begin{align}
\varepsilon &\leq \eta_r, \label{eq:1} \tag{I}\\
\eta_r^{-4} \varepsilon^{6} &\lesssim_{\nu,s} \varepsilon^{6-\nu/2},  \label{eq:2} \tag{II}\\
 \|P\|_{\mathscr{H}(\omega^{(i)})}  |\omega^{(i)}|_\infty \varepsilon^{2r+2} \eta_r^{-2r} T_\varepsilon &\lesssim_{\nu,s} \varepsilon^{6},  \label{eq:3} \tag{III}\\
 \varepsilon  \|g\|_{L^\infty(-T_\varepsilon;T_\varepsilon)L^2_x} T_\varepsilon &\lesssim_{s,\nu} \varepsilon^{6}.\label{eq:4} \tag{IV}
\end{align}

\noindent \underline{$\triangleright$ \emph{Step 2 : Simplification of the estimates.}} Now we aim at simplifying these constraints.

First, we note that since, by assumption, without loss of generality, we can assume that $\nu \leq 2$, we see that provided that $\varepsilon$ is small enough, the first estimate \eqref{eq:1} is a consequence of the second one \eqref{eq:2}. Therefore, these estimate reduces to the second one \eqref{eq:2} :
\begin{equation}
\label{eq:2bis} \tag{IIb}
\varepsilon^{\nu/8}   \lesssim_{\nu,s} \eta_r.
\end{equation}
Then plugging this estimate \eqref{eq:2bis} in the third one \eqref{eq:3} and using the estimate we proved on $ \|P\|_{\mathscr{H}(\omega^{(i)})}$ (and using that, since $r\geq 3$ and $\nu \leq 2$,  we have $2r-4-r\nu /4 \geq   r/3$ ), the constraint \eqref{eq:3}  can be replaced by
\begin{equation}
\label{eq:3bis} \tag{IIIb}
 M^3 \varepsilon^{r/3}  T_\varepsilon\lesssim_{\nu,s} 1.
\end{equation}
Finally, using the estimate \eqref{eq:est_g_NLS*} that we proved on $\|g\|_{L^\infty(-T_\varepsilon;T_\varepsilon)L^2_x}$, the last constraint \eqref{eq:4} will be satisfied if we can ensure the stronger constraint
\begin{equation}
\label{eq:4bis} \tag{IVb}
T_{\varepsilon}^{5\beta_s} M^{-1} \lesssim_{\nu,s} 1.
\end{equation}

\noindent \underline{$\triangleright$ \emph{Step 3 : Choice of the parameters.}}  In order to satisfy the estimates \eqref{eq:3bis} and \eqref{eq:4bis}, it is enough to set
$$
T_\varepsilon = \varepsilon^{- (30\beta_s)^{-1} r} \quad \mathrm{and} \quad M = \varepsilon^{-r/12}.
$$
Now the only remaining constraint is \eqref{eq:2bis}, i.e.
$$
\varepsilon^{\nu/8}   \lesssim_{\nu,s} e^{-\frac{c}2 \frac{\log M}{\log \log M}} \big( 2  \langle k \rangle\big)^{-\frac12e^{\alpha r} }.
$$
Then, noticing that (since $M = \varepsilon^{-r/12}$) provided that $\varepsilon$ is smaller than a constant depending only on $\nu$ and $\alpha$ (i.e. uniform with respect to $r\geq 3$), we have
$$
e^{-\frac{c}2 \frac{\log M}{\log \log M}} \geq  \varepsilon^{\nu/16}  2  ^{-\frac12e^{\alpha r} },
$$
the constraint \eqref{eq:2bis} can be replaced by
$$
\varepsilon^{\nu/16}   \lesssim_{\nu,s}  \big( 2  \langle k \rangle\big)^{-e^{\alpha r} }.
$$
and so by
\begin{equation}
\label{eq:2ter} \tag{IIt}
\varepsilon  \lesssim_{\nu,s}  \big( 2  \langle k \rangle\big)^{- e^{\alpha_\nu r} } \quad \mathrm{where} \quad \alpha_\nu := \alpha + \frac13 \log(16\nu^{-1}).
\end{equation}
Now, we set
$$
r_* := \frac1{2\alpha_\nu} \log \frac{\log (\varepsilon^{-1})}{\log( 2  \langle k \rangle)},
$$
in such a way that
$$
\big( 2  \langle k \rangle\big)^{- e^{\alpha_\nu 2r_*} } = \varepsilon,
$$
and so that if $r\geq 3$ is an integer in the interval $[r_*,2 r_*]$ then the constraint \eqref{eq:2ter} holds. To be able to choose such a $r$, it is enough to prove that $r_* \geq 3/2$. Fortunately, we recall that we are dealing with the case $2\langle k \rangle < \varepsilon^{-\upsilon_{\alpha,s,\nu}} $ where $\upsilon_{\alpha,s,\nu} \in (0;1)$ is a constant we have to optimize. Therefore, we know that
$$
r_* \geq \frac1{2\alpha_\nu} \log \upsilon_{\alpha,s,\nu}^{-1},
$$
and so it is enough to require that $\upsilon_{\alpha,s,\nu}$ is small enough to have
$$
 \frac1{2\alpha_\nu} \log \upsilon_{\alpha,s,\nu}^{-1} \geq \frac32.
$$
That is why we set
\begin{equation}
\label{eq:def_upsi}
\upsilon_{\alpha,s,\nu} :=\frac{\nu}{16} e^{-3 \alpha}.
\end{equation}
Note that therefore, since $\nu\leq 2$, we have $\upsilon_{\alpha,s,\nu}\leq 1/8$. Moreover, since $r/12 \geq 1/4$, the assumption $|k|\leq M$ is satisfied :
$$
|k|\leq \varepsilon^{-\upsilon_{\alpha,s,\nu}} \leq \varepsilon^{-1/8} \leq \varepsilon^{-1/4} \leq \varepsilon^{-r/12}=M.
$$
To conclude this proof, it is enough to note that by construction, we have proven that 
$$||u_k(t)|^2 - |u_k(0)|^2|\leq \varepsilon^{6-\nu},$$ while 
$$
|t| \leq T_\varepsilon = \varepsilon^{- (30\beta_s)^{-1} r} \quad \mathrm{where} \quad T_\varepsilon \geq \varepsilon^{- (30\beta_s)^{-1} r_*} = \varepsilon^{- \frac1{60 \beta_s \alpha_\nu} \log \frac{\log (\varepsilon^{-1})}{\log( 2  \langle k \rangle)}}.
$$

\section{Proof of Theorem \ref{thm:mult}} The proof of Theorem \ref{thm:mult} is very similar to the one of Theorem \ref{thm:conv}. It has the same structure, but some estimates are more involved.

In all the proof we consider a real valued even potential $W \in H^4(\mathbb{T};\mathbb{R})$ (i.e. $W(x) \in \mathbb{R}$ and $W(-x) = W(x)$) such that the frequencies $\omega = (\lambda_k)_{k\geq 1}$
 are strongly-non-resonant according to Definition \ref{def:strg_nr}, where $(\lambda_k)_{k\geq 1}$ is the increasing sequence of eigenvalues of the Sturm--Liouville operator $-\partial_x^2 + W_{|[0;\pi]}$ with homogeneous Dirichlet boundary conditions. We denote by $(f_k)_{k\geq 1}$ the associated eigenfunctions (see Prop \ref{prop_dir}).

Now, thanks to Proposition \ref{prop:GWPW}, we consider a global odd solution $u\in C^0(\mathbb{R}; H^s(\mathbb{T};\mathbb{C}))$, $s>2/5$, of  \eqref{eq:NLS} such that 
$$
\| u(0) \|_{H^s} =: \varepsilon \leq \frac1{100}.
$$
Note that the existence of odd solutions is ensured by the assumption that the potential $V$ is even. Without loss of generality, we assume that $s\leq 1$.
Since the $L^2$ norm is a constant of the motion for \eqref{eq:NLS}, we have
$$
\forall t\in \mathbb{R}, \quad   \|u(t) \|_{L^2} = \|u(0) \|_{L^2} \leq \|u(0) \|_{H^s} = \varepsilon.
$$
Moreover thanks to Proposition \ref{thm:gwp}, we know that there exists $\beta_s\geq 1$ such that
\begin{equation}
\label{eq:grow_hs}
\forall t\in \mathbb{R}, \quad \| u(t) \|_{H^s} \lesssim_s \varepsilon \langle t  \rangle^{\beta_s}.
\end{equation}
In this proof, being given an odd function $v\in L^2(\mathbb{T})$, we denote
$$
v_k = \int_{0}^\pi v(x) f_k(x) \mathrm{d}x.
$$
Moreover, since $(f_k)_{k\geq 1}$ is a Hilbertian basis of $L^2(0;\pi)$ (see Proposition \ref{prop_dir}), we know that
$$
v(x) = \sum_{k\geq 1} v_k f_k(x)
$$
where $f_k$ is extended as an odd function on $\mathbb{T}$.

\subsection{Control of the high actions} Proceeding exactly as in the proof of Theorem \ref{thm:conv} (see subsection \ref{subsec:high}), it can be proven that if 
$\langle k \rangle \geq \varepsilon^{-\upsilon_{\alpha,s,\nu}}$ (where $\upsilon_{\alpha,s,\nu}$ is also given by \eqref{eq:def_upsi}) then there exists $\mu$ such that
\begin{equation}
\label{eq:what_we_want_triv}
|t|\leq \varepsilon^{- \mu \log \frac{\log \varepsilon^{-1}}{\log (\langle k\rangle)}}  \quad\Rightarrow \quad ||u_k(t)|^2 - |u_k(0)|^2| \leq \varepsilon^{6-\nu}.
\end{equation}
\subsection{Setting for the low actions} Now, and until the end of this proof, we aim at controlling the variations of $|u_k(t)|^2$ when $\langle k \rangle < \varepsilon^{-\upsilon_{\alpha,s,\nu}}$. We consider a  large parameter $M\geq \varepsilon^{-\upsilon_{\alpha,s,\nu}}$ that will be optimized later, and we define
$$
 \mathcal{E}_M:= \mathrm{Span}_\mathbb{C} \{f_k \ | \ k\leq M\}.
$$
As usual, we identify $\mathcal{E}_M$ with $\mathbb{C}^\mathcal{M}$ (through the Fourier transform) where
$$
\mathcal{M}= \llbracket 1,M\rrbracket
$$
We denote by $\Pi^{(M)}$ the $L^2$-orthogonal projection on $\mathcal{E}_M$ and we set 
$$
u^{(\leq M)} :=  \Pi^{(M)}(u(t)).
$$
Note that, the $L^2$ norm being a constant of the motion and $\Pi^{(M)}$ being an orthogonal projection, we have
\begin{equation}
\label{eq:pres_L2_M}
\forall t\in \mathbb{R}, \quad \|u^{(\leq M)}(t) \| \leq  \|u(t) \|_{L^2} = \|u(0) \|_{L^2} \leq \|u(0) \|_{H^s} = \varepsilon  .
\end{equation}
Moreover, $u^{(\leq M)}$ solves the equation
\begin{equation}
\label{eq:defgt}
i\partial_t u^{(\leq M)} = \nabla H (u^{(\leq M)}) + g(t) \quad \mathrm{with} \quad g(t):= \sigma \Pi^{(M)} \big[  |u(t)|^4u(t) -   |\Pi^{(M)} u(t)|^4 \Pi^{(M)} u(t) \big],
\end{equation}
and
$$
H = Z_{2,\omega} + P \quad \mathrm{with} \quad P = \frac1{12} (\| \cdot \|_{L^6}^6)_{| \mathcal{E}_M} \quad \mathrm{and} \quad Z_{2,\omega} \mathrm{ \ given \ by \ \eqref{eq:defZ2}}.
$$

\subsection{Strichartz estimates} Now we aim at proving the same estimate on $\| P \|_{\mathscr{H}(\omega^{(i)})}$ (where $\omega^{(i)}_k :=k^2$) as in the case of \eqref{eq:NLS*}. In the paragraph \emph{Identification of the Hamiltonian structure} page 737 of \cite{BG21}, it is proven that (provided that $\sum_{k \geq 1} \langle k\rangle^2 |u_k|^2 <\infty $)
$$
\| \sum_{k \geq 1} u_k f_k \|_{L^6}^6 = \sum_{\boldsymbol{k},\boldsymbol{\ell} \in (\mathbb{N}^*)^3} Q_{\boldsymbol{k},\boldsymbol{\ell}} u_{\boldsymbol{k}_1}u_{\boldsymbol{k}_2} u_{\boldsymbol{k}_3} \ \overline{u_{\boldsymbol{\ell}_1}} \overline{u_{\boldsymbol{\ell}_2}} \overline{u_{\boldsymbol{\ell}_3}}
$$
where the coefficients $Q_{\boldsymbol{k},\boldsymbol{\ell}}$ are symmetric and satisfy 
$$
|Q_{\boldsymbol{k},\boldsymbol{\ell}}|\lesssim_{\|W\|_{L^2}} \sum_{\nu,\mu \in \{-1,1\}^3} \langle \nu_1\boldsymbol{k}_1 + \nu_2\boldsymbol{k}_2 + \nu_3\boldsymbol{k}_3  + \mu_1 \boldsymbol{\ell}_1 + \mu_2 \boldsymbol{\ell}_2 + \mu_3\boldsymbol{\ell}_3  \rangle^{-2}.
$$
Therefore, for all $a\in \mathbb{Z}$, we have 
$$
\forall u \in \mathcal{E}_M, \quad \lfloor \Pi_{\omega^{(i)},a} P \rceil(u) =  \frac1{12}  \sum_{\substack{\boldsymbol{k},\boldsymbol{\ell} \in \llbracket 1,M\rrbracket^3\\ \boldsymbol{k}_1^2 + \boldsymbol{k}_2^2 + \boldsymbol{k}_3^2  -\boldsymbol{\ell}_1^2 - \boldsymbol{\ell}_2^2 - \boldsymbol{\ell}_3^2 = a}} |Q_{\boldsymbol{k},\boldsymbol{\ell}}| u_{\boldsymbol{k}_1}u_{\boldsymbol{k}_2} u_{\boldsymbol{k}_3} \ \overline{u_{\boldsymbol{\ell}_1}} \overline{u_{\boldsymbol{\ell}_2}} \overline{u_{\boldsymbol{\ell}_3}}.
$$
and so
\begin{equation}\label{moment}
\begin{split}
|\lfloor \Pi_{\omega^{(i)},a} P \rceil(u)| &\lesssim_{\|W\|_{L^2}}  \sum_{\nu,\mu \in \{-1,1\}^3}  \! \! \! \!  \sum_{\substack{\boldsymbol{k},\boldsymbol{\ell} \in \llbracket 1,M\rrbracket^3\\ \boldsymbol{k}_1^2 + \boldsymbol{k}_2^2 + \boldsymbol{k}_3^2  -\boldsymbol{\ell}_1^2 - \boldsymbol{\ell}_2^2 - \boldsymbol{\ell}_3^2 = a \\ \nu_1\boldsymbol{k}_1 + \nu_2\boldsymbol{k}_2 + \nu_3\boldsymbol{k}_3  + \mu_1 \boldsymbol{\ell}_1 + \mu_2 \boldsymbol{\ell}_2 + \mu_3\boldsymbol{\ell}_3 = j}} \! \! \! \! \frac{ |u_{\boldsymbol{k}_1}u_{\boldsymbol{k}_2} u_{\boldsymbol{k}_3} \ \overline{u_{\boldsymbol{\ell}_1}} \overline{u_{\boldsymbol{\ell}_2}} \overline{u_{\boldsymbol{\ell}_3}}|}{\langle j  \rangle^2} \\
&=  \sum_{\nu,\mu \in \{-1,1\}^3} \int_{\mathbb{T}}  \int_{\mathbb{T}} w(x) e^{-ita} \prod_{n=1}^3 e^{-it\partial_x^2}v^{(\nu_n)}(x) e^{it\partial_x^2}v^{(\mu_n)}(x) \mathrm{d}x \mathrm{d}t,
\end{split}
\end{equation}
where $v^{(\pm 1)}(x) := \sum_{1\leq k \leq M} |u_k| e^{\pm ikx}$ and $w(x) = \sum_{k\in \mathbb{Z}} \langle k \rangle^{-2} e^{ ik x} \in H^1(\mathbb{T}) \subset L^\infty(\mathbb{T})$. As a consequence, applying H\"older's inequality, we get
$$
|\lfloor \Pi_{\omega^{(i)},a} P \rceil(u)| \lesssim_{\|W\|_{L^2}} \|w\|_{L^\infty}  \sum_{\nu,\mu \in \{-1,1\}^3}   \prod_{n=1}^3 \| e^{-it\partial_x^2}v^{(\nu_n)} \|_{L^6(\mathbb{T}^2)}  \| e^{it\partial_x^2}v^{(\mu_n)} \|_{L^6(\mathbb{T}^2)}.
$$
Finally, noticing that $\|v^{(\pm 1)} \|_{L^2} = \| u\|_{L^2}$ and applying the Strichartz estimate \eqref{strichartz}, we get
$$
\| P \|_{\mathscr{H}(\omega^{(i)})}\lesssim_{\|W\|_{L^2}}  e^{c \frac{\log M}{\log \log M}},
$$
where $c>0$ is a universal constant.

\subsection{Estimate of the remainder term}
As in subsection \ref{sub:estg*} (i.e. for \eqref{eq:NLS*}), we aim at proving that, provided that $s\leq 1$, we have
\begin{equation}
\label{eq:y_fait_chaud}
\| g \|_{L^2} \lesssim \| u\|_{H^s}^5 M^{-(s-\frac25)},
\end{equation}
which, using the a priori bound \eqref{eq:grow_hs} on the growth of the $H^s$ norm, provides 
\begin{equation}
\label{eq:est_g_NLS}
\| g \|_{L^2} \lesssim  \varepsilon^5 \langle t\rangle^{5 \beta_s} M^{-(s-\frac25)}.
\end{equation}
First, we note that for the same reasons as in subsection \ref{sub:estg*}, the remainder term enjoys the estimate
\begin{equation}
\label{eq:maison}
\| g \|_{L^2} \lesssim \| u\|_{H^s}^4 \| u - \Pi^{(M)} u \|_{H^{2/5}} \simeq  \| u\|_{H^s}^4 \| \sum_{k > M} u_k f_k \|_{H^{2/5}}.
\end{equation}
Then we note that
\begin{equation}
\label{eq:lapin}
\forall s' \in[0;1], \forall v\in \ell^2(\mathbb{N}^*;\mathbb{C}), \quad \| \sum_{k\geq 1} v_k f_k \|_{H^{s'}}^2 \simeq  \sum_{k\geq 1} \langle k \rangle^{2 s'} |v_k|^2.
\end{equation}
Indeed, the case $s'=1$ is proven in Proposition 6.2 page 733 of \cite{BG21} while the case $s'=0$ is just a consequence of the fact that $(f_j)_{j\geq 1}$ is a Hilbertian basis. Therefore, the case $0<s'<1$ follows directly by interpolation\footnote{we refer the reader to \cite[Thm page 130]{Tri78} and \cite[Thm 13.2.2 page 198 and Thm 13.2.1 page 197]{Agra15} for specific results of interpolation well suited to this setting.}.

Finally, plugging \eqref{eq:lapin} into \eqref{eq:maison}, it comes (as expected)
\begin{equation*}
\begin{split}
 \| g \|_{L^2} &\lesssim \| u\|_{H^s}^4 \| \sum_{k > M} u_k f_k \|_{H^{2/5}} \simeq \| u\|_{H^s}^4 \big( \sum_{k > M} \langle k \rangle^{4/5} |u_k|^2  \big)^{1/2}  \\ &\lesssim \| u\|_{H^s}^4  M^{-(s-\frac25)}  \big( \sum_{k > M} \langle k \rangle^{2s} |u_k|^2  \big)^{1/2} \simeq  \| u\|_{H^s}^5 M^{-(s-\frac25)}.
\end{split}
\end{equation*}

\subsection{Optimization of the parameters} Since the estimates on $g$ and $P$ are the same as for \eqref{eq:NLS*}, the rest of the proof is exactly the same as in subsection \ref{sub:opt*}. For completeness, it may be just relevant to mention that the estimates
$$
 |\omega^{(i)}|_\infty \lesssim M^2 \quad \mathrm{and} \quad   |\omega^{(f)}|_\infty \lesssim 1
$$
follow directly from \cite[Theorem 4 p.35]{PT}. We also mention that since $|k|\geq 1$ then $\langle k \rangle \geq \sqrt{2}>1 $ and so the quantity $2\langle k \rangle$ of subsection \ref{sub:opt*} can always be replaced by $\langle k \rangle$ here.

\section{Small divisor estimates}
This section is devoted to the proof of Proposition \ref{prop:main_conv} and Proposition \ref{prop:main_mult}.
\subsection{Proof of Proposition \ref{prop:main_conv}} First, we recall the classical proof stating that, almost surely, the frequencies of \eqref{eq:NLS*} for the random convolution potential $V$ as in \eqref{random*} enjoy a weak non-resonance condition.
\begin{lemma} \label{lem:weak_nr_*}Almost surely, there exists $\gamma>0$ such that, for all $q\geq 1$, $\boldsymbol{m} \in (\mathbb{Z}^*)^q$, $\boldsymbol{h}_1,\cdots,\boldsymbol{h}_q \in \mathbb{Z}$ all distinct, we have
\begin{equation}
\label{eq:sd_triv}
\forall a\in \mathbb{Z}, \quad \big|  a+ \sum_{j=1}^{q} \boldsymbol{m}_j \omega^{\eqref{eq:NLS*}}_{\boldsymbol{h}_j}  \big| \geq \gamma   \Big(\min_j \langle \boldsymbol{h}_j \rangle\Big)^{-s_*} \prod_{j=1}^q |  \boldsymbol{m}_j |^{-4} \langle  \boldsymbol{h}_j \rangle^{-4}.
\end{equation}
\end{lemma}
\begin{proof} Being given $a,q,\boldsymbol{m},\boldsymbol{h}$ satisfying the assumptions of the lemma, by definition of the random convolution potential (see  \eqref{random*}), we have for all $\gamma>0$ and all $j_* \in \llbracket 1,q\rrbracket$
\begin{equation*}
\begin{split}
& \mathbb{P} \Big( \big|  a+ \sum_{j=1}^{q} \boldsymbol{m}_j \omega^{\eqref{eq:NLS*}}_{\boldsymbol{h}_j}  \big| < \gamma \Big)\\
=& \frac1{\sqrt{2\pi}} \mathbb{E} \int_{\mathbb{R}}   \mathbbm{1}_{ |  m_{j_*}  \boldsymbol{h}_{j_*}  ^2 +   (2\pi)^{1/2} y m_{j_*}  \langle  \boldsymbol{h}_{j_*} \rangle^{-s_*}   +a+ \sum_{j\neq j_*} \boldsymbol{m}_j \omega^{\eqref{eq:NLS*}}_{\boldsymbol{h}_j}  | < \gamma}  e^{-y^2/2}\mathrm{d}y \\
\lesssim & \, \mathbb{E} \int_{\mathbb{R}}   \mathbbm{1}_{ |  m_{j_*}  \boldsymbol{h}_{j_*}  ^2 +  (2\pi)^{1/2}y m_{j_*}  \langle  \boldsymbol{h}_{j_*} \rangle^{-s_*}   +a+ \sum_{j\neq j_*} \boldsymbol{m}_j \omega^{\eqref{eq:NLS*}}_{\boldsymbol{h}_j}  | < \gamma} \mathrm{d}y \sim \gamma |m_{j_*} |^{-1} \langle  \boldsymbol{h}_{j_*} \rangle^{s_*}
\end{split}
\end{equation*}
and so
$$
\mathbb{P} \Big( \big|  a+ \sum_{j=1}^{q} \boldsymbol{m}_j \omega^{\eqref{eq:NLS*}}_{\boldsymbol{h}_j}  \big| < \gamma \Big) \lesssim \gamma \min_j \langle \boldsymbol{h}_j \rangle^{s_*}.
$$
Therefore, for all $\gamma>0$, we have
\begin{equation*}
\begin{split}
&\mathbb{P}\Big(\exists (a,q,\boldsymbol{m},\boldsymbol{h}), \quad  \big|  a+ \sum_{j=1}^{q} \boldsymbol{m}_j \omega^{\eqref{eq:NLS*}}_{\boldsymbol{h}_j}  \big| < \gamma  \langle a \rangle^{-2} \Big(\min_j \langle \boldsymbol{h}_j \rangle\Big)^{-s_*} \prod_{j=1}^q |  \boldsymbol{m}_j |^{-2}   \langle  \boldsymbol{h}_j \rangle^{-2} \Big) \\
&\lesssim \gamma \sum_{a \in \mathbb{Z}} \sum_{q\geq 1} \sum_{\boldsymbol{m} \in (\mathbb{Z}^*)^q} \sum_{\substack{\boldsymbol{h} \in \mathbb{Z}^q\\ \boldsymbol{h}_1 < \cdots < \boldsymbol{h}_q }} \langle a \rangle^{-2} \prod_{j=1}^q |  \boldsymbol{m}_j |^{-2}  \langle  \boldsymbol{h}_j \rangle^{-2} \\
& \lesssim \gamma \sum_{q\geq 1} \sum_{\boldsymbol{m} \in (\mathbb{Z}^*)^q} \sum_{\boldsymbol{h} \in \mathbb{Z}^q } (q!)^{-1} \prod_{j=1}^q | \boldsymbol{m}_j |^{-2}  \langle  \boldsymbol{h}_j \rangle^{-2}  \lesssim \gamma \mathop{\longrightarrow}_{\gamma \to 0} 0, 
\end{split}
\end{equation*}
where at the first line $(a,q,\boldsymbol{m},\boldsymbol{h})$ have implicitly to satisfy the assumptions of the lemma. It means that,
 almost surely, there exists $\gamma>0$ such that, for all $q\geq 1$, $\boldsymbol{m} \in (\mathbb{Z}^*)^q$, $\boldsymbol{h}_1,\cdots,\boldsymbol{h}_q \in \mathbb{Z}$ all distinct, we have
\begin{equation}
\label{eq:bon_interm}
\forall a\in \mathbb{Z}, \quad \big|  a+ \sum_{j=1}^{q} \boldsymbol{m}_j \omega^{\eqref{eq:NLS*}}_{\boldsymbol{h}_j}  \big| \geq \gamma  \langle a \rangle^{-2} \Big(\min_j \langle \boldsymbol{h}_j \rangle\Big)^{-s_*} \prod_{j=1}^q |  \boldsymbol{m}_j |^{-2} \langle  \boldsymbol{h}_j \rangle^{-2}.
\end{equation}
Then, we note that either
$$
|a|\geq 1+ \sum_{j=1}^{q} |\boldsymbol{m}_j| |\omega^{\eqref{eq:NLS*}}_{\boldsymbol{h}_j}|,
$$
and so the small divisor is larger than or equal to $1$ (and so the estimate \eqref{eq:sd_triv} is satisfied with $\gamma=1$) or (using that $|\omega^{\eqref{eq:NLS*}}_{\boldsymbol{h}_j}| \lesssim_{\|V\|_{L^2}} \langle \boldsymbol{h}_j \rangle^{2}$)
$$
|a| \leq 1+ \sum_{j=1}^{q} |\boldsymbol{m}_j| |\omega^{\eqref{eq:NLS*}}_{\boldsymbol{h}_j}| \lesssim  \sum_{j=1}^{q} |\boldsymbol{m}_j| \langle  \boldsymbol{h}_j \rangle^2 \lesssim _{\|V\|_{L^2}} \prod_{j=1}^q |  \boldsymbol{m}_j |^2 \langle  \boldsymbol{h}_j \rangle^2
$$
and so, plugging this estimate in \eqref{eq:bon_interm} we get  \eqref{eq:sd_triv}.
\end{proof}
Now, we aim at improving the small divisor estimate \eqref{eq:sd_triv} in order to prove that almost surely the frequencies of \eqref{eq:NLS*} enjoy a strong non-resonance condition. 

\noindent \underline{\emph{Step 1 : Setting.}}
Let $B>0$ be a constant such that
$$
\forall k \in \mathbb{Z}, \quad |\omega^{\eqref{eq:NLS*}}_k - k^2|\leq B \langle k \rangle^{-s_*}.
$$
We fix $q\geq 1$,  $\boldsymbol{m} \in (\mathbb{Z}^*)^q$, $\boldsymbol{h}_1,\cdots,\boldsymbol{h}_q \in \mathbb{Z}$ all distinct such that $|\boldsymbol{h}_1| \leq \cdots \leq |\boldsymbol{h}_q| $. We define $q_\star \in \llbracket 1,q\rrbracket$ as the maximal index such that
$$
\forall p\in \llbracket 2,q_\star \rrbracket, \quad B \sum_{j=p}^q | \boldsymbol{m}_j| \langle \boldsymbol{h}_j \rangle^{-s_*} \geq \frac{\gamma}2 \langle \boldsymbol{h}_1 \rangle^{-s_*} \prod_{j=1}^{p-1}  |  \boldsymbol{m}_j |^{-4} \langle  \boldsymbol{h}_j \rangle^{-4}.
$$
Therefore it is enough to estimate $\big|a+\sum_{j=1}^{q_\star} \boldsymbol{m}_j \omega^{\eqref{eq:NLS*}}_{\boldsymbol{h}_j}  \big|$ where $a\in\mathbb{Z}$. Indeed, if $q_\star< q$, by maximality of $q_\star$ we have
$$
B \sum_{j=q_\star +1}^q |m_j| \langle \boldsymbol{h}_j \rangle^{-s_*} < \frac{\gamma}2 \langle \boldsymbol{h}_1 \rangle^{-s_*} \prod_{j=1}^{q_\star}  |  \boldsymbol{m}_j |^{-4} \langle  \boldsymbol{h}_j \rangle^{-4},
$$
and so applying the triangle inequality and the non-resonance estimate, we have
\begin{equation}
\label{eq:jy_reviens}
\begin{split}
& \big| \sum_{j=1}^{q} \boldsymbol{m}_j \omega^{\eqref{eq:NLS*}}_{\boldsymbol{h}_j} \big| \\
& \geq \big|  \sum_{j=q_\star+1}^{q} \boldsymbol{m}_j \boldsymbol{h}_j^2 + \sum_{j=1}^{q_\star} \boldsymbol{m}_j \omega^{\eqref{eq:NLS*}}_{\boldsymbol{h}_j} \big|   - \sum_{j=q_\star+1}^{q} |\boldsymbol{m}_j| |\omega^{\eqref{eq:NLS*}}_{\boldsymbol{h}_j} - \boldsymbol{h}_j^2| \\
&\geq \frac12 \big|  \sum_{j=q_\star+1}^{q} \boldsymbol{m}_j \boldsymbol{h}_j^2 + \sum_{j=1}^{q_\star} \boldsymbol{m}_j \omega^{\eqref{eq:NLS*}}_{\boldsymbol{h}_j} \big|   + \frac{\gamma}2 \langle \boldsymbol{h}_1 \rangle^{-s_*} \prod_{j=1}^{q_\star}  |  \boldsymbol{m}_j |^{-4} \langle  \boldsymbol{h}_j \rangle^{-4} -B \sum_{j=q_\star+1}^{q} |\boldsymbol{m}_j| \langle \boldsymbol{h}_j \rangle^{-s_*} \\
&\geq  \frac12 \big|  \sum_{j=q_\star+1}^{q} \boldsymbol{m}_j \boldsymbol{h}_j^2 + \sum_{j=1}^{q_\star} \boldsymbol{m}_j \omega^{\eqref{eq:NLS*}}_{\boldsymbol{h}_j} \big|.
\end{split}
\end{equation}
Now, to estimate $\big|a+\sum_{j=1}^{q_\star} \boldsymbol{m}_j \omega^{\eqref{eq:NLS*}}_{\boldsymbol{h}_j}  \big|$ uniformly with respect to $a\in \mathbb{Z}$, we are just going to use the lower bound given by Lemma \ref{lem:weak_nr_*}, but in order to have a bound depending only on $\boldsymbol{h}_1$, we have to estimate $\langle \boldsymbol{h}_p \rangle$ for all $p\in \llbracket 2,q_* \rrbracket$.

\medskip

\noindent \underline{\emph{Step 2 : Estimation of $\langle \boldsymbol{h}_p \rangle$ with respect to $\langle \boldsymbol{h}_1 \rangle$.}} By definition of $q_\star$, we deduce that if $p\in \llbracket 2,q_\star \rrbracket$ then
$$
B  | \boldsymbol{m}|_1 \langle \boldsymbol{h}_p \rangle^{-s_*} \geq \frac{\gamma}2 \langle \boldsymbol{h}_1 \rangle^{-s_*}  | \boldsymbol{m}|_1^{-4(p-1)}\prod_{j=1}^{p-1} \langle  \boldsymbol{h}_j \rangle^{-4}.
$$
Applying the $\log$ function to this estimate, we get
\begin{equation}
\label{eq:cest_la_fin}
y_p \leq \log(C) + 4 s_*^{-1}q_{\star} \log(| \boldsymbol{m}|_1 ) + y_1 + 4s_*^{-1}\sum_{j=1}^{p-1} y_j,
\end{equation}
where $\boldsymbol{y}_j =  \log \langle  \boldsymbol{h}_j \rangle$ and $C:= (2B\gamma^{-1} )^{s_*^{-1}}$. Here, we note that this last relation is also valid for $p=1$. Therefore, as a consequence of the discrete Gr\"onwall inequality, we have
$$
y_p \leq (\log(C) + 4 s_*^{-1}q_{\star} \log(| \boldsymbol{m}|_1 ) + y_1) e^{\frac{4(p-1)}{s_*}},
$$
and so
$$
\langle  \boldsymbol{h}_p \rangle \leq \Big( C | \boldsymbol{m}|_1^{ 4 s_*^{-1}q_{\star}}   \langle  \boldsymbol{h}_1 \rangle \Big)^{ \exp(\frac{4q_\star}{s_*})}.
$$

\medskip
 
 \noindent \underline{\emph{Step 3 : Conclusion.}} Plugging this last estimate in the classical non-resonance condition \eqref{eq:sd_triv} yield
$$
\big|a+\sum_{j=1}^{q_\star} \boldsymbol{m}_j \omega^{\eqref{eq:NLS*}}_{\boldsymbol{h}_j}  \big| \geq \gamma    \langle \boldsymbol{h}_1 \rangle^{-s_*}|  \boldsymbol{m} |_1^{-4q_\star} \Big( C | \boldsymbol{m}|_1^{ 4 s_*^{-1}q_{\star}}   \langle  \boldsymbol{h}_1 \rangle \Big)^{-4q_\star \exp(\frac{4q_\star}{s_*})}.
$$
Then, using the rough estimate $q_\star\leq q\leq | \boldsymbol{m}|_1=:r$, it comes
$$
\big|a+\sum_{j=1}^{q_\star} \boldsymbol{m}_j \omega^{\eqref{eq:NLS*}}_{\boldsymbol{h}_j}  \big| \geq \gamma  C^{-4r \exp(\frac{4r}{s_*})} r^{-4r-16r^2 s_*^{-1} \exp(\frac{4r}{s_*})} \langle \boldsymbol{h}_1 \rangle^{-s_*-4r \exp(\frac{4r}{s_*})}  .
$$
As a consequence, a standard asymptotic analysis proves that there exists a constant $\rho>0$ and a constant $\alpha$ depending only on $s_*$ such that we have
$$
\big|a+\sum_{j=1}^{q_\star} \boldsymbol{m}_j \omega^{\eqref{eq:NLS*}}_{\boldsymbol{h}_j}  \big| \geq 2 \rho \big( 2  \langle \boldsymbol{h}_1 \rangle\big)^{-e^{- \alpha r}},
$$
which plugged in \eqref{eq:jy_reviens} concludes this proof.
\subsection{Proof of Proposition \ref{prop:main_mult}} The scheme of this proof is very similar to the previous one. 
First, we recall the classical proof stating that, almost surely, provided that the potential is small enough, the frequencies of \eqref{eq:NLS} for the random multiplicative potential \eqref{random} enjoy a weak non-resonance condition.
\begin{lemma}\label{lem:weak_nr} There exists $\eta >0$ such that, almost surely, provided that $\|W^{\eqref{eq:NLS}} \|_{H^1}<\eta$, there exists $\gamma>0$ such that, for all $q\geq 1$, $\boldsymbol{m} \in (\mathbb{Z}^*)^q$, $\boldsymbol{h}_1,\cdots,\boldsymbol{h}_q \in \mathbb{N}^*$ all distinct, we have
\begin{equation}
\label{eq:sd_triv_mult}
\forall a\in \mathbb{Z}, \quad \big|  a+ \sum_{j=1}^{q} \boldsymbol{m}_j \omega^{\eqref{eq:NLS}}_{\boldsymbol{h}_j}  \big| \geq \gamma   \Big(\max_j \langle \boldsymbol{h}_j \rangle\Big)^{-s_*} \prod_{j=1}^q |  \boldsymbol{m}_j |^{-4} \langle  \boldsymbol{h}_j \rangle^{-4}.
\end{equation}
\end{lemma}
\begin{proof} Actually, we just have to prove that there exists $\eta >0$ such that, being given $a,q,\boldsymbol{m},\boldsymbol{h}$ satisfying the assumptions of the lemma, we have
\begin{equation}
\label{eq:la_poste}
\forall \gamma>0, \quad \mathbb{P} \Big( \big|  a+ \sum_{j=1}^{q} \boldsymbol{m}_j \omega^{\eqref{eq:NLS}}_{\boldsymbol{h}_j}  \big| < \gamma \ \Big| \ \|W^{\eqref{eq:NLS}}\|_{H^1} < \eta \Big) \lesssim \gamma \max_j \langle \boldsymbol{h}_j \rangle^{s_*}.
\end{equation}
Indeed, the rest of the proof is the same as the one of Lemma \ref{lem:weak_nr_*} (the estimate $|\omega^{\eqref{eq:NLS}}_{\boldsymbol{h}_j}| \lesssim \langle \boldsymbol{h}_j \rangle^{2}$ follows directly from \cite[Theorem 4 p.35]{PT}). Moreover, the existence of $\eta>0$ and the estimate \eqref{eq:la_poste} are proven in \cite{BG21}. More precisely, we refer the reader to the proof of Proposition 1.12 in \cite{BG21} and in particular to the last estimate page 703 of \cite{BG21}.
\end{proof}
Now, we aim at improving the small divisor estimate \eqref{eq:sd_triv_mult} in order to prove that almost surely the frequencies of \eqref{eq:NLS} are strongly non-resonant. From now on we condition the potential $W^{\eqref{eq:NLS}}$ in \eqref{random} to be small enough in $H^1$ (in any case $\| W^{\eqref{eq:NLS}}\|_{H^1}<\eta$) in such a way that by \cite[Proposition 2.7 p.700]{BG21} (which is a variation of \cite[Theorem 4 p.35]{PT}), almost surely $\omega^{\eqref{eq:NLS}}_1 \geq 1/2$ and there exists $B>0$ such that
$$
\forall k \in \mathbb{N}^*, \quad |\omega^{\eqref{eq:NLS}}_k - k^2|\leq B \langle k \rangle^{-1} .
$$
Note that here we have used that $\int_0^\pi W^{\eqref{eq:NLS}}(x) \mathrm{d}x=0$.
We fix $q\geq 1$,  $\boldsymbol{m} \in (\mathbb{Z}^*)^q$, $\boldsymbol{h}_1,\cdots,\boldsymbol{h}_q \in \mathbb{N}^*$ all distinct such that $\boldsymbol{h}_1 < \cdots < \boldsymbol{h}_q $. We define $q_\star \in \llbracket 1,q\rrbracket$ as the maximal index such that
$$
\forall p\in \llbracket 2,q_\star \rrbracket, \quad B \sum_{j=p}^q | \boldsymbol{m}_j| \langle \boldsymbol{h}_j \rangle^{-1} \geq \frac{\gamma}2 \langle \boldsymbol{h}_{p-1} \rangle^{-s_*} \prod_{j=1}^{p-1}  |  \boldsymbol{m}_j |^{-4} \langle  \boldsymbol{h}_j \rangle^{-4}.
$$
Therefore it is enough to estimate $\big|a+\sum_{j=1}^{q_\star} \boldsymbol{m}_j \omega^{\eqref{eq:NLS}}_{\boldsymbol{h}_j}  \big|$ where $a\in\mathbb{Z}$. Indeed, if $q_\star< q$, by maximality of $q_\star$ we have
$$
B \sum_{j=q_\star +1}^q |m_j| \langle \boldsymbol{h}_j \rangle^{-1} < \frac{\gamma}2 \langle \boldsymbol{h}_{q_\star} \rangle^{-s_*} \prod_{j=1}^{q_\star}  |  \boldsymbol{m}_j |^{-4} \langle  \boldsymbol{h}_j \rangle^{-4},
$$
and so applying the triangle inequality and the non-resonance estimate, we have (as previously, see \eqref{eq:jy_reviens} for details)
$$
 \big| \sum_{j=1}^{q} \boldsymbol{m}_j \omega^{\eqref{eq:NLS}}_{\boldsymbol{h}_j} \big| 
\geq \frac12 \big|  \sum_{j=q_\star+1}^{q} \boldsymbol{m}_j \boldsymbol{h}_j^2 + \sum_{j=1}^{q_\star} \boldsymbol{m}_j \omega^{\eqref{eq:NLS}}_{\boldsymbol{h}_j} \big|.
$$
Now, as previously, to estimate $\big|a+\sum_{j=1}^{q_\star} \boldsymbol{m}_j \omega^{\eqref{eq:NLS}}_{\boldsymbol{h}_j}  \big|$ uniformly with respect to $a\in \mathbb{Z}$, we are just going to use the lower bound given by Lemma \ref{lem:weak_nr}, but in order to have a bound depending only on $\boldsymbol{h}_1$, we have to estimate $\langle \boldsymbol{h}_p \rangle$ for all $p\in \llbracket 2,q_* \rrbracket$.

 By definition of $q_\star$, we deduce that if $p\in \llbracket 2,q_\star \rrbracket$ then
$$
B  | \boldsymbol{m}|_1 \langle \boldsymbol{h}_p \rangle^{-1} \geq \frac{\gamma}2 \langle \boldsymbol{h}_{p-1} \rangle^{-s_*}  | \boldsymbol{m}|_1^{-4(p-1)}\prod_{j=1}^{p-1} \langle  \boldsymbol{h}_j \rangle^{-4}.
$$
Applying the $\log$ function to this estimate, we get
$$
y_p \leq \log(C) + 4 q_{\star} \log(| \boldsymbol{m}|_1 ) + s_*y_{p-1} +4\sum_{j=1}^{p-1} y_j \leq \log(C) + (4+s_*) q_{\star} \log(| \boldsymbol{m}|_1 ) + y_1 + (4+s_*)\sum_{j=1}^{p-1} y_j,
$$
where $\boldsymbol{y}_j =  \log \langle  \boldsymbol{h}_j \rangle$ and $C:= 2B\gamma^{-1} $. Note that this estimate is the same as \eqref{eq:cest_la_fin} except that $4s_*^{-1}$ is replaced by $4+s_*$. Up to this change of constant, the rest of the proof is exactly the same as the one of Proposition \ref{prop:main_conv}.

\section{Global well-posedness of the full dynamics}\label{sec:gwp}
In this section, we give the proof of Proposition~\ref{thm:gwp}. We first review the general strategy of the $I$-method and the argument of \cite{LWX}, and then give the necessary modifications, first in the case of a convolution potential, and then in the case of a multiplicative potential. The proofs of the technical lemma will be postponed to Appendix~\ref{appendixproof}.
\subsection{Strategy of the proofs}
We will thus closely follow the argument in \cite{LWX} which dealt with the periodic quintic NLS without potentials. This argument relies on the so-called ``second generation $I$-method'', introduced in \cite{CKSTT1,CKSTT2,CKSTT3,CKSTT4} and widely applied to nonlinear dispersive equations both on tori or on Euclidean spaces; see for example \cite{Bou04b,DSPST1,DSPST2,DSPST3} and references therein. The (classical) $I$-method exploits the almost conservation of the modified energy $E(I_Nu)$ for some appropriate choice of parameter $N\gg1$, where $E$ is the standard energy functional associated with the Hamiltonian structure of NLS, and $I_N$ is a smooth Fourier multiplier which behaves like the identity on frequencies smaller than $N$ and like a smoothing operator of order $1-s$ for higher frequencies, $s<1$ being a regularity where local well-posedness holds. Indeed, the standard energy $E$ is conserved but cannot be used at this level of regularity. Then the core of the argument is to prove that $E(I_Nu)$ is almost conserved, in the sense that its time derivative decays sufficiently fast with $N$. Moreover, a faster decay yields a smaller threshold for the admissible regularity on the initial data. Interestingly, regarding the case of the periodic quintic NLS without potential, in order to extend globally the local solutions of \cite{Bou93} when the regularity is $s<1$, Bourgain \cite{Bou04b} combined the $I$-method with Birkhoff normal form transformations in order to get a better decay of the time derivative of $E(I_Nu)$. This idea was also exploited in \cite{CKO} where the authors implemented the \textit{upside-down} $I$-method together with Birkhoff normal forms to study the growth of Sobolev norms $H^s(\mathbb{T})$ for $s>1$; namely, replacing the fractional integration operator $I_N$ of order $1-s$ by a fractional \textit{derivative} $D_N$ of order $s-1$. However, in the aforementioned seminal paper \cite{Bou04b}, the argument could only deal with regularity $s^\ast<s<1$ for some $s^\ast$ smaller but very close to $\frac12$. It was later pointed out in \cite{DSPST2,LWX} that one can improve on the range of $s$ by letting aside the Birkhoff normal form transformation, and by resorting instead to both rescaling and a ``second'' modified energy. The former point relies on the observation that in the Euclidean space $\mathbb{R}^d$, one has an improved bilinear estimate for the Schr\"odinger flow of type 
\begin{equation}\label{bilinear}
\big\|e^{it\Delta}f_{N_1}e^{it\Delta}g_{N_2}\big\|_{L^2_{t,x}}\lesssim N_2^{-\frac12}N_1^{\frac{d-1}{2}}\|f_{N_1}\|_{L^2}\|g_{N_2}\|_{L^2}
\end{equation}
for any $N_1\ll N_2$ and functions $f_{N_1}$ (resp. $g_{N_2}$) whose Fourier transform is supported in the region $\{|\xi|\sim N_1\}$ (resp. $\{|\xi|\sim N_2\}$). This in turn provides some gain of negative powers of $N$ in the context of the $I$-method when estimating multilinear interactions where one input function has dominant frequencies. The refined bilinear estimate above is however known to be \textit{false} on $\mathbb{T}^d$. One of the crucial observation in \cite{CKSTT4,DSPST2,LWX} is that, after a proper rescaling to work on a very large torus, one can still get an estimate which gets closer to \eqref{bilinear} and allows to get some decay in $N$; see \cite[Proposition 2.1]{LWX} and Lemma~\ref{lem:bil} below. The latter point is the introduction of correcting terms in the modified energy, which cancel the interactions having less decay in the time derivative of $E(I_Nu)$. This actually amounts to performing one step of a \textit{Poincaré-Dulac} normal form transformation, but at the level of the energy functional instead of the equation; see for example the discussion in \cite{GKO}. Implementing this method in the context of \eqref{eq:NLS*}-\eqref{eq:NLS} yields the following results (see \eqref{Xsb*} below for the definition of the Bourgain type space $X^{s,b}_{\mathcal{L}}$).

\begin{proposition}\label{prop:GWP*}
Let $V\in \mathcal{F}L^\infty(\mathbb{T};\mathbb{C})$ with real-valued Fourier coefficients. Then for any $\frac25<s< \frac12$ there exists $\epsilon_0\in (0;1]$ such that for any $u_0\in H^s(\mathbb{T};\mathbb{C})$ with $\|u_0\|_{H^s}\le \epsilon_0$, there exists a unique global mild solution $u\in X_{-\partial_x^2+V\ast,\mathrm{loc}}^{s,\frac12+}$ with initial data $u(0)=u_0$ to \eqref{eq:NLS*}.
Moreover, we have the growth estimate
\begin{equation}\label{global}
\|u(t)\|_{H^s}\lesssim_s  C(t,\|u_0\|_{H^s})\|u_0\|_{H^s},~~t\in\mathbb{R},
\end{equation}
where
\begin{align*}
C(t,\|u_0\|_{H^s}) =\begin{cases} \langle t\rangle^\frac12,~~|t|\lesssim \|u_0\|_{H^s}^{-1};\\ \|u_0\|_{H^s}^{-1},~~\|u_0\|_{H^s}^{-1}\lesssim |t|\lesssim \|u_0\|_{H^s}^{-\frac{3-\epsilon}{1-s}}\\ |t|^{\frac{1-s}{\alpha(s)}}\|u_0\|_{H^s}^{\frac{2(1-s)}{s\alpha(s)}},~~|t|\gtrsim \|u_0\|_{H^s}^{-\frac{3-\epsilon}{1-s}}.\end{cases}
\end{align*}
with $\alpha(s)=3-2\frac{1-s}{s}-\epsilon$ for some $0<\epsilon\ll_s 1$.
\end{proposition}
In particular, note that the second regime gives the bound $$\|u(t)\|_{H^s}\lesssim 1 \lesssim |t|\|u_0\|_{H^s}.$$
Thus \eqref{global} implies the more standard growth estimate\footnote{Recall that $\alpha(s)$ increases in $s$, with $\alpha(\frac25+)=0+$ and $\alpha(\frac12)=1-$.}
\begin{align}\label{growth*}
\|u(t)\|_{H^s}\lesssim \langle t\rangle^{\beta_s}\|u_0\|_{H^s}
\end{align}
with $$\beta_s=\max(\frac{1-s}{\alpha(s)},1).$$
\begin{remark}~\\
\textup{(i)} Proposition~\ref{prop:GWP*} only deals with the case $s< \frac12$. This is not restrictive since the point in this paper is to run the Birkhoff normal form at regularity $s< \frac12$, the standard theory covering the case $s>\frac12$.\\
\textup{(ii)} We point out again that we only consider initial data satisfying $\|u_0\|_{H^s}\le 1$ since this will be the case to apply the Birkhoff normal form. But in the defocusing case $\sigma>0$, the same result as in Proposition~\ref{prop:GWP*} holds for any initial data. Moreover, the above remarks also apply to Proposition~\ref{prop:GWPW} below.
\end{remark}

We have a similar result in the case of a multiplicative potential.
\begin{proposition}\label{prop:GWPW}
Let $W\in H^4(\mathbb{T};\mathbb{R})$ be even. Then for any $\frac25<s\le \frac12$ there exists $\epsilon_0\in(0;1]$ such that for any $u_0\in H^s(\mathbb{T};\mathbb{C})$ with $\|u_0\|_{H^s}\le \epsilon_0$, there exists a unique global mild solution $u\in X_{-\partial_x^2+W,\mathrm{loc}}^{s,\frac12+}$ with initial data $u(0)=u_0$ to \eqref{eq:NLS}.
Moreover, we have the growth estimate
\begin{equation}\label{growth}
\|u(t)\|_{H^s}\lesssim_s \langle t\rangle^{\beta_s}\|u_0\|_{H^s},~~t\in\mathbb{R}.
\end{equation}
\end{proposition}

The proof of Propositions~\ref{prop:GWP*} and~\ref{prop:GWPW} will occupy the rest of this section. There are some slight but essential modifications compared to the argument in \cite{LWX} in order to prove the results above. First, for the local well-posedness theory, in order to get existence of local solutions beyond times of order $O((\|V\|_{\ell^\infty}+\|W\|_{L^2})^{-1})$ for small initial data (the case we are interested in), we need to remove the linear terms from the nonlinearity by incorporating them in the linear operator. This requires to prove that the $X^{s,b}$ spaces adapted to lower order perturbations of $-\partial_x^2$ still have the same properties (see Lemma~\ref{lem:linear*} and~\ref{lem:linearW} below), in particular Strichartz estimates. This also explains why we deal with \eqref{eq:NLS*} and~\eqref{eq:NLS} separately in Propositions~\ref{prop:GWP*} and~\ref{prop:GWPW}.

As for the globalization part, the extra linear terms $V\ast u$ and $Wu$ are also dealt with differently. The potential term $Wu$ is treated as a perturbation term with respect to the nonlinearity. Indeed, as detailed below, the rescaling performed in the argument of \cite{LWX} is very favourable on the potential and one gains for free a factor $O(\lambda^{-2})$ when estimating terms with $W$ in the time variations of the modified energy, where $\lambda\sim N^{\frac{1-s}{s}}$ is the scaling factor. This is not quite enough though, as the lower bound $s>\frac25$ in \cite{LWX} comes from the use of a modified energy where one gains a factor $N^{-3}$. But since the operator $I_N$ at the base of the $I$-method does not commute with $W$, terms with $W$ only appear in commutators where one can gain an extra $N^{-1}$ factor at the expense of requiring more regularity for $W$, which we can afford. In comparison, we do not incorporate it in the linear operator as the Dirichlet and Neumann eigenfunctions of the Sturm-Liouville operator $-\partial_x^2+W_{|[0;\pi]}$ on $[0;\pi]$, though localized on complex exponentials (see \eqref{est:ef} and \eqref{est:ef-tri} below), only satisfy convolution relations $\int_\mathbb{T}f_{k_1}f_{k_2}f_{k_3}\approx \delta_{k_1+k_2=k_3}$ up to error terms. This is good enough to handle trilinear interactions. However, the argument in \cite{LWX} is really tailored to the precise restrictions on frequencies in the multilinear forms, and in particular to the null moment condition $k_1+\dots+k_6=0$. But this is destroyed when replacing $-\partial_x^2$ by $-\partial_x^2+W$ (see e.g. \eqref{moment} above), and so it is not clear to us how to implement the second generation $I$-method for the quintic equation without treating the term $Wu$ as part of the nonlinearity as we do here.
 
On the contrary, the convolution potential $V$ scales as $W$ but commutes with $I_N$. Thus, we need to view it again as part of the linear operator and use the argument in \cite{LWX} with $-\partial_x^2+V\ast$ in place of $-\partial_x^2$. This brings small changes in all the parts of the argument which rely on the particular form of the symbol of $-\partial_x^2$. 

\subsection{Proof of Proposition~\ref{prop:GWP*}}
\subsubsection{Notations and local well-posedness}
We first recall some notations. We will build the solution $u$ to \eqref{eq:NLS*} in the Bourgain type space $X^{s,b}_{-\partial_x^2+V\ast}$ adapted to the linear Schr\"odinger equation with convolution potential, namely the Banach space defined through the norm
\begin{equation}\label{Xsb*}
\|u\|_{X^{s,b}} := \big\|\langle \omega_k\rangle^{\frac{s}{2}}\langle \tau - \omega_k\rangle^b \hat{u}_k(\tau)\big\|_{L^2_\tau\ell^2_k},
\end{equation}
for any $s,b\in\mathbb{R}$, where $\hat{u}_k(\tau) = \int_{\mathbb{R}\times\mathbb{T}}u(t,x)e^{-i(t\tau+kx)}dxdt$ is the space-time Fourier transform of $u : (t,x)\in\mathbb{R}\times\mathbb{T}\mapsto u(t,x)\in \mathbb{C}$, and
\begin{align*}
\omega_k = k^2+(2\pi)^{\frac12}V_k.
\end{align*} 
Recall that we assume that the $V_k$ are real, with $\|V\|_{\ell^\infty}:=\sup_{k\in\Z}|V_k|<\infty$.

In the rest of this section we will drop the subscript $-\partial_x^2+V\ast$ in the notation of the Bourgain type space since there is no risk of confusion. We also define its time-localized version
\begin{equation}\label{Xsb*T}
\|u\|_{X^{s,b}(T)}=\inf\{\|v\|_{X^{s,b}},~~v\equiv u\text{ on }[-T,T]\}.
\end{equation}
Next we collect some linear estimates in the space $X^{s,b}$.
\begin{lemma}\label{lem:linear*}
The following properties hold:\\
\textup{(i) ($X^{s,b}$ as a resolution space)} Let $s\in\mathbb{R}$ and $b>\frac12$. Then there exists $C=C(s,b,1+\|V\|_{\ell^\infty})>0$ such that for any $u\in X^{s,b}$ it holds $u\in C(\mathbb{R};H^s(\mathbb{T}))$ and $\|u\|_{L^\infty_tH^s}\le C\|u\|_{X^{s,b}}$.\\
\textup{(ii) (Time localization)} Let $s\in\mathbb{R}$ and $-\frac12 <b'\le b <\frac12$. Then there exists $C=C(s,b,1+\|V\|_{\ell^\infty})>0$ such that for any $T\in(0;1]$ and $u\in X^{s,b}(T)$, it holds $\|u\|_{X^{s,b'}(T)}\le C T^{b-b'}\|u\|_{X^{s,b}(T)}$.\\
\textup{(iii) (Linear estimate)} For any $s\in\mathbb{R}$, $b>\frac12$, there is $C(s,b,1+\|V\|_{\ell^\infty})>0$ such that $$\|e^{it(-\partial_x^2+V\ast)}u_0\|_{X^{s,b}(T)}\leq C \langle T\rangle^\frac12\|u_0\|_{H^s}$$ for any $T>0$ and $u_0\in H^s(\mathbb{T})$.\\
\textup{(iv) (Energy estimate)} For any $s\in\mathbb{R}$, $b>\frac12$, there is $C(s,b,1+\|V\|_{\ell^\infty})>0$ such that it holds $$\Big\|\int_0^te^{i(t-t')(-\partial_x^2+V\ast)}F(t')dt'\Big\|_{X^{s,b}(T)}\leq C\langle T\rangle^2\|F\|_{X^{s,b-1}}$$ for any $T>0$ and $F\in X^{s,b-1}$.\\
\textup{(v) ($L^4$ Strichartz estimate)} There is $C(1+\|V\|_{\ell^\infty})>0$ such that it holds $\|u\|_{L^4_{t,x}}\le C \|u\|_{X^{0,\frac38+}}$.\\
\textup{(vi) (Equivalence of norms)} For any $s,b\in\R$, there exists $C(s,b,1+\|V\|_{\ell^\infty})\ge 1$ such that it holds $\frac1C\|u\|_{X^{s,b}_{-\partial_x^2}}\le \|u\|_{X^{s,b}_{-\partial_x^2+V\ast}}\le C\|u\|_{X^{s,b}_{-\partial_x^2}}$ for any $u\in X^{s,b}_{-\partial_x^2+V\ast}$.
\end{lemma}
The general properties (i), (ii), and (v) of $X^{s,b}$ spaces in Lemma~\ref{lem:linear*} in the case $V=0$ are standard, and we refer to \cite{Tao06} and \cite{ET}. See Appendix~\ref{appendix} for the modifications in the case $V\neq0$. Note however that the estimates (iii) and (iv), due to the factors $\langle T\rangle^\frac12$ and $\langle T\rangle^\frac32$, are somewhat less standard; see also Remark~\ref{rk:LWP} below and Appendix~\ref{appendix} for the proofs. In the case $V=0$, the Strichartz estimate (v) is due to Bourgain \cite{Bou93}. The last estimate (vi) shows that $X^{s,b}$ norms with respect to $V=0$ or $V\neq 0$ are equivalent, which in particular implies properties (i), (ii), and (v) for the case $V\neq 0$ from the classical case $V=0$. Let us also emphasize that in order to close the fixed point argument, one needs to get a small factor of $T$. In this perspective, (ii) is used for the large data local theory, whereas (iv) when $T>1$ is more suited for the small data theory.

 As a consequence of the estimates in Lemma~\ref{lem:linear*}, we then have the following small\footnote{Here we only treat the case of small initial data as this is the setting for the proof of Theorem~\ref{thm:conv}. Of course a similar local well-posedness result holds for large data, with a different time of existence, and with a proof which relies on Lemma~\ref{lem:linear*}~(ii) on top of Lemma~\ref{lem:linear*}~(iv) to get a small power of $T\in (0;1]$.} data local well-posedness result for \eqref{eq:NLS*}.
\begin{lemma}\label{lem:LWP*}
Let $V\in \mathcal{F}L^\infty(\mathbb{T};\mathbb{C})$ with real-valued Fourier coefficients, and $s>\frac14$. Then there exists $\epsilon_0\in(0;1]$ such that for any $u_0\in H^s(\mathbb{T};\mathbb{C})$ with $\|u_0\|_{H^s}\le \epsilon_0$, letting $\delta\sim \|u_0\|_{H^s}^{-1}\ge 1$, there exists a unique mild solution $u\in X^{s,\frac12+}(\delta)$ to \eqref{eq:NLS*} on $[-\delta,\delta]\times\mathbb{T}$ with $u(0)=u_0$, which satisfies
\begin{equation}\label{local*}
\|u\|_{X^{s,\frac12+\epsilon}(t)}\lesssim \langle t\rangle^\frac12\|u_0\|_{H^s}
\end{equation}
for any $0\le t\le \delta$.
\end{lemma}
\begin{remark}\label{rk:LWP}\rm~~\\
(i) The point in the previous lemma is the dependence of the local time $\delta$ on $\|u_0\|_{H^s}$. Note that the estimates (iii) and (iv) in Lemma~\ref{lem:linear*} are somehow non-standard, since at scaling subcritical regularity one is usually concerned with local well-posedness for large initial data, for which the local time in the fixed point argument is taken to be $\delta\le 1$. Here we have a time of existence $\delta\sim \|u_0\|_{H^s}^{-1}\ge 1$, which is crucial to get the growth estimate \eqref{growth*} needed to control the high modes in the proof of Theorem~\ref{thm:conv}.\\
(ii) Note also that on the local time of existence, we can only get the local estimate \eqref{local*} with a loss of the factor $\langle t\rangle^\frac12$ due to the linear estimate (iii) in Lemma~\ref{lem:linear*}, which is sharp. This is different from what happens at higher regularity ($s>\frac{d}2$) where one can perform a fixed point argument directly in $C([-\delta;\delta];H^s(\mathbb{T}))$ without having to use $X^{s,b}$ spaces, which in particular provides a longer local time $\delta \sim \|u_0\|_{H^s}^{-4}$ and an estimate $\|u\|_{L^\infty_\delta H^s}\lesssim \|u_0\|_{H^s}$. In particular, \eqref{localI*} only provides stability of Fourier modes $\big||u_k(t)|^2-|u_k(0)|^2\big|\ll \|u_0\|_{H^s}^2$ up to time $O(1)$ instead of times $O(\|u_0\|_{H^s}^{-4})$ compared to the local well-posedness theory at regularity $s>\frac{d}2$. This loss of $\langle t\rangle^\frac12$ may be avoided by using refined versions of $X^{s,b}$ spaces ($U^2/V^2$ type spaces) used in the Cauchy theory at scaling critical regularity, but we do not pursue this refinement here, as \eqref{growth} suffices for our purpose.
\end{remark}
\begin{proof}[Proof of Lemma~\ref{lem:LWP*}]
Note that here we do no try to cover the best possible range for the regularity of the initial data, namely $s>0$, since we will be restricted to the range $s>\frac25$ by the globalization argument.
The proof would follow from a straightforward adaptation of Bourgain's argument \cite{Bou93}; see also \cite{ET}. Let $\epsilon_0\in(0;1]$, $\|u_0\|_{H^s}\le \epsilon_0$, $\delta=A\|u_0\|_{H^s}^{-1}\ge 1$ for some $A\ge 1$, and $R=2C \delta^\frac12 \|u_0\|_{H^s}$ for some $C>0$, and let $B(R)$ be the ball of radius $R$ in $X^{s,\frac12+}(\delta)$. Setting
\begin{equation*}
\Gamma : u \in B(R)\mapsto e^{it(-\partial_x^2+V\ast)}u_0 -i\sigma\int_0^te^{i(t-t')(-\partial_x^2+V\ast)}(|u|^4u)(t')dt',
\end{equation*}
we will prove that $\Gamma$ is a contraction on $B(R)$ for $\delta$ appropriately chosen. Let then $u\in B(R)$, and take $v$ to be an extension of $u$ such that $\|v\|_{X^{s,\frac12+}}\le 2\|u\|_{X^{s,\frac12+}(\delta)}$. First, using Cauchy-Schwartz inequality, we have the Sobolev type estimate
\begin{align*}
\|v\|_{L^\infty_{t,x}}\leq \|\hat{v}_k(\tau)\|_{L^1_\tau L^1(dk)_\lambda} \leq \|\langle\tau-k^2\rangle^{-\frac12-}\langle k\rangle^{-\frac12-}\|_{L^2_\tau\ell^2(dk)_\lambda}\|v\|_{X^{\frac12+,\frac12+}}\lesssim \|v\|_{X^{\frac12+,\frac12+}}.
\end{align*}
Interpolating this estimate with the $L^4$-Strichartz estimate of Lemma~\ref{lem:linear*}~(v), we get
\begin{align}\label{L8*}
\|v\|_{L^8_{t,x}}\lesssim \|v\|_{X^{\frac14+,\frac7{16}+}}.
\end{align}
Then, we first use Lemma~\ref{lem:linear*}~(iii) and (vi) to estimate
\begin{align*}
\big\|\Gamma u\big\|_{X^{s,\frac12+}(\delta)}&\lesssim \delta^\frac12\|u_0\|_{H^s} + \delta^2\big\||v|^4v\big\|_{X^{s,-\frac12+}}\intertext{Next, using the dual version of the $L^4$ Strichartz estimate of Lemma~\ref{lem:linear*}~(v), we can continue with}
&\lesssim \delta^\frac12\|u_0\|_{H^s}+\delta^2\big\||v|^4v\big\|_{L^\frac43_t W^{s,\frac43}}.
\end{align*}
Now, by the fractional Leibniz rule (see e.g. \cite[Lemma 1.11]{ET}), it holds
\begin{align*}
\big\||v|^4v\big\|_{L^\frac43_t W^{s,\frac43}}&\lesssim \Big\|\|v\|_{W^{s,4}_x}\big\||v|^4\big\|_{L^2_x}\Big\|_{L^\frac43_t} \lesssim \|v\|_{L^4_tW^{s,4}}\|v\|_{L^8_{t,x}}^4,
\end{align*}
where the second step follows from H\"older's inequality. Using now \eqref{L8*}, we finally get
\begin{align*}
\big\|\Gamma u\big\|_{X^{s,\frac12+}(\delta)}&\lesssim \delta^\frac12\|u_0\|_{H^s}+\delta^2\|v\|_{X^{s,\frac12+}}^5 \le C\delta^\frac12\|u_0\|_{H^s}+C\delta^2\|u\|_{X^{s,\frac12+}(\delta)}^5 .
\end{align*}
Therefore, from $\delta=A\|u_0\|_{H^s}^{-1}$, $R=2C\delta^\frac12\|u_0\|_{H^s}=2CA^\frac12\|u_0\|_{H^s}^\frac12$, and $\|u\|_{X^{s,\frac12+}(\delta)}\le R$,  taking $\epsilon_0\in(0;1]$ such that
$$2CA^\frac12\epsilon_0^\frac12\le 1,$$
we have $R\le 1$. Then choosing $A$ such that
$$C\delta^2R^4=2^4C^5\delta^4\|u_0\|_{H^s}^4=2^4C^5A^4<\frac12,$$
  we get that $\Gamma$ maps $B(R)$ to $B(R)$. The contraction property follows from similar estimate. This proves \eqref{local*}.
\end{proof}

\subsubsection{Rescaling}\label{sub:scal}
In order to implement the $I$-method to globalize the local solution provided by Lemma~\ref{lem:LWP*}, recall that for $N\gg 1$ to be chosen later, the $I$ operator is defined as the Fourier multiplier with symbol $m(k)=m_s(N^{-1}k)$ for some smooth even function $m_s$ which equals 1 on $[0;1]$ and behaves like $\langle k\rangle^{s-1}$ for $|k|\ge 1$. In particular, for $u_0\in H^s(\mathbb{T};\mathbb{C})$, we have $I_Nu_0\in H^1(\mathbb{T};\mathbb{C})$ and it holds
\begin{align}\label{eq:HsI}
\|u_0\|_{H^s}\lesssim \|I_Nu_0\|_{H^1}\lesssim N^{1-s}\|u_0\|_{H^s}.
\end{align}

However, as we mentioned above, in order to benefit from the improved bilinear Strichartz estimate and get a better decay of the modified energy $E(I_Nu)$, we will use a rescaling procedure. Indeed, recall that \eqref{eq:NLS*} has the following scaling property : $u(t,x)$ solves \eqref{eq:NLS*} on $[-T,T]\times\mathbb{T}$
if and only if
$$u_{\lambda}(t,x)=\lambda^{-\frac12}u(\lambda^{-2}t,\lambda^{-1}x)$$
is a solution of 
\begin{equation}\label{eq:NLSscal*}
i\partial_t u_\lambda = -\partial_x^2 u_\lambda + V_\lambda\ast u_\lambda +\sigma|u_\lambda|^4u_\lambda
\end{equation} 
on $[-\lambda^2T,\lambda^2T]\times\mathbb{T}_\lambda$, where $\mathbb{T}_\lambda = \mathbb{R}/(2\pi\lambda)\mathbb{Z}$ and
\begin{align}\label{eq:V}
V_\lambda(x)=\lambda^{-3}V(\lambda^{-1}x),~~x\in (-\pi\lambda,\pi\lambda).
\end{align}

Following \cite{CKSTT3,DSPST1,LWX}, let us then recall some properties of $\lambda-$periodic functions. Define $(dk)_{\lambda}$ to be the normalized counting measure on $\frac{1}{\lambda}\mathbb{Z}$:
$$\int a(k)(dk)_{\lambda}=\frac{1}{{\lambda}}\sum_{k \in \frac{1}{\lambda}{\mathbb{Z}}}a(k).$$
We define the Fourier transform of $u\in L^1([0;2\pi\lambda];\mathbb{C})$ by
\begin{align}\label{Fourier}
\hat{u}(k)=(2\pi)^{-\frac12}\int_0^{2\pi\lambda} e^{-ikx}u(x)dx,
\end{align}
$k\in\frac1\lambda\mathbb{Z}$. Fourier inversion formula reads
$$
u(x)=(2\pi)^{-\frac12}\int e^{ikx}\hat{u}(k)(dk)_{\lambda},
$$
and the following identities are true:
\begin{enumerate}
\item $\|u\|_{L^{2}([0;2\pi\lambda])}=\|\hat{u}\|_{L^{2}((dk)_{\lambda})}$, (Plancherel)
\item $\int_0^{2\pi\lambda} u(x)\bar{v}(x)dx=\int \hat{u}(k)\bar{\hat{v}}(k) (dk)_{\lambda}$, (Parseval)
\item $\widehat{uv}(k)=\hat{u}\star _{\lambda} \hat{v}(k)=(2\pi)^{-\frac12}\int \hat u(k-k_{1})\hat v(k_{1})(dk_{1})_{\lambda}$.
\end{enumerate}

The Sobolev space of $\lambda$-periodic functions is $H^{s}_\lambda = H^{s}([0;2\pi\lambda];\mathbb{C})$ defined by the norm
$$\|u\|_{H^{s}_\lambda}=\|\langle k \rangle^{s}\hat{u}(k)\|_{L^{2}((dk)_{\lambda})}.$$
Then \eqref{eq:HsI} implies
\begin{align}\label{eq:HsIscal}
\|u_\lambda(0)\|_{H^s_\lambda}\lesssim \|I_Nu_\lambda(0)\|_{H^1_\lambda}\lesssim N^{1-s}\|u_\lambda(0)\|_{H^s_\lambda}\lesssim N^{1-s}\lambda^{-s}\|u_0\|_{H^s}.
\end{align}

 We also denote by $X^{s,b}_\lambda = X^{s,b}(\mathbb{T}_{\lambda} \times \mathbb R;\mathbb{C})$ the Bourgain type space of complex-valued space-time functions $\lambda$-periodic in $x$ associated to $-\partial_x^2+V_\lambda\ast$, endowed with

$$\|u\|_{X^{s,b}_\lambda}=\|\langle k \rangle^{s} \langle \tau -\omega_{k,\lambda} \rangle^{b}\hat{u}_k(\tau)\|_{L_{\tau}^{2}L_{(dk)_{\lambda}}^{2}},$$
where as above $\hat{u}_k(\tau)$ is the space-time Fourier transform, and 
\begin{align}\label{eq:omegalambda}
\omega_{k,\lambda}=k^2+(2\pi)^{\frac12}(V_\lambda)_k = k^2+(2\pi)^{\frac12}\lambda^{-2}V_{\lambda k}
\end{align}
for $k\in\frac1\lambda\mathbb{Z}$.

Before starting to get long-time bounds on $I_Nu_\lambda$, we recast the local estimate \eqref{local*} in terms of $I_Nu_\lambda$.

\begin{lemma}\label{lem:LWPI*}
Let $V\in \mathcal{F}L^\infty(\mathbb{T};\mathbb{C})$ with real-valued Fourier coefficients, $\lambda\ge 1$, $s>\frac14$ and $u_\lambda(0)\in H^s_\lambda$. Then for 
$$ \delta\sim_{\|V\|_{\ell^\infty}}\langle\|I_Nu_\lambda(0)\|_{H^1_\lambda}\rangle^{-8-},$$ it holds
\begin{equation}\label{localI*}
\|I_Nu_\lambda\|_{X^{1,\frac12+}_\lambda(\delta)} \lesssim_{\|V\|_{\ell^\infty}} \|I_Nu_\lambda(0)\|_{H^1_\lambda}.
\end{equation}
\end{lemma}
\begin{proof}
Again, we do not try to cover the whole range of regularity $s>0$ where local well-posedness of \eqref{eq:NLS*} holds. The proof of \eqref{localI*} relies on the fact that, since $u_\lambda$ solves \eqref{eq:NLSscal*}, $I_Nu_\lambda$ solves
\begin{align}\label{eq:NLSI*}
i\partial_t I_Nu_\lambda = -\partial_x^2I_Nu_\lambda + V_\lambda\ast (I_Nu_\lambda) + \sigma I_N\big(|u_\lambda|^4u_\lambda\big)
\end{align}
with initial data $I_Nu_\lambda(0)\in H^1(\mathbb{T}_\lambda)$. Thus, we will use estimates similar to that in the proof of Lemma~\ref{lem:LWP*}. Note that the estimates (i)--(iv) of Lemma~\ref{lem:linear*}  are unchanged for $\lambda$-periodic functions, uniformly in $\lambda\ge 1$ due to the dependence of the constants in $1+\|V_\lambda\|_{\ell^\infty}\sim 1+\lambda^{-2}\|V\|_{\ell^\infty}$. This is also the case for Lemma~\ref{lem:linear*}~(v) (see for example \cite{DSPST1} or Appendix~\ref{appendix} below), which will be enough for our purpose as mentioned above. Indeed, note first that we have the Sobolev type inequality
\begin{align*}
\|u\|_{L^\infty_{t,x}}\le \|\hat{u}_k(\tau)\|_{L^1_\tau L^1_{(dk)_\lambda}}\le \|\langle k\rangle^{-\frac12+}\langle\tau-\omega_k\rangle^{-\frac12+}\|_{L^2_\tau L^2_{(dk)_\lambda}}\|u\|_{X^{\frac12+,\frac12+}_\lambda}.
\end{align*}
Interpolating this bound with the $L^4$ Strichartz estimate of Lemma~\ref{lem:linear*}~(v) gives
\begin{align*}
\|u_\lambda\|_{L^8_{\delta,x}}\lesssim \|u_\lambda\|_{X^{\frac14+,\frac7{16}+}_\lambda(\delta)}
\end{align*}
for any $\delta>0$. Together with Lemma~\ref{lem:linear*}, duality, H\"older's inequality and the fractional Leibniz rule, this yields again (note that $\delta\le 1$ now)
\begin{align}\label{multiI}
\big\||u_\lambda|^4u_\lambda\big\|_{X^{s,-\frac12+}_\lambda(\delta)}&\lesssim \delta^{\frac18-}\big\||u_\lambda|^4u_\lambda\big\|_{X^{s,-\frac38-}_\lambda(\delta)}\lesssim \delta^{\frac18-}\big\||u_\lambda|^4u_\lambda\big\|_{L^\frac43_\delta W^{s,\frac43}} \notag\\
&\lesssim \delta^{\frac18-}\|u_\lambda\|_{L^4_{\delta,x}}\|u\|_{L^8_{\delta,x}}^4 \lesssim \delta^{\frac18-}\|u_\lambda\|_{X^{s,\frac38+}_\lambda(\delta)}\|u\|_{X^{\frac14+,\frac7{16}+}_\lambda(\delta)}^4\\
&\lesssim \delta^{\frac12-}\|u_\lambda\|_{X^{s,\frac12+}_\lambda(\delta)}^5
\end{align}
for $s>\frac14$.

Now, from the mild formulation of \eqref{eq:NLSI*} and Lemma~\ref{lem:linear*}, we have
\begin{align*}
\|I_Nu_\lambda\|_{X^{1,\frac12+}_\lambda(\delta)} &\lesssim \|I_Nu_\lambda(0)\|_{H^1_\lambda}+\big\|I_N(|u_\lambda|^4u_\lambda)\big\|_{X^{1,-\frac12+}_\lambda(\delta)}.
\end{align*}
Setting 
\begin{align*}
U(t,x)=\int_\mathbb{R}\int e^{i(t\tau + kx)}m(k)\langle k\rangle^{1-s}|\hat{u}_k(\tau)|(dk)_\lambda d\tau,
\end{align*}
we can estimate the last term above by
\begin{align*}
\big\|I_N(|u_\lambda|^4u_\lambda)\big\|_{X^{1,-\frac12+}_\lambda(\delta)}&\le \big\|M_5(U)\big\|_{X^{s,-\frac12+}_\lambda(\delta)}
\end{align*}
where the multilinear operator is
\begin{align*}
\widehat{M_5(U)}_k(\tau)=\int_{\tau_1+\cdots+\tau_5=\tau}\int_{k_1+\dots+k_5=k}\frac{m_N(k)\langle k\rangle^{1-s}}{\prod_{j=1}^5m_N(k_j)\langle k_j\rangle^{1-s}}\prod_{j=1}^5 \widehat{U_{k_j}}(\tau_j)(dk_j)_\lambda d\tau_j.
\end{align*}
But since 
\begin{align*}
m_N(k)\langle k\rangle^{1-s}\sim\begin{cases}\langle k\rangle^{1-s},~~|k|\lesssim N\\N^{1-s},~~|k|\gg N\end{cases},
\end{align*}
in particular 
\begin{align*}
m_N(k_1+\cdots+k_5)\langle k_1+\cdots+k_5\rangle^{1-s}\lesssim \sum_{j=1}^5m_N(k_j)\langle k_j\rangle^{1-s}
\end{align*}
and thus the symbol of $M_5$ is bounded uniformly in $\lambda,N$. Together with \eqref{multiI}, this yields
\begin{align*}
\big\|I_N(|u_\lambda|^4u_\lambda)\big\|_{X^{1,-\frac12+}_\lambda(\delta)}&\lesssim \big\|M_5(U)\big\|_{X^{s,-\frac12+}_\lambda(\delta)} \lesssim \big\||U|^4U\big\|_{X^{s,-\frac12+}_\lambda(\delta)}\\ &\lesssim \delta^{\frac12-}\|U\|_{X^{s,\frac12+}_\lambda(\delta)}^5=\delta^{\frac12-}\|I_Nu_\lambda\|_{X^{1,\frac12+}_\lambda(\delta)}^5.
\end{align*}
All in all, we get
\begin{align*}
\|I_Nu_\lambda\|_{X^{1,\frac12+}_\lambda(\delta)} &\lesssim \|I_Nu_\lambda(0)\|_{H^1_\lambda}+ \delta^{\frac12-}\|I_Nu_\lambda\|_{X^{1,\frac12+}_\lambda(\delta)}^5.
\end{align*}
Thus \eqref{localI*} follows from the previous estimate with our choice of $\delta$ as in the proof of Lemma~\ref{lem:LWP*}.
\end{proof}

\subsubsection{Modified energy and globalization}
Now we set up the $I$-method for rescaled functions. Let $T\gg 1$ be a target time of existence, and $N=N(T)\gg 1$ to be chosen later. Note that it is enough to consider the case $T\ge \|u_0\|_{H^s}^{-1}$, the other case being dealt with by the local theory (Lemma~\ref{lem:LWP*}). Then in view of \eqref{eq:HsIscal}, we take $$\lambda \sim N^{\frac{1-s}{s}}\|u_0\|_{H^s}^{\frac1s},$$ so that $\|I_Nu_\lambda(0)\|_{H^1}\sim 1$, and thus \eqref{localI*} holds with $\delta\sim 1$ by Lemma~\ref{lem:LWPI*}. From now on, we drop the subscript $\lambda$.

The ``first generation'' $I$-method then corresponds to the use of the modified energy
\begin{align*}
E_1(u)&=E(I_Nu)=\frac12\|\partial_xI_N u\|_{L^2}^2+\frac{\gamma}{2}\|I_Nu\|_{L^2}^2+\frac12\int_{\mathbb{T}_\lambda}\big(V_\lambda\ast I_Nu\big) \cdot \overline{I_Nu}dx+\frac\sigma6\|I_Nu\|_{L^6}^6.
\end{align*}
Here the mass $\gamma>0$ is chosen such that $\gamma>(2\pi)^{\frac12}\lambda^{-2}\sup_k|V_k|$, which ensures that 
\begin{align}\label{omegatilde}
\omega_k+\gamma \ge c>0
\end{align}
and so that $E_1$ controls $\|\cdot\|_{H^1}^2$.
%
To estimate the time variations of the modified energy, let us recall some notations on multilinear forms from \cite{CKSTT3,DSPST2,LWX}. For $n\in\mathbb{N}$, $\Gamma_n \subset (\lambda^{-1}\mathbb{Z})^n$ denotes the space
$$ \Gamma_n := \{ (k_1,\ldots,k_n) \in (\lambda^{-1}\mathbb{Z})^n,~~ k_1 + \ldots + k_n = 0 \}.$$
For a smooth $M_n: \Gamma_n \to \mathbb{C}$, we define the $n$-linear functional
$$ \Lambda_n( M_n; u_1,\ldots,u_n ) := \int_{\Gamma_n} M_n(k_1,\ldots,k_n)\big(\prod_{j=1}^n (u_j)_{k_j}\big)(dk_1)_\lambda \cdots (dk_{n-1})_\lambda,$$
and for $n$ even we simply write
$$ \Lambda_n( M_n; u ) := \Lambda_n( M_n; u, \overline{u}, \ldots, u, \overline{u} ).$$
Then
\begin{align*}
E_1(u)=E(I_Nu)&= \frac1{2\lambda}\sum_{k\in\lambda^{-1}\Z}m_N(k)^2(\omega_k+\gamma)|u_k|^2\\
&\qquad\qquad+\frac\sigma{6\lambda^5}\sum_{\substack{k_1,\dots,k_6\in\lambda^{-1}\Z\\k_1+\dots+k_6=0}}\big(\prod_{\ell=1}^3m_N(k_{2\ell-1})u_{k_{2\ell-1}}\big)\big(\prod_{\ell=1}^3m_N(k_{2\ell})(\overline{u})_{k_{2\ell}}\big)\\&=\Lambda_2(\sigma_2^V;u)+\sigma\Lambda_6(\sigma_6;u),
\end{align*}
where here
\begin{align}\label{eq:sigma2V}
\sigma_2^V=\frac14m_N(k_1)m_N(k_2)(\omega_{k_1}+\omega_{-k_2}+2\gamma),
\end{align}
and as  in \cite{LWX}, $$\sigma_6=\frac16m_N(k_1)\cdots m_N(k_6).$$
Recall that $\omega_k$ are the eigenvalues of $-\partial_x^2+V_\lambda\ast$. In particular, when $V=0$, $$\sigma_2^0=\sigma_2=-\frac12m_N(k_1)m_N(k_2)k_1k_2$$ as in \cite{LWX}.

Writing \eqref{eq:NLSscal*} as
\begin{align*}
\begin{cases}
{\displaystyle \partial_t u_k = -i\big(k^2+(2\pi)^{\frac12}(V_\lambda)_k\big) u_k-i\frac{\sigma}{\lambda^5} \sum_{\substack{k_1,\dots,k_5\in\lambda^{-1}\Z\\k_1+\cdots+k_5=k}}u_{k_1}(\bar{u})_{k_2}u_{k_3}(\bar{u})_{k_4}u_{k_5},}\\
{\displaystyle \partial_t (\overline{u})_k = i\big(k^2+(2\pi)^{\frac12}(\overline{V_\lambda})_k\big) (\overline{u})_k+i\frac{\sigma}{\lambda^5} \sum_{\substack{k_1,\dots,k_5\in\lambda^{-1}\Z\\k_1+\cdots+k_5=k}}(\bar{u})_{k_1}u_{k_2}(\bar{u})_{k_3}u_{k_4}(\bar{u})_{k_5},}
\end{cases}
~~k\in\lambda^{-1}\Z,
\end{align*} 
a computation similar to (3.35) in \cite{LWX} gives that for any symbol $M_n$,
\begin{align*}
\frac{d}{dt}\Lambda_n(M_n)&=i\Lambda_n(M_n\alpha_n^V)+i\sigma\Lambda_{n+4}\big(\sum_{j=1}^n(-1)^jX_j(M_n)\big),
\end{align*}
where in our case (recall that $V_k$ are real)
\begin{align}\label{eq:alpha}
\alpha_n^V=-\sum_{j=1}^n(-1)^{j+1}\omega_{(-1)^{j+1}k_j},
\end{align}
and as in \cite{LWX}
\begin{align*}
X_j(M_n)=M_n(k_1,\ldots,k_{j-1},k_j+\cdots+k_{j+4},k_{j+5},\ldots,k_{n+4}).
\end{align*}
This yields
\begin{align*}
\frac{d}{dt}E_1(u)=i\Lambda_2(\sigma_2^V\alpha_2^V)+i\sigma\Lambda_6(M_6^V;u)+i\sigma\Lambda_{10}(M_{10}^V;u)
\end{align*}
with (using the symmetry with respect to both $\{k_1,k_3,k_5\}$ and $\{k_2,k_4,k_6\}$) 
\begin{align}\label{eq:M6V}
M_6^V=\frac{1}{6}\sum_{j=1}^6(-1)^{j+1}m(k_j)^2(\omega_{(-1)^{j+1}k_j}+\gamma)+\sigma_6\alpha_6^V=M_6^{V,1}+M_6^{V,2},
\end{align}
and
\begin{align*}
M_{10}^V=\sum_{j=1}^6(-1)^jX_j(\sigma_6^V).
\end{align*}

Note that, as in \cite{LWX},
\begin{align}\label{eq:alpha2V}
\alpha_2^V=0
\end{align}
on $\Gamma_2$.

To treat the main term $\Lambda_6(M_6^V)$, we still follow \cite{LWX} and make the non-resonant/resonant decomposition $\widetilde{M_6^V}=\mathbf{1}_\Omega M_6^{V,1} +\mathbf{1}_\Upsilon M_6^{V,2}$ and $\overline{M_6^V} = M_6^V-\widetilde{M_6^V}$. Here the sets of frequencies $\Omega,\Upsilon\subset \Gamma_6$ are the same as in \cite{LWX}, namely
\begin{align}\label{eq:Ups}
\Upsilon=\{(k_1,\ldots,k_6)\in\Gamma_6,~~|k_1|\ge|k_3|\ge|k_5|,~~|k_2|\ge|k_4|\ge|k_6|,~~|k_1|\ge|k_2|,~~|k_1^*|\sim|k_2^*|\gg N\},
\end{align}
\begin{align}\label{eq:O1}
\Omega_1=\{(k_1,\ldots,k_6)\in\Upsilon,~~|k_1|\gg |k_2|\},
\end{align}
\begin{align}\label{eq:O2}
\Omega_2 = \{(k_1,\ldots,k_6)\in\Upsilon,~~|k_3^*|\gg |k_4^*|\},
\end{align}
\begin{align}\label{eq:O3}
\Omega_3 = \{(k_1,\ldots,k_6)\in\Upsilon,~~|k_1|\sim|k_5|,~~|k_5|\gg|k_4|\},
\end{align}
\begin{align}\label{eq:O4}
\Omega_4 = \Big\{(k_1,\ldots,k_6)\in\Upsilon,~~|k_1|\sim|k_6|\gg|k_3|,~~\big(||k_1|-|k_2||\ll|k_1|\text{ or }k_2k_4>0,~~k_2k_6>0\big)\Big\},
\end{align}
and
\begin{align}\label{eq:O5}
\Omega_5 = \Big\{(k_1,\ldots,k_6)\in\Upsilon,~~|k_1|\sim|k_2|\gtrsim N\gg|k_3^*|,~~|\omega_{k_1}-\omega_{-k_2}|\gg\big|\sum_{j=3}^6(-1)^{j+1}\omega_{(-1)^{j+1}k_j}\big|\Big\},
\end{align}
and $\Omega=\cup_{j=1}^5\Omega_j$, where $|k_1^*|\ge\ldots\ge|k_6^*|$ denotes the decreasing rearrangement of $(k_1,\ldots,k_6)$.

\medskip
With the decomposition above, the ``second generation'' $I$-method consists then in using the modified energy
\begin{align*}
E_2(u) = E(I_Nu)-\Lambda_6\big(\frac{\widetilde{M_6^V}}{\alpha_6^V}\big).
\end{align*}

Indeed, from the conservation of the $L^2$-norm and \eqref{eq:HsIscal} it holds
\begin{align}\label{eq:modifiedenergy*}
\|u(t)\|_{H^s}&\lesssim \|u(0)\|_{L^2}+\|\partial_x I_Nu\|_{L^2}\lesssim \|u(0)\|_{H^s}+E_1(u(t))^{\frac12}\notag\\
& \lesssim \|u(0)\|_{H^s}+E_2(u(t))^{\frac12}+N^{0-}E_2(u(t))^3,
\end{align}
where the second estimate follows from the positivity of $E$ in the defocusing case and the Gagliardo-Nirenberg inequality with the smallness assumption on $\|u_0\|_{H^s}$ in the focusing case, and the last estimate is a direct consequence of the following lemma, which is the exact analogue of \cite[Lemma 3.3]{LWX} (see Appendix~\ref{appendixproof} for a proof).

\begin{lemma}\label{lem:E1E2}
The following estimate holds for any $\frac13<s<1$ and $t>0$: $\big|\Lambda_6\big(\frac{\widetilde{M_6^V}}{\alpha_6^V};u(t)\big)\big|\lesssim N^{2(s-1)}\|I_Nu(t)\|_{H^1}^6$.
\end{lemma}

Thus, \eqref{eq:modifiedenergy*} implies that it is enough to prove the almost conservation of $E_2$ in order to globalize the solution provided by Lemma~\ref{lem:LWP*}.

Now, with the previous computations and the fundamental theorem of calculus we have
\begin{align}\label{eq:Evar}
E_2(u(t))&=E_2(u(0))+i\sigma\int_0^t\Big\{\Lambda_6(\overline{M_6^V})+\Lambda_{10}(\overline{M_{10}^V})\Big\}dt'
\end{align}
with
\begin{align*}
\overline{M_{10}^V}=M_{10}^V+\sum_{j=1}^6(-1)^jX_j\big(\frac{\widetilde{M_6}^V}{\alpha_6^V}\big).
\end{align*}

Then we can estimate the terms above similarly as in \cite[Proposition 3.2 \& 3.3]{LWX}.

\begin{lemma}\label{lem:Ikey}
For any $\frac25<s\le \frac12$ and $\delta\in (0;1)$, the following estimates hold uniformly in $\lambda,N$:\\
\textup{(i)} ${\displaystyle \big|\int_0^\delta \Lambda_6(\overline{M_6^V})dt\big|\lesssim N^{-3}\lambda^{0+}\|I_Nu\|_{X^{1,\frac12+}(\delta)}^6}$;\\
\textup{(ii)} ${\displaystyle \big|\int_0^\delta \Lambda_{10}(\overline{M_{10}^V})dt\big|\lesssim N^{-3}\lambda^{0+}\|I_Nu\|_{X^{1,\frac12+}(\delta)}^{10}}$.
\end{lemma}

We also postpone the proof of Lemma~\ref{lem:Ikey} to Appendix~\ref{appendixproof} and conclude the proof of Proposition~\ref{prop:GWP*}. Indeed, from \eqref{eq:modifiedenergy*}, and \eqref{eq:Evar} with Lemma~\ref{lem:Iperturb} and~\ref{lem:Ikey}, we get that there exists a constant $C(V)>0$ such that for\footnote{Recall that the local time is $\delta\sim 1$ by our choice of $\lambda$ and Lemma~\ref{lem:LWPIW}, and $\|I_Nu\|_{X^{1,\frac12+}(\delta)}\lesssim 1$.} any $t\in[0;1]$,
\begin{align*}
\|I_Nu(t)\|_{H^1}^2\lesssim E_2(u(t)) = E_2(u(0))+O(N^{-3}\lambda^{0+}).
\end{align*}
We can iterate this bound for $t\in [0;\lambda^2T]$ as long as $\|I_Nu(t)\|_{H^1}\lesssim 1$. Thus, after $\lambda^2T$ iterations we get
\begin{align*}
\|I_Nu(t)\|_{H^1}^2\lesssim E_2(u(t)) = E_2(u(0))+O(TN^{-3}\lambda^{2+}),~~|t|\le \lambda^2T.
\end{align*}
Since $\lambda\sim N^{\frac{1-s}{s}}\|u_0\|_{H^s}^{\frac1s}$ and $s>\frac25$, by setting 
\begin{align*}
\alpha(s)=3-2\frac{1-s}{s}-\epsilon>0
\end{align*}
provided that $0<\epsilon\ll 1$,
and\footnote{Recall that the last lower bound on $N$ comes from the need to have $\lambda\ge 1$ together with the definition $\lambda = N^{\frac{1-s}{s}}\|u_0\|_{H^s}^{\frac1s}$. This condition is only restrictive when $\|u_0\|_{H^s}\ll 1$, but this is the case we are interested in.}
\begin{align*}
N = \max\big(T^{\frac1{\alpha(s)}}\|u_0\|_{H^s}^{\frac2{s\alpha(s)}},\|u_0\|_{H^s}^{-\frac1{1-s}}\big),
\end{align*}
we obtain that $I_Nu$ can be extended as a solution on $[-\lambda^2 T;\lambda^2T]\times\mathbb{T}_\lambda$ which satisfies
\begin{align}\label{eq:Iu}
\|I_Nu\|_{L^\infty_{\lambda^2T}H^1_\lambda}^2\lesssim 1.
\end{align}
Reversing the scaling, this shows that the local solution $u$ to \eqref{eq:NLS*} provided by Lemma~\ref{lem:LWP*} can be extended on $[-T;T]\times\mathbb{T}$, thus proving global well-posedness and the estimate
\begin{align*}
\|u\|_{L^\infty_TH^s}\lesssim \lambda^s\|u_\lambda\|_{L^\infty_{\lambda^2T}H^s_\lambda} \lesssim \lambda^s\|I_Nu_\lambda\|_{L^\infty_{\lambda^2T}H^1_\lambda}\lesssim \lambda^s\sim N^{1-s}\|u_0\|_{H^s}\sim C(T,\|u_0\|_{H^s})\|u_0\|_{H^s}
\end{align*}
with
\begin{align*}
C(T,\|u_0\|_{H^s}) \sim\begin{cases} \langle t\rangle^\frac12,~~|t|\lesssim \|u_0\|_{H^s}^{-1};\\ \|u_0\|_{H^s}^{-1},~~\|u_0\|_{H^s}^{-1}\lesssim |t|\lesssim \|u_0\|_{H^s}^{-\frac{3-\epsilon}{1-s}}\\ |t|^{\frac{1-s}{\alpha(s)}}\|u_0\|_{H^s}^{\frac{2(1-s)}{s\alpha(s)}},~~|t|\gtrsim \|u_0\|_{H^s}^{-\frac{3-\epsilon}{1-s}}.\end{cases}
\end{align*}
This proves \eqref{global}.

\subsection{Proof of Proposition~\ref{prop:GWPW}}
We now move on to the proof of the global well-posedness for \eqref{eq:NLS}. We keep the same notations as for Proposition~\ref{prop:GWP*}, except that now the Bourgain type space $X^{s,b}_{-\partial_x^2+W}$ is defined with respect to the eigenvalues $\lambda_k$ and eigenfunctions $f_k$ of the Sturm-Liouville operator $-\partial_x^2+W$ for an even potential $W\in H^1(\mathbb{T};\mathbb{R})$ (see Proposition~\ref{prop_dir}):
\begin{align}\label{XW}
\|u\|_{X^{s,b}_{-\partial_x^2+W}}=\big\|\langle\lambda_k\rangle^{\frac{s}2}\langle\tau-\lambda_k\rangle^b \langle \hat{u}(\tau,\cdot),f_k\rangle_{L^2}\big\|_{L^2_\tau\ell^2_k},
\end{align}
where now $\hat{u}$ is only the temporal Fourier transform of $u$, and the coefficients $\langle u,f_k\rangle_{L^2}$ now play the role of the Fourier coefficients.

In the rest of this subsection, however, we will keep the subscript $-\partial_x^2+W$, and simply write $X^{s,b}$ (without subscript) when $W=0$. The time-localized version is defined as in \eqref{Xsb*T}, and we have the same linear estimates as in Lemma~\ref{lem:linear*}.
\begin{lemma}\label{lem:linearW}
The following properties hold:\\
\textup{(i) ($X^{s,b}_{-\partial_x^2+W}$ as a resolution space)} If $u\in X^{s,b}_{-\partial_x^2+W}$ for some $s\in\mathbb{R}$ and $b>\frac12$, then $u\in C(\mathbb{R};H^s(\mathbb{T}))$ and $\|u\|_{L^\infty_tH^s}\lesssim \|u\|_{X^{s,b}_{-\partial_x^2+W}}$.\\
\textup{(ii) (Time localization)} For any $T\in(0;1]$ and $s\in\mathbb{R}$, $-\frac12 <b'\le b <\frac12$, it holds $\|u\|_{X^{s,b'}_{-\partial_x^2+W}(T)}\lesssim T^{b-b'}\|u\|_{X^{s,b}_{-\partial_x^2+W}(T)}$.\\
\textup{(iii) (Linear estimate)} It holds $\|e^{it(-\partial_x^2+W)}u_0\|_{X^{s,b}_{-\partial_x^2+W}(T)}\lesssim \langle T\rangle^\frac12\|u_0\|_{H^s}$ uniformly in $T>0$, for any $s\in\mathbb{R}$ and $b>\frac12$.\\
\textup{(iv) (Energy estimate)} For any $s\in\mathbb{R}$ and $b>\frac12$ it holds $${\displaystyle \Big\|\int_0^te^{i(t-t')(-\partial_x^2+W)}F(t')dt'\Big\|_{X^{s,b}(T)}\lesssim \langle T\rangle^2\|F\|_{X^{s,b-1}}}$$ uniformly in $T>0$.\\
\textup{(v) ($L^4$ Strichartz estimate)} It holds $\|u\|_{L^4_{t,x}}\lesssim \|u\|_{X^{0,\frac38+}_{-\partial_x^2+W}}$.\\
\textup{(vi) (Equivalence of norms)} If $W\in H^\sigma(\mathbb{T})$, $\sigma\ge 1$, then for any $s,b,\beta \ge 0$ satisfying $b<\frac12+\beta$, $2b<1+\beta+\sigma$ and $2b+s<\frac32+\beta+\sigma$, there is $C(\|W\|_{H^\sigma})\ge 1$ such that $\frac1C\|u\|_{X^{s-\beta,b}}\le \|u\|_{X^{s,b}_{-\partial_x^2+W}}\le C\|u\|_{X^{s+\beta,b}}$.
\end{lemma}
Note that compared to Lemma~\ref{lem:linear*}~(vi), here we have a loss of derivatives in the embeddings between $X^{s,b}$ and $X^{s,b}_{-\partial_x^2+W}$. Indeed, in the following we build the local solution to \eqref{eq:NLS} in $X^{s,b}_{-\partial_x^2+W}$ (Lemma~\ref{lem:LWPW} below) but after rescaling we extend it globally by iterating the local theory in $X^{s,b}$ (Lemma~\ref{lem:LWPIW} below), both with $b>\frac12$. Thus, we need $\beta>0$ in Lemma~\ref{lem:linearW}~(vi).

Again, we refer to Appendix~\ref{appendix} for the proof of this statement. In particular, as in Lemma~\ref{lem:LWP*}, the estimates above imply the following local well-posedness result.
\begin{lemma}\label{lem:LWPW}
Let $W\in H^1(\mathbb{T})$, and $s>\frac14$. Then there exists $\epsilon_0\in (0;1]$ such that for any $u_0\in H^s(\mathbb{T})$ with $\|u_0\|_{H^s}\le \epsilon_0$, letting $\delta\sim \|u_0\|_{H^s}^{-1}$, there exists a unique mild solution $u\in X^{s,\frac12+}_{-\partial_x^2+W}(\delta)$ to \eqref{eq:NLS} on $[-\delta,\delta]\times\mathbb{T}$ with $u(0)=u_0$, which satisfies
\begin{equation}\label{localW}
\|u\|_{X^{s,\frac12+\epsilon}_{-\partial_x^2+W}(t)}\lesssim \langle t\rangle^\frac12\|u_0\|_{H^s}
\end{equation}
for any $0\le t\le \delta$.
\end{lemma}
The proof is exactly the same as for Lemma~\ref{lem:LWP*}, as the latter only relied on the estimates of Lemma~\ref{lem:linear*}~(i)--(v), which remain true in this context.
                                                        
Next, we define again the $I$ operator as the Fourier multiplier (now back to the usual Fourier basis) with symbol $m(k)=m_s(N^{-1}k)$ for some smooth even function $m_s$ which equals 1 on $[0;1]$ and behaves like $\langle k\rangle^{s-1}$ for $|k|\ge 1$. We will then make the same rescaling procedure and set
$$u_{\lambda}(t,x)=\lambda^{-\frac12}u(\lambda^{-2}t,\lambda^{-1}x)$$
which now is a solution of 
\begin{equation}\label{eq:NLSscalW}
i\partial_t u_\lambda = -\partial_x^2 u_\lambda + W_\lambda u_\lambda +\sigma|u_\lambda|^4u_\lambda
\end{equation} 
on $[-\lambda^2T,\lambda^2T]\times\mathbb{T}_\lambda$, where
\begin{align}\label{eq:W}
W_\lambda(x)=\lambda^{-2}W(\lambda^{-1}x),~~x\in (-\pi\lambda,\pi\lambda).
\end{align}

Again, we start with a local in time estimate for $I_Nu_\lambda$.
\begin{lemma}\label{lem:LWPIW}
Let $W\in H^4(\mathbb{T})$, $\lambda\ge 1$, $s>\frac14$ and $u_\lambda(0)\in H^s(\mathbb{T}_\lambda)$. Then for $$ \delta\sim_{W}(1+\|I_Nu_\lambda(0)\|_{H^1_\lambda})^{-8-},$$ it holds
\begin{equation}\label{localI}
\|I_Nu_\lambda\|_{X^{1,\frac12+}_\lambda(\delta)} \lesssim \|I_Nu_\lambda(0)\|_{H^1_\lambda}.
\end{equation}
\end{lemma}
\begin{proof}
Note that this time we work in the standard $X^{s,b}$ space, namely the one corresponding to $W=0$. In particular, the solution $u$ to \eqref{eq:NLS} with initial data $u_0$ obtained from Lemma~\ref{lem:LWPW} belongs to $X^{s,\frac12+}_{-\partial_x^2+W}(T)$, $T\sim\|u_0\|_{H^s}^{-1}$. By Lemma~\ref{lem:linearW}~(vi), it thus belongs also to $X^{s-,\frac12+}(T)$. Then by rescaling $u_\lambda\in X^{s-,\frac12+}_\lambda(\lambda^2T)$ and $I_Nu_\lambda\in X^{1-,\frac12+}_\lambda(\lambda^2T)$. On the other hand, we can apply a fixed point argument to the equation solved by $I_Nu_\lambda$:
\begin{align}\label{eq:NLSI}
i\partial_t I_Nu_\lambda = -\partial_x^2I_Nu_\lambda + I_N(W_\lambda u_\lambda) + \sigma I_N\big(|u_\lambda|^4u_\lambda\big)
\end{align}
with initial data $I_Nu_\lambda(0)\in H^{1-}(\mathbb{T}_\lambda)$. We then proceed as in the proof of Lemma~\ref{lem:LWPI*}, except that we have to deal with the extra term $I_N(W_\lambda u_\lambda)$ as part of the nonlinearity. To estimate this term, we write it as
\begin{align*}
I_N(W_\lambda u_\lambda)= W_\lambda I_N u_\lambda + [I_N,W_\lambda]u_\lambda.
\end{align*}
The first term is straightforward to estimate with the fractional Leibniz rule:
\begin{align*}
\|W_\lambda I_N u_\lambda\|_{X^{1-,-\frac12+}(\delta)}&\lesssim \delta^{\frac12-}\|W_\lambda I_N u_\lambda\|_{L^2_\delta H^{1-}}\\
&\lesssim \delta^{1-}\big(\|W_\lambda\|_{W^{1-,\infty}}\|I_Nu_\lambda\|_{X^{0,\frac12+}(\delta)}+\|W_\lambda\|_{L^\infty}\|I_N u_\lambda\|_{X^{1-,\frac12+}(\delta)}\big)\\
&\lesssim \delta^{1-}\|W_\lambda\|_{H^4}\|I_N u_\lambda\|_{X^{1-,\frac12+}(\delta)}.
\end{align*}
As for the second term, we exploit that it is a commutator between $I_N$ and $W_\lambda$ to gain a factor $N^{-1}$ at the expense of putting a derivative on $W_\lambda$. Indeed, for fixed $t\in\mathbb{R}$ we can write its Fourier coefficient as
\begin{align*}
\big|([I_N,W_\lambda]u_\lambda)_k(t)\big|&=\Big|\int_{k_1}(m_N(k)-m_N(k_1))u_{k_1}(t)W_{k-k_1}(dk_1)_\lambda\Big|\\
&\lesssim N^{-1}\int_{|k|+|k_1|\gtrsim N}\int_0^1\langle N^{-1}(k_1+\theta(k-k_1))\rangle^{s-2}d\theta|u_{k_1}(t)||(k-k_1)W_{k-k_1}|(dk_1)_\lambda
\end{align*}
where we used that the commutator vanishes if $|k|+|k_1|\lesssim N$, and the mean value theorem to estimate its symbol. Thus, we can estimate for fixed $t\in\mathbb{R}$
\begin{align}\label{commutator}
&\big\|[I_N,W_\lambda]u_\lambda(t)\big\|_{H^{1-}}^2\notag\\
&\lesssim N^{-2}\int_{k}\langle k\rangle^{2-} \Big(\int_{|k|+|k_1|\gg N}\int_0^1\langle N^{-1}(k_1+\theta(k-k_1))\rangle^{s-2}d\theta|u_{k_1}(t)||(k-k_1)W_{k-k_1}|(dk_1)_\lambda\Big)^2(dk)_\lambda\notag\\
&\lesssim N^{-2}\int_{k}\langle k\rangle^{2-} \Big(\int_{|k-k_1|\sim|k_1|\gtrsim\max(|k|,N)}|u_{k_1}(t)||(k-k_1)W_{k-k_1}|(dk_1)_\lambda\Big)^2(dk)_\lambda\notag\\
&\qquad+ N^{-2}\int_{k}\langle k\rangle^{2-} \Big(\int_{|k-k_1|\sim|k|\gtrsim\max(|k_1|,N)}|u_{k_1}(t)||(k-k_1)W_{k-k_1}|(dk_1)_\lambda\Big)^2(dk)_\lambda\\
&\qquad+N^{-2}\int_{k}\langle k\rangle^{2-} \Big(\int_{|k|\sim|k_1|\gg\max(|k-k_1|,N)}(N^{-1}|k_1|)^{s-2}|u_{k_1}(t)||(k-k_1)W_{k-k_1}|(dk_1)_\lambda\Big)^2(dk)_\lambda\notag\\
&\lesssim N^{-2}\|W_\lambda\|_{W^{2-,\infty}}^2\|u_\lambda(t)\|_{L^2}^2 + N^{-2}\|\partial_x W_\lambda\|_{L^\infty}^2\|I_N u_\lambda(t)\|_{H^{1-}}^2\notag.
\end{align}
Thus
\begin{align*}
\big\|[I_N,W_\lambda]u_\lambda\big\|_{X^{1,-\frac12+}(\delta)}&\lesssim \delta^{1-}\lambda^{-2}\|W\|_{H^4}\|I_Nu_\lambda\|_{X^{1-,\frac12+}(\delta)}.
\end{align*}
Then we can finish as in the proof of Lemma~\ref{lem:LWPI*} to get
\begin{align*}
\|I_Nu_\lambda\|_{X^{1,\frac12+}(\delta)}&\lesssim \delta^\frac12\|I_Nu_\lambda(0)\|_{H^{1-}}+\delta^{1-}\lambda^{-2}\|W\|_{H^4}\|I_Nu_\lambda\|_{X^{1-,\frac12+}(\delta)}\\
&\qquad\qquad+\delta^{\frac12-}\|I_Nu_\lambda\|_{X^{1-,\frac12+}(\delta)}^5.
\end{align*}
A similar estimate holds for the difference equation, allowing to close a fixed point argument with our choice of $\delta$. This shows local well-posedness of \eqref{eq:NLSI} in $X^{1-,\frac12+}(\delta)$ with initial data $I_Nu_\lambda(0)$, and which thus agrees with $I_Nu_\lambda$ on $[-\delta;\delta]$. Moreover, similar estimates as above replacing $1-$ by $1$ show propagation of regularity, namely that since $I_Nu_\lambda(0)\in H^1(\mathbb{T})$ it actually holds $I_Nu_\lambda\in X^{1,\frac12+}(\delta)$. Together with a similar estimate as above shows local well-posedness in $X^{1,\frac12+}(\delta)$ and \eqref{localI} by our choice of $\delta$.
\end{proof}
\begin{remark}\rm
Note that the local well-posedness results in Lemmas~\ref{lem:LWP*},~\ref{lem:LWPI*},~\ref{lem:LWPW}, and~\ref{lem:LWPIW} are conditional, meaning that we cannot claim uniqueness of the solution in $C([-T;T];H^s(\mathbb{T}))$ but only in the smaller space $X^{s,\frac12+}_{-\partial_x^2+V\ast}(T)$ (respectively $X^{1,\frac12+}(\delta)$, $X^{s,\frac12+}_{-\partial_x^2+W}(T)$, and $X^{1,\frac12+}(\delta)$). In particular, we can only compare $I_Nu_\lambda$, where $u$ is provided by Lemma~\ref{lem:LWPW}, to the solution of \eqref{eq:NLSI}, if they both belong to $X^{1,\frac12+}(\delta)$. But the embeddings of Lemma~\ref{lem:linearW}~(vi) can only guarantee that $I_Nu_\lambda\in X^{1-,\frac12+}(\delta)$. So we need both local well-posedness of \eqref{eq:NLSI} in $X^{1-,\frac12+}(\delta)$ and the propagation of regularity to obtain that actually $I_Nu_\lambda\in X^{1,\frac12+}(\delta)$.
\end{remark}

We can then proceed with the second modified energy as above, by introducing
\begin{align*}
E_2(u)&= E(I_Nu)-\Lambda_6\big(\frac{\widetilde{M_6}}{\alpha_6}\big)\\
&= \Lambda_2(\sigma_2)+\Lambda_3(\sigma_3;u,\overline{u},W)+\Lambda_6(\sigma_6)-\Lambda_6\big(\frac{\widetilde{M_6}}{\alpha_6}\big)
\end{align*}
where $\sigma_3=\frac12m_N(k_1)m_N(k_2)$.

From \eqref{eq:NLS}, we thus get
\begin{align}\label{E2W}
\frac{d}{dt}E_2(u(t))&=\Lambda_2(\sigma_2\alpha_2)+\Lambda_6(\overline{M_6})+\Lambda_{10}(\overline{M_{10}})+F(u,W)
\end{align}
where
\begin{align*}
F(u,W)&=\langle -\partial_x^2I_Nu ,iI_N(W_\lambda u)\rangle_{L^2}+\langle W_\lambda I_Nu,-i\partial_x^2I_Nu+iI_N(W_\lambda u)+iI_N(|u|^4u)\rangle_{L^2}\\
&\qquad\qquad+\langle |I_Nu|^4I_Nu,iI_N(W_\lambda u)\rangle_{L^2}+\Lambda_7(M_7;u,\ldots,\overline{u},W)\\
&=\mathrm{I}+\mathrm{II}+\mathrm{III}+\mathrm{IV}+\mathrm{V}+\Lambda_7(M_7;u,\ldots,\overline{u},W),
\end{align*}
where 
\begin{align*}
M_7=\sum_{j=1}^6(-1)^j\frac{\widetilde{M_6}}{\alpha_6}(k_1,\ldots,k_{j-1},k_j+k_7,k_{j+1},\ldots,k_6).
\end{align*}

Again, recall that $\alpha_2\equiv 0$ on $\Gamma_2$, and moreover
\begin{align*}
\mathrm{I}+\mathrm{II}&=\langle \partial_xI_Nu,\partial_x[I_N,W_\lambda]u\rangle_{L^2},
\end{align*}
\begin{align*}
\mathrm{III}&=-\langle W_\lambda I_Nu,iI_N(W_\lambda u)\rangle_{L^2} = \langle iW_\lambda I_Nu,[I_N,W_\lambda ]u\rangle_{L^2},
\end{align*}
and
\begin{align*}
\mathrm{IV}+\mathrm{V}&=-\langle iW_\lambda I_Nu,I_N(|u|^4u)\rangle_{L^2}-\langle I_N(W_\lambda u),|I_Nu|^4I_Nu)\rangle_{L^2}\\
&=-\langle iW_\lambda I_Nu - iI_N(W_\lambda u),I_N(|u|^4u)\rangle_{L^2}-\langle iI_N(W_\lambda u),I_N(|u|^4u)-|I_Nu|^4I_Nu\rangle_{L^2}.
\end{align*}
The term $\Lambda_7(M_7;u,\ldots,\overline{u},W)$, and the commutator terms $\mathrm{I}+\mathrm{II}$, $\mathrm{III}$, and $\mathrm{IV}+\mathrm{V}$, are then perturbation with respect to the analysis in \cite{LWX}. Indeed, we have the following lemma, whose proof is postponed to Appendix~\ref{appendixproof}.
\begin{lemma}\label{lem:Iperturb}
The following estimates hold uniformly in $\lambda,N$ and $t\in[0;1]$:\\
\textup{(i)} $\big|\mathrm{I}+\mathrm{II}\big|\lesssim_{\|W\|_{H^4}} \lambda^{-2}N^{-1}\|I_Nu\|_{H^1}^2$;\\
\textup{(ii)} $\big|\mathrm{III}\big|\lesssim_{\|W\|_{H^4}}N^{-1-s}\lambda^{-4}\|u\|_{L^2}^2$;\\
\textup{(iii)} $\big|\mathrm{IV}+\mathrm{V}\big|\lesssim_{\|W\|_{H^4}}\lambda^{-2}N^{-1}\big\|I_Nu\big\|_{H^1}^6$;\\
\textup{(iv)}  $\big|\Lambda_7(M_7;u,\ldots,\overline{u},W)\Big|\lesssim_{\|W\|_{H^4}} N^{2(s-1)}\lambda^{-2}\|I_Nu\|_{H^1}^6$.
\end{lemma}

Since the remaining main terms in \eqref{E2W} are the same as those treated in \cite{LWX}, and $s\le \frac12$, we can finally estimate the increments of the modified energy by
\begin{align*}
E_2(u(t)) = E_2(u(0))+O(N^{-1}\lambda^{-2})+O(N^{-3}\lambda^{0+})
\end{align*}
for $|t|\le \delta\sim 1$. Then, iterating $\lambda^{2}T$ times yields
\begin{align*}
E_2(u(t)) = E_2(u(0))+O(TN^{-1})+O(TN^{-3}\lambda^{2+}).
\end{align*}
We can thus conclude as in the proof of Proposition~\ref{prop:GWP*}.

\appendix

\section{Estimates related to $X^{s,b}$ spaces}\label{appendix}
In this section, we give a proof of the standard estimates in $X^{s,b}$ spaces as stated in Lemma~\ref{lem:linear*} and~\ref{lem:linearW}.

\subsection{Proof of Lemma~\ref{lem:linear*}}
We start by proving the equivalence of norms, i.e. Lemma~\ref{lem:linear*}~(v).
Similarly, as in \eqref{eq:ev}, we have
\begin{align*}
1+|\tau-k^2-V_k|^2\le 2(1+\|V\|_{\ell^\infty}^2)(1+|\tau-k^2|^2)\le 4(1+\|V\|_{\ell^\infty}^2)^2(1+|\tau-\omega_k|^2),
\end{align*}
for any $\tau\in\mathbb{R}$ and $k\in\mathbb{Z}$. This shows that
\begin{align*}
\|u\|_{X^{s,b}_{-\partial_x^2+V\ast}}^2&=\int_\mathbb{R}\sum_{k\in\mathbb{Z}}(1+|\tau-\omega_k|^2)^b(1+\omega_k^2)^s|\hat{u_k}(\tau)|^2d\tau\\
& \le C_b(1+\|V\|_{\ell^\infty})\int_\mathbb{R}\sum_{k\in\mathbb{Z}}(1+|\tau-k^2|^2)^b(1+k^2)^{2s}|\hat{u_k}(\tau)|^2d\tau =\|u\|_{X^{s,b}_{-\partial_x^2}}^2\\
& \le C_b'(1+\|V\|_{\ell^\infty})\|u\|_{X^{s,b}_{-\partial_x^2+V\ast}}^2
\end{align*}
for any $s,b\in\mathbb{R}$ and some constants $C_b=C_b(1+\|V\|_{\ell^\infty})>0$ and $C_b'=C_b'(1+\|V\|_{\ell^\infty})>0$. This proves Lemma~\ref{lem:linear*}~(vi). In particular, it suffices to prove Lemma~\ref{lem:linear*}~(i), (ii), and (v) in the case $V=0$, which we now recall for completeness.

To show (v), we will use a dyadic decomposition in the modulation variable. Namely, for $K\in 2^{\mathbb{N}}$ running on dyadic integers, let $\chi_K$ be a smooth dyadic partition of unity: $\chi_K$ is a smooth compactly supported function such that $\chi_K(x)= 1$ on $K\le |x|\le 2K$ and $\chi_K$ is supported on $\frac34K\le |x|\le \frac52K$, and $\sum_K\chi_K\equiv 1$. Then for $u\in X^{s,b}$ define the smooth projector $\widehat{P_Ku}(\tau,k)=\chi_K(\tau-k^2)\hat{u}(\tau,k)$.  Then, using Cauchy-Schwarz inequality, we get
\begin{align*}
\|u\|_{L^4_{t,x}}^2&=\|u^2\|_{L^2_{t,x}} \le \sum_{K_1,K_2} \|P_{K_1}uP_{K_2}u\|_{L^2_{t,x}} \\
&= \sum_{K_1,K_2}\Big(\int_{\mathbb{R}}\sum_{k\in\mathbb{Z}}\big|\int_{\mathbb{R}}\sum_{k_1\in\mathbb{Z}}\widehat{P_{K_1}u}(\tau_1,k_1)\widehat{P_{K_2}u}(\tau-\tau_1,k-k_1)d\tau_1\big|^2d\tau\Big)^\frac12\\
&\le \sum_{K_1,K_2}\Big(\sup_{\tau,k}|A_{\tau,k}|\|P_{K_1}u\|_{L^2_{t,x}}^2\|P_{K_2}u\|_{L^2_{t,x}}^2\Big)^\frac12,
\end{align*}
where $A_{\tau,k}=\{(\tau_1,k_1)\in\mathbb{R}\times\mathbb{Z},~~\langle\tau_1-k_1^2\rangle\sim K_1,~~\langle\tau-\tau_1-(k-k_1)^2\rangle\sim K_2\}$. In particular, using that $\tau_1$ varies in a set of size $\min(K_1,K_2)$ and that for $(\tau_1,k_1)\in A_{\tau,k}$, $$|\tau-k_1^2-(k-k_1)^2|\le |\tau_1-k_1^2|+|\tau-\tau_1-(k-k_1)^2|\lesssim \max(K_1,K_2),$$ we get 
\begin{align*}
|A_{\tau,k}|&\lesssim \min(K_1,K_2)\#\{k_1\in\mathbb{Z},~~|\tau-k_1^2-(k-k_1)^2|\lesssim \max(K_1,K_2)\}\\
&=\min(K_1,K_2)\#\{k_1\in\mathbb{Z},~~(k_1-\frac{k}2)^2=\frac{\tau}{2}-\frac{k^2}{4}+O( \max(K_1,K_2))\}\\
&\lesssim \min(K_1,K_2)\max(K_1,K_2)^\frac12.
\end{align*}
Since $\min(K_1,K_2)\max(K_1,K_2)^\frac12\le (K_1K_2)^\frac34$, using Cauchy-Schwarz inequality to sum on $K_1,K_2\in 2^{\mathbb{N}}$ and using that $\sum_K K^{0-}<\infty$ and $\sum_{K}K^{2b}\|P_Ku\|_{L^2_{t,x}}^2 \sim \|u\|_{X^{0,b}}^2$ concludes the proof of (v) in the case $V=0$. Note that the estimate on $A_{\tau,k}$ remains valid if $k^2$ is replaced by $k^2+V_k$ for $V\in L^2(\mathbb{T})$.

Next, if $u\in X^{s,b}$ for some $s\in\mathbb{R}$ and $b>\frac12$, we have by Cauchy-Schwarz inequality
\begin{align*}
\|u(t)\|_{H^s}\le \|\langle k\rangle^s\hat{u_k}(\tau)\|_{\ell^2_kL^1_\tau} \le \|\langle\tau - k^2\rangle^{-b}\|_{\ell^\infty_kL^2_\tau}\|u\|_{X^{s,b}}\lesssim \|u\|_{X^{s,b}},
\end{align*}
uniformly in $t\in\mathbb{R}$. This shows (i).

For (ii), if we take $u\in X^{s,b}(T)$ and $v\in X^{s,b}$ be an extension such that $\|v\|_{X^{s,b}}\le 2\|u\|_{X^{s,b}(T)}$, then for a smooth cut-off function $\eta$ such that $\eta\equiv 1$ on $[-1;1]$, $\eta(T^{-1}t)v$ is also an extension of $u$. Thus,
\begin{align*}
\|u\|_{X^{s,b'}(T)}\le \|\eta(T^{-1}\cdot)v\|_{X^{s,b'}} = \|\eta(T^{-1}\cdot)f\|_{H^{b'}_t},
\end{align*}
with $$f(t)=\int_{\mathbb{R}}e^{it\tau}\|\langle k\rangle^s \hat{v}(\tau-k^2)\|_{\ell^2_k}d\tau$$ such that $\|f\|_{H^b_t}=\|v\|_{X^{s,b}}$. Thus, the estimate of Lemma~\ref{lem:linear*}~(ii) is reduced to the general estimate for functions of time only:
\begin{align*}
\|\eta(T^{-1}\cdot)f\|_{H^{b'}_t}\lesssim T^{b-b'}\|f\|_{H^b_t},
\end{align*}
for any $T\in (0;1]$, the proof of which is given in \cite[Lemma 3.11]{ET}.

It remains to prove (iii) and (iv). Just as for (ii), these estimates follow from
\begin{align*}
\|\eta(\langle T\rangle^{-1}\cdot)\|_{H^b_t}\lesssim_\eta \langle T\rangle^\frac12
\end{align*}
and
\begin{align*}
\Big\|\eta(\langle T\rangle^{-1}\cdot)\int_0^tf(t')dt'\Big\|_{H^b_t}\lesssim_\eta \langle T\rangle^2\|f\|_{H^{b-1}}
\end{align*}
for any $f\in H^{b-1}(\mathbb{R})$. The first one follows from a direct computation. As for the second one, we first compute
\begin{align*}
\int_0^t f(t')dt' = \int_{\mathbb{R}}\frac{e^{it\tau}-1}{i\tau}\hat{f}(\tau)d\tau,
\end{align*}
so that
\begin{align*}
\Big\|\eta(\langle T\rangle^{-1}\cdot)\int_0^tf(t')dt'\Big\|_{H^b_t}^2&=\int_{\mathbb{R}}\langle\tau\rangle^{2b}\Big|\int_{\mathbb{R}}\langle T\rangle\frac{\widehat{\eta}(\langle T\rangle(\tau-\tau_1))-\widehat{\eta}(\langle T\rangle \tau)}{i\tau_1}\hat{f}(\tau_1)d\tau_1\Big|^2d\tau\\
&=\int_{\mathbb{R}}\langle\tau\rangle^{2b}\Big|\int_{|\tau_1|\le 1}\langle T\rangle\frac{\widehat{\eta}(\langle T\rangle(\tau-\tau_1))-\widehat{\eta}(\langle T\rangle \tau)}{i\tau_1}\hat{f}(\tau_1)d\tau_1\Big|^2d\tau\\
&\qquad\qquad+\int_{\mathbb{R}}\langle\tau\rangle^{2b}\Big|\int_{|\tau_1|>1}\langle T\rangle\frac{\widehat{\eta}(\langle T\rangle(\tau-\tau_1))-\widehat{\eta}(\langle T\rangle \tau)}{i\tau_1}\hat{f}(\tau_1)d\tau_1\Big|^2d\tau.
\end{align*}
As for the first term, using the mean value theorem, that $\widehat{\eta}$ is a Schwartz function, and Cauchy-Schwarz inequality,
\begin{align*}
&\int_{\mathbb{R}}\langle\tau\rangle^{2b}\Big|\int_{|\tau_1|\le 1}\langle T\rangle\frac{\widehat{\eta}(\langle T\rangle(\tau-\tau_1))-\widehat{\eta}(\langle T\rangle \tau)}{i\tau_1}\hat{f}(\tau_1)d\tau_1\Big|^2d\tau\\
&\qquad\qquad\lesssim\int_{\mathbb{R}}\langle\tau\rangle^{2b}\Big|\int_{|\tau_1|\le 1}\langle T\rangle^2\int_0^1|\widehat{\eta}'(\langle T\rangle(\tau-\theta\tau_1))|d\theta|\hat{f}(\tau_1)|d\tau_1\Big|^2d\tau\\
&\qquad\qquad\lesssim_\eta \int_{|\tau|\gg 1}|\tau|^{2b}\Big(\int_{|\tau_1|\le 1}\langle T\rangle^2 \langle\langle T\rangle \tau\rangle^{-10} |\hat{f}(\tau_1)|d\tau_1\Big)^2 d\tau\\
&\qquad\qquad\qquad\qquad+ \int_{|\tau'|\lesssim \langle T\rangle}\Big(\int_{|\tau_1|\le 1}\langle T\rangle^2  |\hat{f}(\tau_1)|d\tau_1\Big)^2 \langle T\rangle^{-1}d\tau'\\
&\qquad\qquad\lesssim (\langle T\rangle^{-16}+\langle T\rangle^4)\Big(\int_{|\tau_1|\le 1}|\hat{f}(\tau_1)|d\tau_1\Big)^2 \lesssim \langle T\rangle^4\|f\|_{H^{b-1}}^2.
\end{align*}
For the second term, we use instead (recall that $b>\frac12$) $\langle \tau\rangle^b\lesssim \langle \tau-\tau_1\rangle^b\langle\tau_1\rangle^b$ to estimate with Young and Cauchy-Schwarz inequalities together with $b>\frac12$ and that $\widehat{\eta}$ is a Schwartz function:
\begin{align*}
&\int_{\mathbb{R}}\langle\tau\rangle^{2b}\Big|\int_{|\tau_1|>1}\langle T\rangle\frac{\widehat{\eta}(\langle T\rangle(\tau-\tau_1))-\widehat{\eta}(\langle T\rangle \tau)}{i\tau_1}\hat{f}(\tau_1)d\tau_1\Big|^2d\tau\\
&\qquad\qquad\lesssim\int_{\mathbb{R}}\Big(\int_{|\tau_1|>1}\langle T\rangle\big(\langle \tau-\tau_1\rangle^{b}\langle \tau_1\rangle^{b-1}|\widehat{\eta}(\langle T\rangle(\tau-\tau_1))|+\langle \tau\rangle^b\langle\tau_1\rangle^{-1}|\widehat{\eta}(\langle T\rangle \tau)|\big)|\hat{f}(\tau_1)|d\tau_1\Big)^2d\tau\\
&\qquad\qquad\lesssim \langle T\rangle^2\|\langle \tau\rangle^b\widehat{\eta}(\langle T\rangle \tau)\|_{L^1}^2\|f\|_{H^{b-1}}^2 +\langle T\rangle^2\|\langle \tau\rangle^b\widehat{\eta}(\langle T\rangle \tau)\|_{L^2}^2\|\langle \tau_1\rangle^{-1}\hat{f}\|_{L^1}^2\\
&\qquad\qquad\lesssim \langle T\rangle \|f\|_{H^{b-1}}^2.
\end{align*}
This finally shows Lemma~\ref{lem:linear*}~(iii) and (iv).

\subsection{Proof of Lemma~\ref{lem:linearW}}
Contrary to the previous case, in the case of a multiplicative potential we could only prove the equivalence of norms with a derivative loss, see Lemma~\ref{lem:linearW}~(vi). So instead we show the $L^4$ Strichartz estimate of Lemma~\ref{lem:linearW}~(v) directly. Indeed, note that the other estimates of Lemma~\ref{lem:linearW} apart from (vi) are proved exactly as those for Lemma~\ref{lem:linear*}. Recall from \cite[Theorem 4]{PT} that for an even $W\in H^\sigma(\mathbb{T})$, $\sigma\ge 1$, the operator $-\partial_x^2+W$ has eigenvalues $\lambda_n$ satisfying
\begin{align}\label{est:ev0}
\Big|\lambda_n-n^2-\frac1\pi\int_0^\pi W(x)dx\Big|\leq C(\|W\|_{H^\sigma}) \frac1n,
\end{align} 
$n\in\mathbb{Z}$. Here and below $C(\|W\|_{H^\sigma}),C'(\|W\|_{H^\sigma})$ are various constants which may vary from line to line.

 Its orthonormal eigenfunctions $f_n$ are odd (resp. even) when $n$ is positive (resp. nonpositive) and satisfy for any $k\in\mathbb{Z}$
\begin{align}\label{est:ef}
\big|\widehat{f_n}(k)\mathbf{1}_{|n|\neq |k|}\big|\le C(\|W\|_{H^\sigma})\langle|n|+|k|\rangle^{-1}\langle|n|-|k|\rangle^{-1-\sigma}.
\end{align}
Indeed this follows by writing
\begin{align*}
\lambda_n \widehat{f_n}(k)= \langle (-\partial_x^2+W)f_n,e^{ikx}\rangle_{L^2} = \langle f_n,(-\partial_x^2+W)e^{ikx}\rangle_{L^2} = k^2\widehat{f_n}(k)+\langle f_n,We^{ikx}\rangle_{L^2}
\end{align*}
for any $n,k\in\mathbb{Z}$. This yields
\begin{align*}
\widehat{f_n}(k)\mathbf{1}_{|n|\neq |k|}&=\frac{\mathbf{1}_{|n|\neq |k|}}{\lambda_n-k^2}\sum_{k_1\in\mathbb{Z}}W_{k-k_1}\widehat{f_n}(k_1)\\
&=\frac{\mathbf{1}_{|n|\neq |k|}}{\lambda_n-k^2}\Big\{W_{k\pm n}\widehat{f_n}(\pm n)+\sum_{|k_1|\neq |n|}W_{k-k_1}\frac1{\lambda_n-k_1^2}\sum_{k_2\in\mathbb{Z}}W_{k_1-k_2}\widehat{f_n}(k_2)\Big\}.
\end{align*}
Since we have from \eqref{est:ev0} that
\begin{align*}
|\lambda_n-k^2|\geq |n^2-k^2| - \frac1\pi\|W\|_{L^1}-C(\|W\|_{H^\sigma})\frac1n \geq C(\|W\|_{H^\sigma}) \langle n^2-k^2\rangle  
\end{align*}
for $|n|\neq|k|$, together with the trivial bounds 
$$|\hat{f}(j)|\le \|f\|_{L^2}=1 \text{ and }|W_{k\pm n}|\le \langle|n|-|k|\rangle^{-\sigma}\|W\|_{H^\sigma},$$ we indeed infer
\begin{align*}
&\big|\widehat{f_n}(k)\mathbf{1}_{|n|\neq |k|}\big|\\
&\le C(\|W\|_{H^\sigma}) \langle n^2-k^2\rangle^{-1}\langle |n|-|k|\rangle^{-\sigma}+C(\|W\|_{H^\sigma})\langle n^2-k^2\rangle^{-1}\sum_{k_1}\langle k-k_1\rangle^{-\sigma}\langle n^2-k_1^2\rangle^{-1}\\
&\le C(\|W\|_{H^\sigma})\langle|n|+|k|\rangle^{-1}\langle|n|-|k|\rangle^{-1-\sigma}+C(\|W\|_{H^\sigma})\sum_{k_2}\langle k_2\rangle^{-\sigma}\langle n\pm (k-k_2)\rangle^{-1}\langle |n|+|k-k_2|\rangle^{-1}\\
&\le C(\|W\|_{H^\sigma})\langle|n|+|k|\rangle^{-1}\langle|n|-|k|\rangle^{-1-\sigma}\\
&\qquad\qquad+C(\|W\|_{H^\sigma})\Big\{|n\pm k|^{-\sigma}\sum_{|k_2|\sim |n\pm k|\gg |n\pm (k-k_2)|}\langle n\pm (k-k_2)\rangle^{-2}\\
&\qquad\qquad\qquad\qquad+\sum_{|k_2|\sim |n\pm(k-k_2)|\gtrsim |n\pm k|}\langle k_2\rangle^{-\sigma-1}+|n\pm k|^{-1}\sum_{|k_2|\ll |n\pm k|\sim|n\pm (k-k_2)|}\langle k_2\rangle^{-\sigma}\Big\}\\
&\le C(\|W\|_{H^\sigma})\langle|n|+|k|\rangle^{-1}\langle|n|-|k|\rangle^{-1-\sigma}.
\end{align*}
This shows \eqref{est:ef}.

On top of the estimate \eqref{est:ef} on the eigenfunctions, we also have from the expansion \eqref{est:ev0} of the eigenvalues that
\begin{align}\label{est:ev}
\langle\lambda_n\rangle\le C(\|W\|_{H^\sigma})\langle n\rangle^2 \le C'(\|W\|_{H^\sigma}) \langle \lambda_n \rangle,
\end{align}
and thus also
\begin{align}\label{est:ev2}
\langle\tau-\lambda_n\rangle^2\le C(\|W\|_{H^\sigma})\langle\tau-n^2\rangle^2\le C'(\|W\|_{H^\sigma})\langle\tau-\lambda_n\rangle^2
\end{align}
for any $\tau\in\mathbb{R}$ and $n\in\mathbb{Z}$, since
\begin{align*}
1+|\tau-\lambda_n|\le 1+|\tau-n^2|+C(\|W\|_{H^\sigma})\le C'(\|W\|_{H^\sigma})\big(1+|\tau-n^2|\big)
\end{align*}
and similarly for the second estimate in \eqref{est:ev2}.

To show Lemma~\ref{lem:linearW}~(v), we will then proceed as for Lemma~\ref{lem:linear*}~(v) and use a dyadic decomposition in the modulation variable:
\begin{align*}
\|u\|_{L^4_{t,x}}^2&=\|u^2\|_{L^2_{t,x}} \le \sum_{K_1,K_2} \|P_{K_1}uP_{K_2}u\|_{L^2_{t,x}} \\
&= \sum_{K_1,K_2}\Big(\int_{\mathbb{R}}\sum_{n_0\in\mathbb{Z}}\big|\int_{\mathbb{R}}\sum_{n_1,n_2\in\mathbb{Z}}\widehat{P_{K_1}u}(\tau_1,n_1)\widehat{P_{K_2}u}(\tau_0-\tau_1,n_2)\langle f_{n_1}f_{n_2},f_{n_0}\rangle_{L^2} d\tau_1\big|^2d\tau_0\Big)^\frac12.
\end{align*}
Now the main difference with the previous subsection is to estimate the coefficient $\langle f_{n_1}f_{n_2},f_{n_0}\rangle_{L^2}$. Using Plancherel's theorem and \eqref{est:ef} with $\sigma=0$, we have
\begin{align}\label{est:ef-tri}
&\langle f_{n_1}f_{n_2},f_{n_0}\rangle_{L^2}\notag\\
&= \sum_{k_1,k_2}\widehat{f_{n_1}}(k_1)\widehat{f_{n_2}}(k_2)\overline{\widehat{f_{n_0}}}(k_1+k_2)\notag\\
&=\sum_{k_1,k_2}\big(c_{n_1}\mathbf{1}_{|n_1|=|k_1|}+O(\langle n_1^2-k_1^2\rangle^{-1})\mathbf{1}_{|n_1|\neq |k_1|}\big)\big(c_{n_2}\mathbf{1}_{|n_2|=|k_2|}+O(\langle n_2^2-k_2^2\rangle^{-1})\mathbf{1}_{|n_2|\neq |k_2|}\big)\\
&\qquad\qquad\times\big(c_{n_0}\mathbf{1}_{|n_0|=|k_1+k_2|}+O(\langle n_0^2-(k_1+k_2)^2\rangle^{-1})\mathbf{1}_{|n_0|\neq |k_1+k_2|}\big),\notag
\end{align}
with $|c_{n_j}|\le 1$. 

Letting $a_{n_j,k_j}$ denotes either $\mathbf{1}_{|n_j|=|k_j|}$ or $\langle n_j^2-k_j^2\rangle^{-1}\mathbf{1}_{|n_j|\neq |k_j|}$, we have
\begin{align}\label{est-bil}
&\int_{\mathbb{R}}\sum_{n_0\in\mathbb{Z}}\Big|\int_{\mathbb{R}}\sum_{n_1,n_2\in\mathbb{Z}}\widehat{P_{K_1}u}(\tau_1,n_1)\widehat{P_{K_2}u}(\tau_0-\tau_1,n_2)\langle f_{n_1}f_{n_2},f_{n_0}\rangle_{L^2} d\tau_1\Big|^2d\tau_0\notag\\
&\qquad\lesssim \int_{\mathbb{R}}\sum_{n_0}\Big(\int_{\mathbb{R}}\sum_{n_1,n_2,k_1,k_2}\big|\widehat{P_{K_1}u}(\tau_1,n_1)\widehat{P_{K_2}u}(\tau_0-\tau_1,n_2)\big|\\
&\qquad\times\mathbf{1}_{A_{\tau_0}}(\tau_1,n_1,n_2)a_{n_1,k_1}a_{n_2,k_2}a_{n_0,k_1+k_2}d\tau_1\Big)^2d\tau_0,\notag
\end{align}
where similarly as before
\begin{align*}
A_{\tau_0} = \{(\tau_1,n_1,n_2)\in \mathbb{R}\times\mathbb{Z}^2,~~|\tau_1-\lambda_{n_1}|\lesssim K_1,~~|\tau_0-\tau_1-\lambda_{n_2}|\lesssim K_2\}.
\end{align*}
To end the proof of Lemma~\ref{lem:linearW}, we do a case-by-case analysis on \eqref{est-bil} depending on the value of the $a_{n_j,k_j}$'s. We say that $a_{n_j,k_j}$ is of type I if $a_{n_j,k_j}=\mathbf{1}_{|n_j|=|k_j|}$, and of type II otherwise.\\
\textbf{Case 1: I-I-I.} In the diagonal case, we have $n_0=\pm_1n_1\pm_2n_2$ for some choices of signs $\pm_j$, so that we can estimate \eqref{est-bil} as above by Cauchy-Schwarz inequality in $\tau_1,n_1$:
\begin{align*}
&\int_{\mathbb{R}}\sum_{n_0}\Big(\int_{\mathbb{R}}\sum_{n_1}\big|\widehat{P_{K_1}u}(\tau_1,n_1)\widehat{P_{K_2}u}(\tau_0-\tau_1,\mp_2(n_0\mp_1n_1))\mathbf{1}_{A_{\tau_0}}(\tau_1,n_1,\mp_2(n_0\mp_1n_1))d\tau_1\Big)^2d\tau_0\\
&\qquad\lesssim \|P_{K_1}u\|_{L^2_{t,x}}^2\|P_{K_2}u\|_{L^2_{t,x}}^2\sup_{\tau_0,n_0}|\{(\tau_1,n_1)\in\mathbb{R}\times\mathbb{Z},~~(\tau_1,n_1,\mp_2(n_0\mp_1n_1))\in A_{\tau_0}\}|\\
&\qquad\lesssim (K_1K_2)^\frac34\|P_{K_1}u\|_{L^2_{t,x}}^2\|P_{K_2}u\|_{L^2_{t,x}}^2.
\end{align*}
\textbf{Case 2: II-I-I.} If $a_{n_0,k_1+k_1}=\langle n_0^2-(k_1+k_2)^2\rangle^{-1}$ and $a_{n_1,k_1}=\mathbf{1}_{|n_1|=|k_1|}$, $a_{n_2,k_2}=\mathbf{1}_{|n_2|=|k_2|}$, we use Cauchy-Schwarz inequality and then sum on $n_0,n_2$ to estimate \eqref{est-bil} similarly as above by
\begin{align*}
&\int_{\mathbb{R}}\sum_{n_0}\Big(\sum_{n_1,n_2}\big|\widehat{P_{K_1}u}(\tau_1,n_1)\widehat{P_{K_2}u}(\tau_0-\tau_1,n_2)\big|\mathbf{1}_{A_{\tau_0}}(\tau_1,n_1,n_2)\langle n_0^2-(\pm_1n_1\pm_2n_2)^2\rangle^{-1}d\tau_1\Big)^2d\tau_0\\
&\qquad\lesssim \|P_{K_1}u\|_{L^2_{t,x}}^2\|P_{K_2}u\|_{L^2_{t,x}}^2\sup_{\tau_0}\int_{\mathbb{R}}\sum_{n_0,n_1,n_2}\mathbf{1}_{A_{\tau_0}}(\tau_1,n_1,n_2)\langle n_0\rangle^{-2}\langle n_0\mp_1n_1\mp_2n_2\rangle^{-2}d\tau_1\\
&\qquad\lesssim \|P_{K_1}u\|_{L^2_{t,x}}^2\|P_{K_2}u\|_{L^2_{t,x}}^2\sup_{\tau_0,n_2}|\{(\tau_1,n_1)\in\mathbb{R}\times\mathbb{Z},~~(\tau_1,n_1,n_2)\in A_{\tau_0}\}|\\
&\qquad\lesssim (K_1K_2)^\frac34\|P_{K_1}u\|_{L^2_{t,x}}^2\|P_{K_2}u\|_{L^2_{t,x}}^2.
\end{align*}
The cases I-II-I and I-I-II are dealt with similarly.\\
\textbf{Case 3: II-II-I.} In this case we use Cauchy-Schwarz inequality and then sum on $n_0,n_2,k_1$ to estimate \eqref{est-bil} by
\begin{align*}
&\int_{\mathbb{R}}\sum_{n_0}\Big(\sum_{n_1,n_2,k_1}\big|\widehat{P_{K_1}u}(\tau_1,n_1)\widehat{P_{K_2}u}(\tau_0-\tau_1,n_2)\big|\\
& \qquad\qquad\qquad\qquad\times\mathbf{1}_{A_{\tau_0}}(\tau_1,n_1,n_2)\langle n_0-(k_1\pm_2n_2)^2\rangle^{-1}\langle n_1^2-k_1^2\rangle^{-1}d\tau_1\Big)^2d\tau_0\\
&\lesssim \sup_{\tau_0}\int_{\mathbb{R}}\sum_{n_0,n_1,n_2,k_1}\mathbf{1}_{A_{\tau_0}}(\tau_1,n_1,n_2)\langle n_0\rangle^{-2}\langle n_0-k_1\mp_2n_2\rangle^{-2}\langle k_1\rangle^{-2}\langle n_1-k_1\rangle^{-2}d\tau_1\\
& \qquad\qquad\qquad\qquad\times \|P_{K_1}u\|_{L^2_{t,x}}^2\|P_{K_2}u\|_{L^2_{t,x}}^2\\
&\lesssim \|P_{K_1}u\|_{L^2_{t,x}}^2\|P_{K_2}u\|_{L^2_{t,x}}^2\sup_{\tau_0,n_2}|\{(\tau_1,n_1)\in\mathbb{R}\times\mathbb{Z},~~(\tau_1,n_1,n_2)\in A_{\tau_0}\}|\\
&\lesssim (K_1K_2)^\frac34\|P_{K_1}u\|_{L^2_{t,x}}^2\|P_{K_2}u\|_{L^2_{t,x}}^2.
\end{align*}
The cases II-I-II and I-II-II are dealt with similarly.\\
\textbf{Case 4: II-II-II.} Finally, in this last case we use Cauchy-Schwarz inequality and then sum on $n_0,n_2,k_1,k_2$ to estimate \eqref{est-bil} by
\begin{align*}
&\int_{\mathbb{R}}\sum_{n_0}\Big(\sum_{n_1,n_2,k_1,k_2}\big|\widehat{P_{K_1}u}(\tau_1,n_1)\widehat{P_{K_2}u}(\tau_0-\tau_1,n_2)\big|\\
&\qquad\qquad\times\mathbf{1}_{A_{\tau_0}}(\tau_1,n_1,n_2)\langle n_0-(k_1+k_2)^2\rangle^{-1}\langle n_1^2-k_1^2\rangle^{-1}\langle n_2^2-k_2^2\rangle^{-1}d\tau_1\Big)^2d\tau_0\\
&\lesssim \|P_{K_1}u\|_{L^2_{t,x}}^2\|P_{K_2}u\|_{L^2_{t,x}}^2\sup_{\tau_0}\int_{\mathbb{R}}\sum_{n_0,n_1,n_2,k_1,k_2}\mathbf{1}_{A_{\tau_0}}(\tau_1,n_1,n_2)\\
&\qquad\qquad\times\langle n_0\rangle^{-2}\langle n_0-k_1-k_2\rangle^{-2}\langle k_1\rangle^{-2}\langle n_1-k_1\rangle^{-2}\langle k_2\rangle^{-2}\langle n_2-k_2\rangle^{-2}d\tau_1\\
&\lesssim \|P_{K_1}u\|_{L^2_{t,x}}^2\|P_{K_2}u\|_{L^2_{t,x}}^2\sup_{\tau_0,n_2}|\{(\tau_1,n_1)\in\mathbb{R}\times\mathbb{Z},~~(\tau_1,n_1,n_2)\in A_{\tau_0}\}|\\
&\lesssim (K_1K_2)^\frac34\|P_{K_1}u\|_{L^2_{t,x}}^2\|P_{K_2}u\|_{L^2_{t,x}}^2.
\end{align*}
This concludes the proof of Lemma~\ref{lem:linearW}~(v).

It remains to prove the double estimate of Lemma~\ref{lem:linearW}~(vi). By duality, we can assume $b\ge 0$. Using \eqref{est:ef} together with \eqref{est:ev}-\eqref{est:ev2} and Cauchy-Schwarz inequality, we then find
\begin{align*}
\|u\|_{X^{s,b}}^2&= \int_\mathbb{R}\sum_{k\in\mathbb{Z}}\langle \tau-k^2\rangle^{2b}\langle k\rangle^{2s}|\hat{u}_k(\tau)|^2d\tau \lesssim \int_\mathbb{R}\sum_{k\in\mathbb{Z}}\langle \tau-k^2\rangle^{2b}\langle k\rangle^{2s}\Big|\sum_{n\in\mathbb{Z}}\langle\hat{u},f_n\rangle_{L^2_x}\hat{f_n}(k)\Big|^2d\tau\\
&\lesssim \|u\|_{X^{s,b}_{-\partial_x^2+W}}^2+\int_\mathbb{R}\sum_{k\in\mathbb{Z}}\langle \tau-k^2\rangle^{2b}\langle k\rangle^{2s}\Big(\sum_{|n|\neq|k|}|\langle\hat{u},f_n\rangle_{L^2_x}|\langle |n|-|k|\rangle^{-1-\sigma}\langle |n|+|k|\rangle^{-1}\Big)^2d\tau\\
&\lesssim \|u\|_{X^{s+\beta,b}_{-\partial_x^2+W}}^2\Big\{1+\sup_{\tau\in\mathbb{R}}\sum_{\substack{n,k\in\mathbb{Z}\\|n|\neq |k|}}\frac{\langle \tau-k^2\rangle^{2b}}{\langle\tau-n^2\rangle^{2b}}\frac{\langle k\rangle^{2s}}{\langle n\rangle^{2(s+\beta)}}\langle |n|-|k|\rangle^{-2-2\sigma}\langle |n|+|k|\rangle^{-2}\Big\}.
\end{align*}
To estimate the last sum, we treat separately different contributions. We have
\begin{align*}
\sup_{\tau\in\mathbb{R}}\sum_{\substack{n,k\\|n|\ll |k|}}\frac{\langle \tau-k^2\rangle^{2b}}{\langle\tau-n^2\rangle^{2b}}\frac{\langle k\rangle^{2s}}{\langle n\rangle^{2(s+\beta)}}\langle |n|-|k|\rangle^{-2-2\sigma}\langle |n|+|k|\rangle^{-2}& \lesssim \sum_{\substack{n,k\\|n|\ll |k|}}\langle k\rangle^{4b+2s-4-2\sigma}\langle n\rangle^{-2(s+\beta)}\lesssim 1
\end{align*}
provided that $2b+s<\frac32+\sigma$ and $2b<1+\beta+\sigma$. Next,
\begin{align*}
&\sup_{\tau\in\mathbb{R}}\sum_{\substack{n,k\\|n|\sim |k|\gtrsim |n\pm k|}}\frac{\langle \tau-k^2\rangle^{2b}}{\langle\tau-n^2\rangle^{2b}}\frac{\langle k\rangle^{2s}}{\langle n\rangle^{2(s+\beta)}}\langle |n|-|k|\rangle^{-2-2\sigma}\langle |n|+|k|\rangle^{-2}\\
&\qquad\qquad \lesssim \sum_{\substack{n,k\\|n|\sim |k|\gtrsim|n\pm k|}}\langle k\rangle^{2b-2-2\beta}\langle n\pm k\rangle^{2b-2-2\sigma}\lesssim 1
\end{align*}
provided that $b<\frac12+\beta$ and $2b<1+\beta+\sigma$. Finally,
\begin{align*}
\sup_{\tau\in\mathbb{R}}\sum_{\substack{n,k\\|n|\gg |k|}}\frac{\langle \tau-k^2\rangle^{2b}}{\langle\tau-n^2\rangle^{2b}}\frac{\langle k\rangle^{2s}}{\langle n\rangle^{2(s+\beta)}}\langle |n|-|k|\rangle^{-2-2\sigma}\langle |n|+|k|\rangle^{-2}& \lesssim \sum_{\substack{n,k\\|n|\gg |k|}} \langle n\rangle^{4b-2s-2\beta-4-2\sigma}\langle k\rangle^{2s}\lesssim 1
\end{align*}
provided again that $2b+s<\frac32+\sigma+\beta$ and $2b<1+\beta+\sigma$. This proves Lemma~\ref{lem:linearW}~(vi).

\begin{remark}\rm~~\\
(i) In the case $V=0$, the Strichartz estimate of Lemma~\ref{lem:linear*}~(v) is due to Bourgain, and is an improvement on the corresponding $L^4$-Strichartz estimates proved in \cite{Zyg}:
\begin{equation}\label{StrL4}
\|e^{it\partial_x^2}u_0\|_{L^4_{t,x}(\mathbb{T}\times\mathbb{T})}\lesssim \|u_0\|_{L^2}.
\end{equation}
Indeed, for $u$ space-time periodic, expanding the space-time Fourier series and making a change of variables we can write ${\displaystyle u(t,x)=\sum_{j}e^{itj}e^{it\partial_x^2}U_j(x)}$ with ${U_j(x)=\sum_{k\in\Z}e^{ikx}\widehat{u_k}(j+k^2)}$. Together with \eqref{StrL4} and Cauchy-Schwarz inequality we get
\begin{align*}
\|u\|_{L^4_{t,x}}\le \sum_{j\in\Z}\|e^{-it\partial_x^2}U_j(x)\|_{L^4_{t,x}} \lesssim \sum_{j\in\Z}\|U_j\|_{L^2_x} \lesssim \|\langle j\rangle^{\frac12+}U_j\|_{\ell^2_jL^2_x} = \|u\|_{X^0,\frac12+}.
\end{align*}
Thus Lemma~\ref{lem:linear*}~(v) gains almost $\frac18$ regularity in modulation. On the contrary, \eqref{StrL4} does not hold directly anymore for lower order perturbations of $-\partial_x^2$, since the Schr\"odinger semigroup is no longer time periodic. However, the estimate of Lemma~\ref{lem:linear*}~(v) remains true, as we have seen above.\\
(ii) As a consequence of the equivalence of norms in Lemma~\ref{lem:linear*}~(vi), we also have the $L^6$ Strichartz type estimate \textit{in $X^{s,b}_{-\partial_x^2+V\ast}$ space} for $0<T\le 1$
\begin{align}\label{est:L6}
\|u\|_{L^6_{T,x}}\lesssim \|u\|_{X^{0+,\frac12+}_{-\partial_x^2+V\ast}(T)}
\end{align}
since it holds for the case $V=0$ from the $L^6$ estimate \eqref{strichartz} and the argument in (i) above. Note however that the corresponding $L^6$ Strichartz estimate for linear solutions
\begin{align*}
\|e^{it(-\partial_x^2+V\ast)}u_0\|_{L^6_{T,x}}\lesssim \|u_0\|_{H^{0+}}
\end{align*}
cannot be proved directly with the original proof of Bourgain \cite{Bou93} since it relied crucially on the fact that the symbol of the linear operator is integer valued. In particular, it is not clear if these estimates hold in the case of a multiplicative potential.
\end{remark}

\section{Proof of some technical lemmas}\label{appendixproof}

In this section we give the proofs of Lemmas~\ref{lem:E1E2},~\ref{lem:Ikey}, and~\ref{lem:Iperturb}.

\subsection{Multilinear estimates in case of a convolution potential} 
We start with the proof of Lemmas~\ref{lem:E1E2} and~\ref{lem:Ikey}. These are straightforward adaptations of \cite[Lemma 3.3, Propositions 3.2 \& 3.3]{LWX}, that we detail here for completeness.

\begin{proof}[Proof of Lemma~\ref{lem:E1E2}]
We start with the proof of the equivalence of the modified energy $E_1$ and $E_2$ in the sense of \eqref{eq:modifiedenergy*}, which is the analogue to \cite[Lemma 3.3]{LWX}. The latter relied mainly on the bound on the symbol $|\widetilde{M_6}|\lesssim |\alpha_6|$. Thus, we start by showing
\begin{align}\label{est:M6A6}
|\widetilde{M_6}^V|\lesssim |\alpha_6^V|
\end{align}
where the symbols are defined in \eqref{eq:alpha} and \eqref{eq:M6V}. In particular, $$\widetilde{M_6}^V = \mathbf{1}_\Omega M_6^{V,1}+\mathbf{1}_\Upsilon M_6^{V,2}=\frac{\mathbf{1}_\Omega}6\sum_{j=1}^6(-1)^{j+1}m(k_j)^2(\omega_{(-1)^{j+1}k_j}+\gamma)+\mathbf{1}_\Upsilon \sigma_6 \alpha_6^V,$$ where the sets of frequencies have been defined in \eqref{eq:Ups}--\eqref{eq:O5}. Since $\sigma_6=\frac16\prod_{j=1}^6m(k_j)$, we have directly
\begin{align*}
\big|\mathbf{1}_\Upsilon \sigma_6 \alpha_6^V\big|\lesssim |\alpha_6^V|,
\end{align*}
so that we only need to estimate $\mathbf{1}_\Omega\sum_{j=1}^6(-1)^{j+1}m(k_j)^2(\omega_{(-1)^{j+1}k_j}+\gamma)$. We then follow the proof of Lemma 3.2 in \cite{LWX}. We will frequently make use of the equivalence between $k^2$ and $\omega_{(-1)^{j+1}k_j}+\gamma$. Indeed, recalling that the eigenvalues of $-\partial_x^2+V_\lambda\ast$ are $\omega_k = k^2+(2\pi)^{\frac12}\lambda^{-2}V_{\lambda k}$, it holds for any real-valued $V\in \ell^\infty (\Z;\R)$
\begin{align}\label{eq:ev}
1+|\omega_{\pm k}+\gamma|\le (1+\gamma+\lambda^{-2}\|V\|_{\ell^\infty})(1+k^2)\le (1+\gamma +\lambda^{-2}\|V\|_{\ell^\infty})^2(1+|\omega_{\pm k}+\gamma|).
\end{align}

In the following we thus drop the dependence of the constants in $\gamma,\|V\|_{\ell^\infty}$.
Then \eqref{eq:ev} implies the bound $$|\mathbf{1}_\Omega M_6^{V,1}|\lesssim \max_j |k_j|^2.$$
In particular in $\Omega_1$ \eqref{eq:O1}, we have $|\alpha_6^V|\sim |k_1|^2= \max_j |k_j|^2$, which is enough for \eqref{est:M6A6}.

Next, in $\Omega_2\setminus \Omega_1$, since $|\alpha_6^V|=|\alpha_6|+O(1)$ we have from \cite[(3.44)--(3.48)]{LWX} and the mean value theorem that in case $k_2^*=k_2$ and $k_3^*=k_3$,
\begin{align*}
|\mathbf{1}_{\Omega_2\setminus\Omega_1}M_6^{V,1}|&\lesssim |m(k_1)^2k_1^2-m(k_2)^2k_2^2|+O(1)+\sum_{j=3}^6|\omega_{(-1)^{j+1}k_j}+\gamma|\\
&\lesssim |k_1^2-k_2^2|+|k_3|^2 \sim k_2k_3 \lesssim |\alpha_6| \lesssim |\alpha_6^V|,
\end{align*}
where in the last step we used that $|k_2k_3|\geq 1$ in $\Omega_2\setminus \Omega_1$ if $k_2=k_2^*$ and $k_3=k_3^*$, so that $|\alpha_6^V|=|\alpha_6|+O(1)\lesssim |\alpha_6|\lesssim |\alpha_6^V|$.

Similarly as above, in case $k_2^*=k_2$ and $k_3^*=k_4$, we find $|\alpha_6^V|=|\alpha_6|+O(1)\sim k_2k_4 \gtrsim |\mathbf{1}_{\Omega_2\setminus\Omega_1}M_6^{V,1}|$. In the last case $k_2^*=k_3$ and $k_3^*=k_2$ it holds $|\alpha_6^V|=|\alpha_6|+O(1)\sim k_1^2 \sim \max_j |k_j|^2\gtrsim |\mathbf{1}_\Omega M_6^{V,1}|$.

To treat the contribution $\Omega_3\setminus \Omega_1$, as in \cite{LWX} it holds
\begin{align*}
|\alpha_6^V| = (k_1^2-k_2^2+k_3^3+k_5)^2+o(k_1^2)+O(1) \gtrsim k_3^2+k_5^2 \sim k_1^2\gtrsim \max_j |k_j|^2\gtrsim |\mathbf{1}_\Omega M_6^{V,1}|.
\end{align*}

For the case $\Omega_4\setminus\Omega_1$, \cite[(3.49)-(3.50)]{LWX} are replaced by 
\begin{align*}
|\alpha_6^V|&\ge (k_4^2+k_6^2)-|k_1^2-k_2^2|-|k_3^2+k_5^2|+O(1) = (k_4^2+k_6^2)+o(k_1^2)+O(1)\sim k_1^2\sim \max_j|k_j|^2\\
&\gtrsim |\mathbf{1}_\Omega M_6^{V,1}|,
\end{align*}
and
\begin{align*}
|\alpha_6^V|=(k_1^2-k_2^2-k_4^2-k_6^2)+o(k_1^2)+O(1)\sim k_1^2\sim \max_j|k_j|^2\gtrsim |\mathbf{1}_\Omega M_6^{V,1}|.
\end{align*}

At last, in $\Omega_5$, $$|\alpha_6^V|\gtrsim |\omega_{k_1}-\omega_{-k_2}|= |k_1^2-k_2^2|+O(1).$$ Note that the definition of $\Omega_5$ and $\Upsilon$ prevents $k_1-k_2=0$. Note also that $k_1+k_2\neq 0$ in $\Omega_5$, since otherwise $|\omega_{k_1}-\omega_{-k_2}|=0$ which contradicts ${\displaystyle |\omega_{k_1}-\omega_{-k_2}|\gg |\sum_{j=3}^6(-1)^{j+1}\omega_{(-1)^{j+1}k_j}|}$. This yields $|\alpha_6^V|\sim|k_1^2-k_2^2|\geq 1$, and by the mean value theorem
\begin{align*}
\big|M_6^V\big|&\le |m(k_1)^2(\omega_{k_1}+\gamma)-m(k_2)^2(\omega_{-k_2}+\gamma)|+\Big|\sum_{j=3}^6(-1)^{j+1}(\omega_{(-1)^{j+1}k_j}+\gamma)\Big|\\
&=N^{2(1-s)}\big||k_1|^{2s}+|k_1|^{2s-2}(\gamma+(2\pi)^\frac12(V_{\lambda})_{k_1})-|k_2|^{2s}-|k_2|^{2s-2}(\gamma+(2\pi)^\frac12(V_{\lambda})_{-k_2})\big|\\
&\qquad\qquad+\Big|\sum_{j=3}^6(-1)^{j+1}(\omega_{(-1)^{j+1}k_j}+\gamma)\Big|\\
&\lesssim N^{2(1-s)}|k_1^2-k_2^2||k_2|^{2(s-1)}+O(1)+\Big|\sum_{j=3}^6(-1)^{j+1}\widetilde{\omega}_{k_j}\Big|\\
&\lesssim |k_1^2-k_2^2|+\Big|\sum_{j=3}^6(-1)^{j+1}(\omega_{(-1)^{j+1}k_j}+\gamma)\Big|
\\
&\lesssim |\alpha_6^V|.
\end{align*}
Therefore \eqref{est:M6A6} is proved.

Recalling that $|k_1^*|\sim|k_2^*|\gtrsim N$ in $\Upsilon$ and $m(k_1^*)\langle k_1^*\rangle \sim N^{1-s}|k_1^*|^s$, we can then use \eqref{est:M6A6} with H\"older and Sobolev inequalities and \eqref{eq:HsIscal} to finally bound for $s>\frac13$: 
\begin{align*}
\Big|\Lambda_6\big(\frac{\widetilde{M_6^V}}{\alpha_6^V};u(t)\big)\Big|&=\Big|\int_{\Gamma_6}\frac{\widetilde{M_6^V}}{\alpha_6^V}(k_1,\ldots,k_6)\prod_{j=1}^6\hat{u}(t,k_j)(dk_j)_\lambda\Big|\lesssim \int_{\Upsilon}\prod_{j=1}^6|\hat{u}(t,k_j)|(dk_j)_\lambda\\
&\lesssim \|\mathcal{F}^{-1}\{\mathbf{1}_{|k_1^*|\gg N}|\hat{u}(t,k_1^*)|\}|\|_{L^6_x}\|\mathcal{F}^{-1}\{\mathbf{1}_{|k_2^*|\gg N}|\hat{u}(t,k_2^*)|\}|\|_{L^6_x}\|\mathcal{F}^{-1}\{|\hat{u}(t,k)|\}|\|_{L^6_x}^4 \\
&\lesssim \|\langle k_1^*\rangle^s\hat{u}(t,k_1^*)|\|_{\ell^2_{|k_1^*|\gg N}}\|\langle k_2^*\rangle^s\hat{u}(t,k_2^*)|\|_{\ell^2_{|k_2^*|\gg N}} \|u(t)\|_{H^s}^4\\
&\lesssim N^{2(s-1)}\|I_Nu(t)\|_{H^1}^6,
\end{align*}
where $\mathcal{F}^{-1}\{a_k\}(x)=(2\pi)^{-\frac12}\int a_ke^{ikx}(dk)_\lambda$ is the inverse spatial Fourier transform for $\lambda$-periodic functions. This concludes the proof of Lemma~\ref{lem:E1E2}.
\end{proof}

\begin{proof}[Proof of Lemma~\ref{lem:Ikey}]
Let us start with 
the multilinear estimate (i) on $\Lambda_6(\overline{M_6^V})$. First, we decompose dyadically
\begin{align*}
\Lambda_6\big(\overline{M_6^V}\big)=\sum_{N_1,...,N_6}\int_{\Upsilon}\frac{\mathbf{1}_{\Upsilon\setminus\Omega}}6\sum_{j=1}^6(-1)^{j+1}m(k_j)^2(\omega_{(-1)^{j+1}k_j}+\gamma)\prod_{j=1}^6\widehat{P_{N_j}u_{k_j}}(t)(dk_j)_\lambda,
\end{align*}
where the $N_j$ run over dyadic integers, and $P_{N_j}$ is the projector on frequencies $N_j\le|k_j|<2N_j$. Writing similarly as for the $k_j$'s $N_1^*\ge N_2^*\ge \ldots\ge N_6^*$ the decreasing rearrangement of the $N_j$'s, by definition of $\Upsilon$ it holds $N_1^*\sim N_2^*\gtrsim N$. Moreover, we can replace $u\in X^{s,\frac12+}_\lambda(\delta)$ above by an extension $v\in X^{s,\frac12+}_\lambda$ satisfying $\|v\|_{X^{s,\frac12+}}\le 2\|u\|_{X^{s,\frac12+}(\delta)}$. To simplify notations, we still write $u$ in place of $v$.

Then, to deal with the sharp time truncation restricting the time integral to $[0;\delta]$, since multiplication by an indicator function is not bounded on $X^{s,b}$ when $b>\frac12$, we proceed as in \cite{CKSTT1} and decompose 
\begin{align*}
\mathbf{1}_{[0;\delta]}=f(t)+g(t)=\mathbf{1}_{[0;\delta]}\ast (N_1^*)^{100}\chi((N_1^*)^{100}\cdot) + \big[\mathbf{1}_{[0;\delta]}-\mathbf{1}_{[0;\delta]}\ast(N_1^*)^{100}\chi((N_1^*)^{-100}\cdot)\big]
\end{align*}
for some smooth cut-off $\chi$ satisfying $\chi\equiv 1$ on $[-1;1]$. To estimate the contribution of the second term to the integral, since we can bound crudely $|\overline{M_6^V}|\lesssim (N_1^*)^2$, we have by H\"older and Sobolev inequalities
\begin{align*}
\Big|\int_{\mathbb{R}}g(t)\int_{\Upsilon}\overline{M_6^V}\prod_{j=1}^6\widehat{P_{N_j}u_{k_j}}(t)(dk_j)_\lambda dt\Big|&\lesssim (N_1^*)^2\|g\|_{L^2}\|P_{N_1}u\|_{L^2_tL^6_x}\prod_{j=2}^6\|P_{N_j}u\|_{L^\infty_tL^6_x}\\
&\lesssim (N_1^*)^2\|g\|_{L^2}\prod_{j=1}^6\|P_{N_j}u\|_{X^{s,\frac12+}}
\end{align*}
for $s>\frac13$. This is enough to sum on $N_j$'s since by the mean value theorem
\begin{align*}
\|g\|_{L^2}^2&=\int_{\mathbb{R}}\Big|\frac{1-e^{-i\delta\tau}}{i\tau}\big(1-\widehat{\chi}((N_1^*)^{-100}\tau)\big)\Big|^2d\tau\\
&\lesssim\int_{|\tau|\lesssim (N_1^*)^{100}}(N_1^*)^{-200}\Big(\int_0^1\big|\widehat{\chi}'\big(\theta(N_1^*)^{-100}\tau\big)\big|d\theta\Big)^2d\tau+\int_{|\tau|\gg (N_1^*)^{100}}|\tau|^{-2}d\tau\\
&\lesssim (N_1^*)^{-200}\int_{|\tau|\lesssim (N_1^*)^{100}}d\tau+(N_1^*)^{-100}\lesssim (N_1^*)^{-100}.
\end{align*}

To estimate the contribution of the first term, since $H^{\frac12+}(\mathbb{R})$ is an algebra, we have
\begin{align*}
\|fu\|_{X^{s,\frac12+}}=\big\|f(t)\|e^{it(-\partial_x^2+V_\lambda\ast)}u(t)\|_{H^s_x}\big\|_{H^{\frac12+}_t} \lesssim \|f\|_{H^{\frac12+}}\|u\|_{X^{s,\frac12+}},
\end{align*}
and, since $\widehat{\chi}$ is a Schwartz function,
\begin{align*}
\|f\|_{H^{\frac12+}}^2&=\int_{\mathbb{R}}\langle\tau\rangle^{1+}\Big|\frac{1-e^{i\delta\tau}}{i\tau}\widehat{\chi}\big((N_1^*)^{-100}\tau\big)\Big|^2d\tau\\
&\lesssim \int_{|\tau|\lesssim (N_1^*)^{100}}\langle \tau\rangle^{-1+}d\tau +\int_{|\tau|\gtrsim (N_1^*)^{100}}\langle \tau\rangle^{-1+}\langle (N_1^*)^{-100}\tau\rangle^{-10}d\tau\\
&\lesssim (N_1^*)^{0+}.
\end{align*}

Then, we define $$\widehat{U_{N_1}}(t,k_1)=g(t)\mathbf{1}_{[N_1;2N_1)}(k_1)m(k_1)\langle k_1\rangle |\widehat{u}(t,k_1)|$$ and $$\widehat{U_{N_j}}(\tau_j,k_j)=\mathbf{1}_{[N_j;2N_j)}(k_j)m(k_j)\langle k_j\rangle |\widehat{u}(\tau_j,k_j)|,$$ $j=2,\ldots,6$, and seek to prove
\begin{equation}\label{est:multiN}
\begin{split}
\int_{\Gamma_6}\int_{\Upsilon\setminus\Omega}\frac{\big|\sum_{j=1}^6(-1)^{j+1}m(k_j)^2(\omega_{(-1)^{j+1}k_j}+\gamma)\big|}{\prod_{j=1}^6m(k_j)\langle k_j\rangle}&\prod_{j=1}^6\widehat{U_{N_j}}(\tau_j,k_j)d\tau_j(dk_j)_\lambda\\ &\lesssim N^{-3+}(N_1^*)^{0-}\prod_{j=1}^6\|U_{N_j}\|_{X^{0,\frac12+}}.
\end{split}
\end{equation}

In order to prove estimate \eqref{est:multiN}, we will need the refined bilinear estimates for $\lambda$-periodic functions that we alluded to before.

\begin{lemma}\label{lem:bil}
The bilinear map
\begin{align*}
J_N^- : (u,v) \mapsto \int e^{ikx}\int_{k_1+k_2=k}\mathbf{1}_{|k_1-k_2|\gtrsim N}\widehat{u}(k_1)\widehat{v}(k_2)(dk_1)_\lambda(dk)_\lambda
\end{align*}
is bounded from $(X^{0,\frac12+}_\lambda)^2$ to $L^2(\mathbb{R}\times \mathbb{T}_\lambda)$, with 
\begin{align*}
\|J_N^-(u,v)\|_{L^2_{t,x}}\lesssim \big(N^{-1}+\lambda^{-1}\big)^{\frac12}\|u\|_{X^{0,\frac12+}}\|v\|_{X^{0,\frac12+}}.
\end{align*}
The same property holds for
\begin{align*}
J_N^+:(u,v)\mapsto\int e^{ikx}\int_{k_1+k_2=k}\mathbf{1}_{|k_1+k_2|\gtrsim N}\widehat{u}(k_1)\widehat{\overline{v}}(k_2)(dk_1)_\lambda(dk)_\lambda.
\end{align*}
\end{lemma}
We postpone the proof of Lemma~\ref{lem:bil} to the end of this subsection and continue with the proof of Lemma~\ref{lem:Ikey}. With the same arguments as for the proof of \eqref{est:M6A6}, it is straightforward to check that \cite[Lemma 3.4]{LWX} remains true with $\overline{M_6^V}$ in place of $\overline{M_6}$. Then, since the proof of \cite[Proposition 3.2]{LWX} only relies on the bound on the symbol, Lemma~\ref{lem:bil}, Sobolev inequality and \eqref{est:L6}, the exact same arguments prove \eqref{est:multiN} and thus Lemma~\ref{lem:Ikey}~(i).

The same argument holds for the proof of \cite[Proposition 3.3]{LWX} which relied again on Lemma~\ref{lem:bil}, Sobolev inequality, \eqref{est:L6}, and \eqref{est:M6A6}, and thus remains true when replacing $\overline{M_{10}}$ by $\overline{M_{10}^V}$. This concludes the proof of Lemma~\ref{lem:Ikey}~(ii).

\end{proof}

Finally, we give the proof of Lemma~\ref{lem:bil}.
\begin{proof}[Proof of Lemma~\ref{lem:bil}]
We proceed as in the proof of Lemma~\ref{lem:linear*}~(v) given in Appendix~\ref{appendix} above and decompose dyadically
\begin{align*}
&\Big\|\int_{k_1+k_2=k}\mathbf{1}_{|k_1\pm k_2|\gtrsim N}\widehat{u}(k_1)\widehat{v}(k_2)(dk_1)_\lambda\Big\|_{L^2_{t,(dk)_\lambda}}\\
&\qquad\qquad\qquad\qquad\le \sum_{K_1,K_2}\Big\|\int_{k_1+k_2=k}\mathbf{1}_{|k_1\pm k_2|\gtrsim N}\widehat{P_{K_1}u}(k_1)\widehat{P_{K_2}v}(k_2)(dk_1)_\lambda\Big\|_{L^2_{t,(dk)_\lambda}},
\end{align*}
where $P_{K_j}$ is the smooth projector on the modulation $|\tau_j-\omega_{k_j}|\sim K_j$. Then proceeding as in the proof of Lemma~\ref{lem:linear*}~(v),
\begin{align*}
&\int_{\mathbb{R}}\int\Big|\int_{\tau_1+\tau_2=\tau}\int_{k_1+k_2=k}\mathbf{1}_{|k_1- k_2|\gtrsim N}\widehat{P_{K_1}u}(\tau_1,k_1)\widehat{P_{K_2}v}(\tau_2,k_2)d\tau_1(dk_1)_\lambda\Big|^2d\tau(dk)_\lambda\\
 &\qquad\qquad\lesssim \|P_{K_1}u\|_{L^2_{t,x}}^2\|P_{K_2}v\|_{L^2_{t,x}}^2\sup_{\tau,k} \frac1\lambda|A_{N,\lambda,\tau,k}|,
\end{align*}
where
\begin{align*}
A_{N,\lambda,\tau,k}=\big\{(\tau_1,k_1)\in\mathbb{R}\times\frac1\lambda\mathbb{Z},~~|k_1-(k-k_1)|\gtrsim N,~~|\tau_1-\omega_{k_1}|\sim K_1,~~|\tau-\tau_1-\omega_{k-k_1}|\sim K_2\big\}.
\end{align*}

In case $\max(K_1,K_2)\gtrsim N^2$, we proceed as in the proof of Lemma~\ref{lem:linear*}~(v) to get
\begin{align*}
\big|A_{N,\lambda,\tau,k}\big|&\lesssim \min(K_1,K_2)\#\big\{k_1\in\frac1\lambda\mathbb{Z},~~|k_1-(k-k_1)|\gtrsim N,~~|\tau-\omega_{k_1}-\omega_{k-k_1}|\lesssim \max(K_1,K_2)\big\}\\
&\lesssim \min(K_1,K_2)\#\big\{k_1\in\frac1\lambda\mathbb{Z},~~2k_1^2+2kk_1=\tau-k^2=O(\max(K_1,K_2))\big\}\\
&\lesssim \min(K_1,K_2)\lambda\max(K_1,K_2)^\frac12\lesssim \frac{\lambda}{N} K_1K_2.
\end{align*}
In case $\max(K_1,K_2)\ll N^2$, noting that
\begin{align}\label{est:lb}
(k_1-(k-k_1))^2=2\tau -k^2-2(\tau_1-\omega_{k_1})-2((\tau-\tau_1)-\omega_{k-k_1})+O_V(1),
\end{align}
we get
\begin{align*}
&\big|A_{N,\lambda,\tau,k}\big|\\
&\lesssim \min(K_1,K_2)\#\big\{k_1\in\frac1\lambda\mathbb{Z},~~|k_1-(k-k_1)|\gtrsim N,~~|(k_1-(k-k_1))^2-4\tau+k^2|\lesssim \max(K_1,K_2)\big\}\\
&\lesssim \min(K_1,K_2)\#\big\{k_1\in\frac1\lambda\mathbb{Z},~~|k_1-(k-k_1)|\gtrsim N,~~k_1=\frac{k}2\pm\frac12\sqrt{2\tau-k^2+O(\max(K_1,K_2))}\big\}.
\end{align*}
But for $k_1,\widetilde{k_1}$ in the above set, we have from \eqref{est:lb} and the lower bound on $|k_1-(k-k_1)|$ that
\begin{align*}
|k_1-\widetilde{k_1}|&\lesssim\big|\sqrt{2\tau-k^2+O(\max(K_1,K_2))}-\sqrt{2\tau-k^2+O(\max(K_1,K_2))}\big|\\
&\lesssim\frac{O(\max(K_1,K_2)}{\sqrt{2\tau-k^2+O(\max(K_1,K_2))}}\lesssim \frac{\max(K_1,K_2)}{N}.
\end{align*}
This finally gives
\begin{align*}
\big|A_{N,\lambda,\tau,k}\big|&\lesssim \min(K_1,K_2)\big(1+\lambda\frac{\max(K_1,K_2)}{N}\big).
\end{align*}

All in all, this yields
\begin{align*}
&\int_{\mathbb{R}}\int\Big|\int_{\tau_1+\tau_2=\tau}\int_{k_1+k_2=k}\mathbf{1}_{|k_1- k_2|\gtrsim N}\widehat{P_{K_1}u}(\tau_1,k_1)\widehat{P_{K_2}v}(\tau_2,k_2)d\tau_1(dk_1)_\lambda\Big|^2d\tau(dk)_\lambda\\
 &\qquad\qquad\lesssim \frac1\lambda \big(1+\lambda\frac{1}{N}\big)K_1K_2\|P_{K_1}u\|_{L^2_{t,x}}^2\|P_{K_2}v\|_{L^2_{t,x}}^2.
\end{align*}
This is enough to prove Lemma~\ref{lem:bil} after summing on $K_1,K_2$.

When $v$ is replaced by $\overline{v}$, we have similarly
\begin{align*}
&\int_{\mathbb{R}}\int\Big|\int_{\tau_1+\tau_2=\tau}\int_{k_1+k_2=k}\mathbf{1}_{|k_1+k_2|\gtrsim N}\widehat{P_{K_1}u}(\tau_1,k_1)\widehat{\overline{P_{K_2}v}}(\tau_2,k_2)d\tau_1(dk_1)_\lambda\Big|^2d\tau(dk)_\lambda\\
&=\int_{\mathbb{R}}\int\Big|\int_{\tau_1-\tau_2=\tau}\int_{k_1-k_2=k}\mathbf{1}_{|k|\gtrsim N}\widehat{P_{K_1}u}(\tau_1,k_1)\overline{\widehat{P_{K_2}v}}(\tau_2,k_2)d\tau_1(dk_1)_\lambda\Big|^2d\tau(dk)_\lambda\\
&\lesssim \frac1\lambda \sup_{\tau,k}|\widetilde{A}_{N,\lambda,\tau,k}|\|P_{K_1}u\|_{L^2_{t,x}}^2\|P_{K_2}v\|_{L^2_{t,x}}^2,
\end{align*}
with
\begin{align*}
\widetilde{A}_{N,\lambda,\tau,k}=\big\{(\tau_1,k_1)\in\mathbb{R}\times\frac1\lambda\mathbb{Z},~~|k|\gtrsim N,~~|\tau_1-\omega_{k_1}|\sim K_1,~~|\tau_1-\tau-\omega_{k_1-k}|\sim K_2\big\}.
\end{align*}
Proceeding as above,
\begin{align*}
|\widetilde{A}_{N,\lambda,\tau,k}|&\lesssim \min(K_1,K_2)\#\big\{k_1\in\frac1\lambda\mathbb{Z},~~|k|\gtrsim N,~~|\tau-\omega_{k_1}+\omega_{k_1-k}|\lesssim \max(K_1,K_2)\big\}\\
&\lesssim \min(K_1,K_2)\#\big\{k_1\in\frac1\lambda\mathbb{Z},~~|k|\gtrsim N,~~k_1^2-(k_1-k)^2=\tau+O(\max(K_1,K_2))\big\}\\
&\lesssim \min(K_1,K_2)\#\big\{k_1\in\frac1\lambda\mathbb{Z},~~|k|\gtrsim N,~~k_1=\frac{\tau+k^2+O(\max(K_1,K_2))}{2k}\big\}\\
&\lesssim  \min(K_1,K_2)\big(1+\lambda \frac{\max(K_1,K_2)}{N}\big).
\end{align*}
This finishes the proof of Lemma~\ref{lem:bil}.
\end{proof}

\subsection{Estimates in the case of a multiplicative potential}
We finish this section by giving the proof of Lemma~\ref{lem:Iperturb}.

First, from similar computations as for \eqref{commutator} we have for any $t\in\mathbb{R}$
\begin{align*}
\big|\langle \partial_x I_Nu,\partial_x[I_N,W_\lambda]u\rangle_{L^2_x}\big|& \le \|I_Nu\|_{H^1}\|\partial_x [I_N,W_\lambda]u\|_{L^2}\\
&\lesssim \big(N^{-1}\|W_\lambda\|_{H^4}\|u\|_{L^2}+N^{-1}\|\partial_x W_\lambda\|_{L^\infty}\|I_Nu\|_{H^1}\big)\|I_Nu\|_{H^1}\\
&\lesssim N^{-1}\lambda^{-2}\|W\|_{H^4}\|I_Nu\|_{H^1}^2.
\end{align*}
This proves (i). 

Similarly,
\begin{align*}
\big|\langle W_\lambda I_Nu,[I_N,W_\lambda]u\rangle_{L^2}\big|&\lesssim \|W_\lambda\|_{L^\infty}\|I_Nu\|_{L^2}\|[I_N,W_\lambda]u\|_{L^2}\\
& \lesssim N^{-1}\|W_\lambda\|_{L^\infty}\|\partial_x W_\lambda\|_{L^\infty}\|I_Nu\|_{L^2}\|u\|_{L^2}\\
&\lesssim N^{-1}\|W_\lambda\|_{H^4}^2\|u\|_{L^2}^2.
\end{align*}
This proves (ii).

With Sobolev embedding and similar computations, we also have 
\begin{align*}
\big|\langle [I_N,W_\lambda]u,I_N(|u|^4u)\rangle_{L^2}\big|&\lesssim N^{-1}\|\partial_x W_\lambda\|_{L^\infty}\|u\|_{L^2}\||u|^4u\|_{L^2}\lesssim N^{-1}\|W_\lambda\|_{H^4}\|u\|_{H^s}^6\\
&\lesssim N^{-1}\|W_\lambda\|_{H^4}\|I_Nu\|_{H^1}^6,
\end{align*}
since $s>\frac25$. Moreover,
\begin{align*}
\langle I_N(W_\lambda u),I_N(|u|^4u)-|I_Nu|^4I_Nu\rangle_{L^2}&=\int_{\Gamma_6}m_6(k_1,\ldots,k_6)(W_\lambda u)_{k_1}(dk_1)_\lambda\prod_{j=2}^6U_{k_j}(dk_j)_\lambda,
\end{align*}
where the symbol is given by
\begin{align*}
m_6=\frac{m(k_1)\big(m(k_2+\cdots+k_6)-\prod_{j=2}^6m(k_j)}{\prod_{j=2}^6m(k_j)\langle k_j\rangle}
\end{align*}
and $U_{k_j}(t)=m(k_j)\langle k_j\rangle |\hat{u}_{k_j}(t)|$. In particular, if $|k_{j^*}|=\max(|k_j|,~j\ge 2)$, then $m_6$ vanishes unless $|k_{j^*}|\gtrsim N$, and since 
\begin{align*}
m(k)\langle k\rangle\sim\begin{cases} N^{1-s}|k|^s,~~|k|\gtrsim N,\\
\langle k\rangle,~~|k|\ll N,
\end{cases}
\end{align*} 
we have the rough bound
\begin{align*}
|m_6|\lesssim N^{-1}\prod_{j\neq j^*}\langle k_j\rangle^{-s}.
\end{align*}
Together with H\"older and Sobolev inequalities, we get
\begin{align*}
\big|\langle I_N(W_\lambda u),I_N(|u|^4u)-|I_Nu|^4I_Nu\rangle_{L^2}\big|&\lesssim N^{-1}\|W_\lambda\|_{L^\infty}\|u\|_{L^{10}}\|\langle \partial_x\rangle^{-s}U\|_{L^{10}}^4\|U\|_{L^2}\\
&\lesssim N^{-1}\lambda^{-2}\|W\|_{H^4}\|u\|_{H^s}\|U\|_{L^2}^5 \\
&\lesssim N^{-1}\lambda^{-2}\|W\|_{H^4}\|I_Nu\|_{H^1}^6.
\end{align*}
This proves (iii).

Finally, it remains to estimate
\begin{align*}
&\Lambda_7(M_7;u,\ldots,u,W_\lambda)\\
&=\int_{\Gamma_7} \frac{\sum_{j=1}^6(-1)^j\frac{\widetilde{M_6}}{\alpha_6}(k_1,\ldots,k_{j-1},k_j+k_7,k_{j+1},\ldots,k_6)}{\prod_{j=1}^6m(k_j)\langle k_j\rangle}(W_\lambda)_{k_7}(dk_7)_\lambda\prod_{j=1}^6U_{k_j}(dk_j)_\lambda,
\end{align*}
with $U_{k_j}=m(k_j)\langle k_j\rangle |\hat{u_{k_j}}|$. Recall from \cite[Lemma 3.2]{LWX} that $|\widetilde{M_6}|\lesssim |\alpha_6|$, and thus the symbol above is bounded by
\begin{align*}
\big|\sum_{j=1}^6(-1)^j\frac{\widetilde{M_6}}{\alpha_6}(k_1,\ldots,k_{j-1},k_j+k_7,k_{j+1},\ldots,k_6)\big|\lesssim 1,
\end{align*}
thus proceeding as in \cite[Lemma 3.3]{LWX} or equivalently as in the proof of Lemma~\ref{lem:E1E2}, we get
\begin{align*}
\big|\Lambda_7(M_7;u,\ldots,u,W_\lambda)(t)\big|&\lesssim N^{2(s-1)}\lambda^{-2}\|W\|_{H^4}\|I_Nu\|_{H^1}^6.
\end{align*}
This shows (iv). This concludes the proof of Lemma~\ref{lem:Iperturb}.

\end{document}